\pdfoutput=1\relax
\documentclass[reqno]{amsart}

\usepackage{verbatim}
\usepackage[textsize=scriptsize]{todonotes}
\usepackage{tikz-cd}
\usepackage{etoolbox}
\usepackage{etex}
\usepackage{hyperref}
\usepackage{cleveref}
\usepackage[T1]{fontenc}

\usepackage{chemarr}
\usepackage{amssymb}
\usepackage{amsmath}
\usepackage{comment}
\usepackage{spectralsequences}
\usepackage{cleveref}
\usepackage{mathtools}
\usepackage{rotating}
\usepackage{wrapfig}
\usepackage{outlines}
\usepackage{graphicx}
\usepackage{scalerel}

\usepackage{devanagari}
\usepackage{bbm}
\usepackage{multicol}

\usepackage{tipa}

\theoremstyle{definition}
\newtheorem{nul}{}[section]
\newtheorem{dfn}[nul]{Definition}

\newtheorem{rmk}[nul]{Remark}

\newtheorem{cnstr}[nul]{Construction}
\newtheorem{cnv}[nul]{Convention}
\newtheorem{ntn}[nul]{Notation}
\newtheorem{exm}[nul]{Example}

\newtheorem{rec}[nul]{Recollection}

\newtheorem{qst}[nul]{Question}

\newtheorem*{dfn*}{Definition}
\newtheorem*{axm*}{Axiom}
\newtheorem*{ntn*}{Notation}
\newtheorem*{exm*}{Example}
\newtheorem*{exr*}{Exercise}
\newtheorem*{int*}{Intuition}
\newtheorem*{qst*}{Question}
\newtheorem*{rmk*}{Remark}

\theoremstyle{plain}

\newtheorem{thm}[nul]{Theorem}
\newtheorem{prop}[nul]{Proposition}
\newtheorem{summary}[nul]{Summary}
\newtheorem{lem}[nul]{Lemma}

\newtheorem{cor}[nul]{Corollary}

\newtheorem*{thm*}{Theorem}
\newtheorem*{prop*}{Proposition}
\newtheorem*{cor*}{Corollary}
\newtheorem*{lem*}{Lemma}
\newtheorem*{cnj*}{Conjecture}

\let\oldwidetilde\widetilde
\protected\def\widetilde{\oldwidetilde}

\DeclareMathOperator*{\colim}{\mathrm{colim}}
\DeclareMathOperator{\Map}{\mathrm{Map}}

\DeclareMathOperator{\Hom}{\mathrm{Hom}} 
\DeclareMathOperator{\End}{\mathrm{End}}

\DeclareMathOperator{\Ss}{\mathbb{S}}

\DeclareMathOperator{\F}{\mathbb{F}}

\DeclareMathOperator{\A}{\mathcal{A}}
\DeclareMathOperator{\D}{\mathcal{D}}

\DeclareMathOperator{\E}{\mathbb{E}}
\DeclareMathOperator{\sE}{\mathrm{E}}

\DeclareMathOperator{\MO}{\mathrm{MO}}

\DeclareMathOperator{\MU}{\mathrm{MU}}

\DeclareMathOperator{\bS}{\mathbb{S}}

\DeclareMathOperator{\Spec}{\mathrm{Spec}}

\DeclareMathOperator{\Pic}{\mathrm{Pic}}

\DeclareMathOperator{\Ext}{\mathrm{Ext}}

\DeclareMathOperator{\Gr}{\mathbf{Gr}}
\DeclareMathOperator{\Fil}{\mathbf{Fil}}
\DeclareMathOperator{\Fun}{\mathrm{Fun}}

\DeclareMathOperator{\Sp}{\mathrm{Sp}}

\newcommand{\BP}{\mathrm{BP}}
\newcommand{\Syn}{\mathrm{Syn}}
\newcommand{\HZ}{\mathrm{H}\mathbb{Z}}

\newcommand{\Tot}{\mathrm{Tot}}

\newcommand{\tmf}{\mathrm{tmf}}

\newcommand{\Z}{\mathbb{Z}}

\newcommand{\R}{\mathbb{R}}
\newcommand{\G}{\mathbb{G}}
\newcommand{\C}{\mathbb{C}}

\newcommand{\ta}{\text{\footnotesize {\dn t}}}
\newcommand{\SHR}{\mathcal{SH}(\mathbb{R})}
\newcommand{\SHC}{\mathcal{SH}(\mathbb{C})}
\newcommand{\SH}{\mathcal{SH}}
\newcommand{\SHgc}{\mathcal{SH}(\mathbb{R})^{\scaleto{\mathrm{\bf \hspace{-1.5pt}A\hspace{-1.8pt}T}}{3.75pt}}_{i2}}

\newcommand{\Ft}{\mathbb{F}_2}
\newcommand{\Zt}{\mathbb{Z}_2}
\newcommand{\HFt}{\mathrm{M}\mathbb{F}_2}
\newcommand{\MFp}{\mathrm{M}\mathbb{F}_p}
\newcommand{\uFt}{\underline{\mathbb{F}}_2}
\newcommand{\HZt}{\mathrm{M}\mathbb{Z}_{2}}
\newcommand{\uZt}{\underline{\mathbb{Z}}_{2}}
\newcommand{\uZ}{\underline{\mathbb{Z}}}
\newcommand{\MGL}{\mathrm{MGL}}
\newcommand{\MUR}{\mathrm{MU}_{\mathbb{R}}}
\newcommand{\Mod}{\mathrm{Mod}}

\newcommand{\piu}{\underline{\pi}}
\newcommand{\En}{\mathbb{E}_n}
\newcommand{\CC}{\mathcal{C}}
\newcommand{\CCdef}{\mathcal{C}_{\mathrm{def}}}
\newcommand{\Comod}{\mathrm{Comod}}
\newcommand{\id}{\mathrm{id}}
\newcommand{\Ab}{\mathrm{Ab}}
\newcommand{\Abh}{\mathrm{Ab}^{\heartsuit}}

\newcommand{\uAb}{\underline{\mathrm{Ab}}}
\newcommand{\uAbh}{\underline{\mathrm{Ab}}^{\heartsuit}}

\def\Groth{\mathrm{Groth}}
\def\MUR{\mathrm{MU}_{\mathbb{R}}}

\def\even{\mathrm{even}}
\def\twist{\mathrm{tw}}
\def\Mfg{\mathcal{M}_{\mathrm{fg}}}

\def\Ind{\mathrm{Ind}}
\def\QCoh{\mathrm{QCoh}}
\def\IndCoh{\mathrm{IndCoh}}

\def\SHCa{\mathcal{SH}(\mathbb{C})^{\scaleto{\mathrm{\bf \hspace{-1.5pt}A}}{3.75pt}}}
\def\SHCat{\mathcal{SH}(\mathbb{C})^{\scaleto{\mathrm{\bf \hspace{-1.5pt}A\hspace{-1.8pt}T}}{3.75pt}}_{ip}}

\def\SHk{\mathcal{SH}(k)}
\def\SHka{\mathcal{SH}(k)^{\scaleto{\mathrm{\bf \hspace{-1.5pt}A}}{3.75pt}}}
\def\SHkt{\mathcal{SH}(k)^{\scaleto{\mathrm{\bf \hspace{-1.5pt}T}}{3.75pt}}}
\def\SHkat{\mathcal{SH}(k)^{\scaleto{\mathrm{\bf \hspace{-1.5pt}A\hspace{-1.8pt}T}}{3.75pt}}}

\def\SHRa{\mathcal{SH}(\mathbb{R})^{\scaleto{\mathrm{\bf \hspace{-1.5pt}A}}{3.75pt}}}
\def\SHRt{\mathcal{SH}(\mathbb{R})^{\scaleto{\mathrm{\bf \hspace{-1.5pt}T}}{3.75pt}}}
\def\SHRat{\mathcal{SH}(\mathbb{R})^{\scaleto{\mathrm{\bf \hspace{-1.5pt}A\hspace{-1.8pt}T}}{3.75pt}}_{i2}}
\def\SHRatch{\mathcal{SH}(\mathbb{R})^{\scaleto{\mathrm{\bf \hspace{-1.5pt}A\hspace{-1.8pt}T}}{3.75pt}, \mathrm{C}-\heartsuit}_{i2}}

\def\slice{\mathrm{slice}}
\def\cb{\mathrm{cb}}
\def\diag{\mathrm{Diag}}
\def\Q{\mathbb{Q}}
\def\RP{\mathbb{RP}}

\def\DM{\mathrm{DM}}

\def\nur{\nu_{\mathbb{R}}}
\def\TMF{\mathrm{TMF}}
\def\uHom{\underline{\mathrm{Hom}}}
\def\th{\otimes-\heartsuit}
\def\ch{\mathrm{C}-\heartsuit}
\def\cn{\mathrm{C}}
\def\at{\scaleto{\mathrm{\bf \hspace{-1.5pt}A\hspace{-1.8pt}T}}{3.75pt}}
\def\att{\scaleto{\mathrm{\bf A\hspace{-1.8pt}T}}{3.75pt}}
\def\gceff{\scaleto{\mathrm{\bf A\hspace{-1.8pt}T}}{3.75pt}\mathrm{-eff}}
\def\eff{\mathrm{eff}}
\def\Sm{\textbf{Sm}}
\def\cell{\mathrm{cell}}
\def\ko{\mathrm{ko}}
\def\ku{\mathrm{ku}}

\def\triv{\mathrm{triv}}
\def\kq{\mathrm{kq}}
\def\kgl{\mathrm{kgl}}
\def\CAlg{\mathrm{CAlg}}
\def\o{\mathbbm{1}}
\def\Re{\mathrm{Re}}

\def\b{\mathrm{Be}}

\def\P{\mathbb{P}}
\def\AA{\mathbb{A}}

\DeclarePairedDelimiter\abs{\lvert}{\rvert}%
\makeatletter
\let\oldabs\abs
\def\abs{\@ifstar{\oldabs}{\oldabs*}}

\usepackage{tikz}
\usetikzlibrary{matrix,arrows,decorations}
\usepackage{tikz-cd}

\usepackage{adjustbox}

\let\oldtocsection=\tocsection
 
\let\oldtocsubsection=\tocsubsection
 
\let\oldtocsubsubsection=\tocsubsubsection
 
\renewcommand{\tocsection}[2]{\hspace{0em}\oldtocsection{#1}{#2}}
\renewcommand{\tocsubsection}[2]{\hspace{1em}\oldtocsubsection{#1}{#2}}
\renewcommand{\tocsubsubsection}[2]{\hspace{2em}\oldtocsubsubsection{#1}{#2}}

\usepackage{etoolbox}
\newtoggle{draft}
\togglefalse{draft}

\iftoggle{draft} {
\usepackage[margin=1.5in]{geometry}
\newcommand{\NB}[1]{\todo[color=gray!40]{#1}}
\newcommand{\TODO}[1]{\todo[color=red]{#1}}
}{ 
\usepackage[margin=1.5in]{geometry}
\newcommand{\NB}[1]{}
\newcommand{\TODO}[1]{}
\renewcommand{\todo}[1]{}
\renewcommand{\todo}[1]{}
}
\title{Galois reconstruction of Artin--Tate $\R$-motivic spectra}

\author{Robert Burklund}
\address{Department of Mathematics, MIT, Cambridge, MA, USA}
\email{burklund@mit.edu}

\author{Jeremy Hahn}
\address{Department of Mathematics, MIT, Cambridge, MA, USA}
\email{jhahn01@mit.edu}

\author{Andrew Senger}
\address{Department of Mathematics, Harvard University, Cambridge, MA, USA}
\email{senger@math.harvard.edu}

\begin{document}
\begin{abstract}
We explain how to reconstruct the category of Artin--Tate $\R$-motivic spectra as a deformation of the purely topological $C_2$-equivariant stable category.
The special fiber of this deformation is algebraic,
and equivalent to an appropriate category of $C_2$-equivariant sheaves on the moduli stack of formal groups.
As such, our results directly generalize the cofiber of $\tau$ philosophy that has revolutionized classical stable homotopy theory.

A key observation is that the Artin--Tate subcategory of $\R$-motivic spectra is easier to understand than the previously studied cellular subcategory.
In particular, the Artin--Tate category contains a variant of the $\tau$ map, which is a feature conspicuously absent from the cellular category.

\end{abstract}
\maketitle

\setcounter{tocdepth}{1}
\tableofcontents
\vbadness 5000


\section{Introduction} \label{sec:intro}

A striking surprise of modern computational stable homotopy theory is that the category $\SH(\C)$ of $\C$-motivic spectra, as introduced by Morel and Voevodsky \cite{MorelVoevodsky},
has found use not only in algebraic geometry, but also in the computation of stable homotopy groups of spheres, a question of purely topological origin.  In particular, the $p$-completed bi-graded homotopy groups of the unit in $\SH(\C)$ record---in a precise sense---the Adams--Novikov spectral sequence for the sphere spectrum, including all differentials and extensions.  The ultimate expression of this connection is the discovery that the \emph{cellular} subcategory of $p$-complete $\C$-motivic spectra is equivalent to a category of \emph{synthetic} Adams--Novikov spectral sequences, and in particular has a purely homotopy theoretic description without reference to algebraic geometry.


The first evidence that a large sector of the motivic category contains only topological information appears in the work of Thomason on algebraic K-theory and \'etale cohomology \cite{Thomason}. As a consequence of this work, Thomason proves that for well-behaved $\C$-schemes $X$, there is an equivalence $\mathrm{K}^{\mathrm{alg}} (X) [\beta^{-1}]^{\wedge} _{\ell} \simeq \mathrm{KU} (X(\C))^{\wedge} _{\ell}$, where $\beta$ is a certain Bott element. 
These ideas have been refined over the years, culminating in the following theorem.

\begin{thm}[\cite{VoeOpenProbs, DI, LevineComparison, LevineSlicemU, Ctau, ctauactalol, Pstragowski, cmmf}]\label{thm:C-main}\

  In $\SHC$
  there is a map $\tau : \Ss^1 \to (\G_m)_p$
  which enjoys the following properties:
  \begin{enumerate}
  \item[(1.\hspace{1em}]\hspace{-1.1em}$\mathrm{Generically\ Topological})$ The full subcategory of $\tau$-local objects in $\SHC^{cell}_p$ is equivalent to $\Sp_p$.
  \item[(2.\hspace{1em}]\hspace{-1.1em}$\mathrm{Algebraic\ Degeneration})$ The cofiber of $\tau$ (often denoted $C\tau$) is a commutative algebra, and the category of dualizable modules over $C\tau$ is equivalent to the category of $p$-completions of dualizable objects in $\QCoh(\mathcal{M}_{\mathrm{fg}})$, where $\Mfg$ is the moduli stack of formal groups.
  \item[(3.\hspace{1em}]\hspace{-1.1em}$\mathrm{Galois\ Reconstruction})$ There is a purely topological construction of $\SHC_p^{cell}$.
  \end{enumerate}
  
  We summarize this situation by saying that $\SH(\C)_p^{cell}$ is a $1$-parameter deformation of $p$-complete spectra with parameter $\tau$ and a purely algebraic special fiber.
\end{thm}

Although it does not appear in the statement, understanding the algebraic cobordism spectrum, $\MGL$, is key to proving this theorem. 
The importance of $\MGL$ to the study of cellular motivic spectra goes back to Voevodsky's work on the effective slice filtration \cite{VoeOpenProbs}.
In that work, Voevosky conjectured that the effective slice filtration of the unit over an algebraically closed field 
may be described in terms of the Adams-Novikov spectral sequence. 
Levine later proved Voevodsky's conjectures \cite{LevineComparison, LevineSlicemU}.
Voevodsky's notion of ``rigid homotopy groups'' also foreshadowed \emph{algebraic degeneration}.
In this framework he and Rezk predicted that the rigid Adams spectral sequence is equivalent to the algebraic Novikov spectral sequence \cite[p. 20-21]{VoeOpenProbs}, a result later proven in different language by Gheorghe--Wang--Xu \cite[Theorem 1.17]{ctauactalol}.
The fact that the category of $p$-complete $\C$-motivic spectra is \emph{generically topological} is strongly suggested by results of Dugger--Isaksen \cite[Section 2.6]{DI}, though to the best of our knowledge the full result did not appear in print until \cite{Pstragowski}.\todo{I can't find an earlier one at least, but that could easily just be me. Piotr does beat Tom's etale rigidity paper.}
The commutative algebra structure on $C\tau$ was first constructed by Gheorghe \cite{Ctau},
and the category of modules over $C\tau$ was then identified by Gheorghe--Wang--Xu in \cite{ctauactalol}.
Finally, 
Pstragowski \cite{Pstragowski} and Gheorghe--Isaksen--Krause--Ricka \cite{cmmf}
independently provided two different \emph{Galois reconstructions} of cellular $p$-complete $\C$-motivic spectra. 


With the case of $\C$ resolved we raise the following natural question:

\begin{qst}\label{qst:main}
  To what extent can this theorem be extended to a general field $k$?
\end{qst}

As stated, this question is too imprecise to admit a definitive answer, so we begin by refining two points of ambiguity: the appropriate subcategory of $\SHk$ one should consider and the meaning of purely topological. In our study of this question we have found that the appropriate subcategory is the category of Artin--Tate motivic spectra (defined below). Notably, over $\R$, if one restricts to the further subcategory of cellular objects, it becomes significantly more difficult to construct a topological model.

\begin{dfn}
  The category of Artin--Tate motivic spectra over $k$ is the smallest stable full subcategory
  $\SHkat \subset \SHk$,
  closed under tensor products and colimits, that contains the motives of finite \'etale $k$-algebras, 
  $\P_k^1$ and $(\P_k^1)^{-1}$.
\end{dfn}

The phrase `purely topological' has a double meaning.
On the one hand it refers to the input to the construction;
it should not depend directly on the arithmetic of $k$, instead using only the absolute Galois group, $G$, together with the character $G \to \widehat{\Z}^\times$ induced by the maximal cyclotomic extension of $k$ \footnote{More specifically, the $p$-complete category should only depend on the $\Z_p^\times$ component of this map.}.
On the other hand it asks for a construction which uses only the ``standard machinery of homotopy theory.''

The authors are not the first to take up questions of this nature.
Positselski has studied the question of Galois reconstruction for the category $\DM(k;\F_\ell)^{\at}$ of mixed Artin-Tate motives with mod $\ell$ coefficients \cite{DMRecon,GalKosz}. With a few exceptions, he has shown that when $k$ is a finite, local or global field, then $\DM(k,\F_\ell)^{\at}$ may be viewed as a derived category of filtered discrete $G$-modules with restricted sub-quotients. 

In a different direction, work of Bachmann, Elmanto and {\O}stv{\ae}r shows that, up to a completion, $\SH(S)$ is generically \'etale for a wide range of schemes $S$ \cite{TomEtaleI, TomEtaleII}. In particular, Bachmann--Elmanto--{\O}stv{\ae}r show that, after a suitable completion, \'etale localization corresponds to inverting $\tau$. Note that $\tau$ may not exist in the homotopy of the completed unit, so care must be taken to interpret this statement.
Furthermore, Bachmann showed that, again up to completion, the \'etale motivic category is equivalent to the category of hypercomplete sheaves of spectra on the small \'etale site.
Specializing to the case $S = \Spec k$, we find that, up to completion, the category of $\tau$-local objects in $\SH(k)$ admits a description in terms of Borel $G$-equivariant spectra.

Previously, work of Behrens and Shah \cite{BehrensShah} had taken up the question of when a suitable map $\tau$ exists over $\R$. 
Although there is no map $\tau:\Ss^1 \to (\G_m)^{\wedge}_2$ in $\SH(\R)$,
they prove that $\tau$ exists whenever a different class $\rho:\Ss^0 \to \G_m$ has been killed.
If only $\rho^2$ is killed, but not $\rho$ itself, then $\tau$ does not necessarily exist, but $\tau^2$ does.
Continuing in this way, they make sense of inverting $\tau$ in any $\rho$-complete situation.



In this paper we resolve \Cref{qst:main} in the case $k=\R$.
This begins with the observation that if one does not insist that the target of $\tau$ must be a completion of $\G_m$, then a suitable replacement can easily be constructed.

\begin{thm}\label{thm:main}
  In $\SHRat$
  there is an invertible object $Q$ and a map $\ta : \Ss^1 \to Q_2$
  which enjoys the following properties:
  \begin{enumerate}
  \item[(GT)] The full subcategory of $\ta$-local objects in $\SHR^{\at}_{i2}$ is equivalent to $\Sp_{C_2,i2}$.
  \item[(AD)] The cofiber of $\ta$ (denoted $C\ta$) is a commutative algebra, and the category of dualizable modules over $C\ta$ is equivalent to the category of dualizable objects in the derived category of Mackey-functor $\MU_*\MU$-comodules \footnote{See \Cref{sec:cta} for a precise definition of this category.}.
  \item[(GR)] There is a purely topological construction of $\SHRat$.
    In particular, we construct a commutative algebra $R_\bullet$ in filtered, $C_2$-equivariant spectra
    such that the category of filtered modules over $R_\bullet$ is equivalent to $\SHRat$.
    The commutative algebra $R_\bullet$ is the even slice--d\'ecalage of the $\MUR$-Adams tower for the sphere \footnote{See \Cref{sec:top} for precise definitions.}.    
  \end{enumerate}
  
  We summarize this by saying that $\SHR^{\at}_{i2}$ is a $1$-parameter deformation of $C_2$-equivariant stable homotopy theory with parameter $\ta$ and purely algebraic special fiber.
\end{thm}

The romanization of the Devanagari letter $\ta$ is `ta' \footnote{However, the reader who is not familiar with the Devanagari alphabet is advised that the pronunciation of $\ta$ in the majority of South Asian languages using Devanagari is given in IPA by /\textsubbridge{t}\textschwa/, which is closer to `tuh' in English than `ta.'}  and
our choice of this symbol will be explained in \Cref{rmk:ta-explanation}. As in the case over $\C$, we are also able to give an explicit formula for the homotopy groups of $C\ta$, though since we have not yet set up the appropriate notion of homotopy groups we will defer an explicit statement until later.
The subscript $i2$ appearing in the theorem statement refers to the category of modules over the $2$-completion of the unit. 
The reader who is worried by this departure from the norm may wish to read the section on notations and conventions before proceeding.

\ 

With the resolution of \Cref{qst:main} in the case of $\R$, the authors are hopeful that we might see this question resolved in its entirety in the near future.


	

\subsection{Examples of Artin--Tate motivic spectra}\ 

In this subsection, which is preparatory to all later material, we give a more comprehensive introduction to the category of Artin--Tate motivic spectra over a field. This begins with introducing the smaller subcategories of Artin and Tate objects, which together generate the category of Artin--Tate objects. The bulk of the subsection is spent building a collection of important examples of Artin--Tate motivic spectra. Over $\R$ a special role is played by the invertible objects and so we focus specific attention there. We close the section by discussing homotopy groups.

\begin{dfn}
  The categories of Artin, Tate and Artin--Tate motivic spectra over $k$ are the stable, full subcategories of $\SHk$, closed under tensor products and colimits, with the following generators:
  \begin{itemize}
  \item The category of Artin motivic spectra, $\SHka$, is generated by the motives of finite \'etale $k$-algebras.
  \item The category of Tate motivic spectra, $\SHkt$, is generated by the motives of $\P_k^1$ and $(\P_k^1)^{-1}$ \footnote{This category is sometimes referred to as the cellular category following \cite{DICell}.}.
  \item The category of Artin--Tate motivic spectra, $\SHkat$, is generated by the motives of finite \'etale $k$-algebras together with $\P_k^1$ and $(\P_k^1)^{-1}$.
  \end{itemize}
  Since these subcategories are closed under tensor products and colimits, they each inherit the structure of a stable, presentably symmetric monoidal category from $\SHk$. \footnote{In the sequel, we will state without remark results that are only known to be true after the characteristic of $k$ is inverted. The reader is invited to restrict themself to $k$ of characteristic zero if they prefer.}
\end{dfn}

We now turn to examples of objects in each of these categories, starting with the trivial and heading towards the non-trival.

\begin{exm} \label{exm:Sp-unit}
  Every stable, presentably symmetric monoidal category admits a unique symmetric monoidal functor from the category of spectra \cite[Corollary 4.8.2.19]{HA} \footnote{Another way of saying this is that $\Sp$ is the initial object of $\mathrm{CAlg}(Pr^{L,\mathrm{stab}})$.}.
  This gives us objects $X \otimes \o_k$ for every spectrum $X$.
  Of particular interest are the integer simplicial suspensions of the unit, $\Ss^{n} \otimes \o_k$.
\end{exm}

Inductively applying the homotopy purity theorem \cite[Theorem 3.2.23]{MorelVoevodsky} to the decomposition $\P^n _k = \mathbb{A}^{n} _k \coprod \P^{n-1} _k$, we learn that:

\begin{exm}[{\cite[Example 2.12]{DICell}}]
  For each $n$, the motivic spectrum associated to $\P^n_k$ is Tate.
\end{exm}

More generally, using a Bialynicki-Birula decomposition, Wendt shows that any smooth projective variety which admits a $\G_m$-action whose fixed points are discrete and rational is Tate \cite{Wendt} \footnote{In the case where the fixed points are not rational, the authors wonder what conditions are necessary to guarantees the motive is Artin--Tate.}.

\begin{exm}[{\cite[Theorem 6.4]{DICell}, \cite[Proposition 8.1]{HM}, \cite[Proposition 8.12]{etaPeriodic}}]
  Each of the following commutative algebras is Tate:
  \[ \MGL,\quad \kq,\quad \kgl,\quad \mathrm{M}\Z,\quad \mathrm{M}\F_p. \]
\end{exm}

\begin{exm} \label{exm:galois-corr}
  Essentially by definition, the Galois correspondence provides functors
  \begin{center}
    \begin{tikzcd}
      {\left\{ \footnotesize\parbox{2.25cm}{finite, continuous Gal($\overline{k}/k$)-sets} \right\}} \ar[r] &
      {\left\{ \footnotesize\parbox{1.5cm}{finite, etale $k$-algebras} \right\}^{op}} \ar[r, hook] \ar[d] &
      \mathrm{Sm}_k \ar[d] \\
      & \SHka \ar[r, hook] &
      \SHk,
    \end{tikzcd}
  \end{center}
  which give our first examples of Artin objects.
\end{exm}

Since the functor in the example above sends disjoint unions to sums there is no loss in generality if we restrict to transitive Gal($\overline{k}/k$)-sets (field extensions). In the case of $\R$ this gives only two objects: $\Spec(\R)$, which we will temporarily denote $\o_\R$ since is it the monoidal unit of the category, and $\Spec(\C)$. Using the natural map we get a cofiber sequence
\begin{align} \Spec(\C) \to \o_\R \xrightarrow{a} \Ss^{\C}. \label{eq:pic-C} \end{align}
On the level of $\C$-points (with Galois action), this cofiber sequence gives the representation sphere $\Ss^{\sigma}$.

Given a quadric in the plane, $V$, we can take its closure in $\P_k^2$ to obtain $\overline{V}$ which is a form of $\P_k^1$. If we let $x_1,\dots,x_n$ denote the points at infinity, then by the homotopy purity theorem \cite[Theorem 3.2.23]{MorelVoevodsky} we have a cofiber sequence
\[ V \to \overline{V} \to \bigoplus_{i=1,\dots,n} \P_{k(x_i)}^1. \]
Under the assumption that $V$ admits a rational point, $\overline{V}$ is just $\P^1_k$, and so we obtain our first non-trivial example of a motive which is neither Artin nor Tate, but is Artin--Tate:

\begin{exm}
  The motive of an affine quadric which admits a rational point is Artin--Tate.
\end{exm}

Over $\R$ there is a particular affine quadric of interest to us.
It is the object $Q$ which appeared in the statement of \Cref{thm:main}.

\begin{exm}
  Let $Q \coloneqq \{x^2+y^2=1\} \subset \AA^2 _{\R}$, which we shall call the algebraic circle.
  $Q$ has the additional property that it is a form of $\G_m$ \footnote{By Cartier duality, forms of $\G_m$ are classified by rank 1 lattices with Galois action. $Q$ corresponds to the unique nontrival action.}.
  Its group scheme structure comes from the usual rule for multiplication of complex numbers,
  namely
  \[ (x_1,y_1) \cdot (x_2,y_2) = (x_1 x_2 - y_1 y_2,\ x_1 y_2 + y_1 x_2). \]

  In \Cref{sec:ta} we will construct the maps $\ta: \Ss^1 \to Q_p$ using our explicit understanding of this group scheme.
\end{exm}

\begin{prop}[{\cite[Proposition 1.1]{HuPicard}}] \label{prop:pic-Q}
  There is an equivalence $Q \otimes \Ss^\C \simeq \P^1 _\R$.
  In particular, both $\Ss^\C$ and $Q$ are invertible.
\end{prop}

At this point we are ready to define the appropriate notion of homotopy groups for studying $\SHR^{\at}$.
In the case of $\SHC^{\at}$, the category has a bi-graded family of compact invertible generators given by the spheres $\Ss^{p - 2w} \otimes (\P_\C^1)^{\otimes w}$. Therefore, it is typical to study objects at the level of their bigraded homotopy groups, given by
\[ \pi_{p,w}^\C(X) \coloneqq \pi_0\Map ( \Ss^{p-2w} \otimes (\P_\C^1)^{\otimes w}, X ). \]
Similarly, $\SHRt$ has a bi-graded family of compact invertible generators given by the spheres $\Ss^{p-2w} \otimes (\P_\R^1)^{\otimes w}$, and it is typical to study objects through their bi-graded homotopy groups.
Using \Cref{prop:pic-Q} and \Cref{eq:pic-C}, the category $\SHR^{\at}$ has a tri-graded family of compact invertible generators given by the spheres $\Ss^{p-w} \otimes (\Ss^\C)^{\otimes q-w} \otimes (\P_\R^1)^{\otimes w}$, and we will study most objects through their tri-graded homotopy groups.

\begin{ntn}
  We endow the categories $\Sp$, $\Sp_{C_2}$, $\SHC$ and $\SHR$ with Picard gradings by spheres as follows:
  \begin{center}
    \renewcommand{\arraystretch}{1.1}
    \begin{tabular}{|l|l|l|} \hline
      Category & Picard & Spheres \\\hline\hline
      $\Sp$ & $\Z$ & $\Ss^p \simeq (S^1)^{\otimes p}$ \\\hline
      $\Sp_{C_2}$ & $\Z \times \Z$ & $\Ss^{p+q\sigma} \simeq (S^1)^{\otimes p} \otimes (S^\sigma)^{\otimes q}$ \\\hline
      $\SHC$ & $\Z \times \Z$ & $\Ss^{p,w} \simeq (S^1)^{\otimes p-2w} \otimes (\P_\C^1)^{\otimes w}$ \\\hline
      $\SHR$ & $\Z \times \Z \times \Z$ & $\Ss^{p,q,w} \simeq (S^1)^{\otimes p-w} \otimes (S^\C)^{\otimes q-w} \otimes (\P_\R^1)^{\otimes w}$ \\\hline
    \end{tabular}
    \renewcommand{\arraystretch}{1.0}
  \end{center}
  
  For any $p,q,w \in \Z$, and any $\R$-motivic spectrum $X$, we let $\pi^{\R}_{p,q,w} X$ denote the group of homotopy classes of maps,
  \[ (S^1)^{\otimes p-w} \otimes (S^\C)^{\otimes q-w} \otimes (\P_\R^1)^{\otimes w} \to X. \]
\end{ntn}

\begin{exm}
  For convenience we record what the invertible objects considered in this section look like under this new notation:
  \begin{center}
    \begin{tabular}{ccc}   
      $\Ss^{0,0,0} \cong \o_\R$ \qquad\qquad\qquad &
      $\Ss^{1,0,0} \cong S^1$ \qquad\qquad\qquad & 
      $\Ss^{0,1,0} \cong S^{\C}$ \\
      $\Ss^{0,1,1} \cong \G_m$ \qquad\qquad\qquad &
      $\Ss^{1,0,1} \cong Q$ \qquad\qquad\qquad &
      $\Ss^{1,1,1} \cong \P_\R^1$
    \end{tabular}
  \end{center}
  In the case of $\Ss^{0,0,0}$ we will often drop the indices for brevity, writing only $\Ss$.
  The two maps between Picard elements considered thus far, $\ta$ and $a$, become
  \[ \ta \in \pi_{0,0,-1}^\R \Ss_p \qquad\quad \text{ and } \qquad\quad a \in \pi_{0,-1,0}^\R \Ss. \]
\end{exm}

\begin{rmk}
  The tri-graded spheres in $\SHR^{\at}$ which are Tate are precisely the spheres of the form $\Ss^{p,q,q}$.
  For this reason the Tate category only sees a `slice' of the total information present in the tri-graded homotopy groups.

  The tri-graded spheres in $\SHR^{\at}$ which are Artin are precisely the spheres of the form $\Ss^{p,q,0}$.
  The Artin category similarly only sees a `slice' of the total homotopy groups.

\end{rmk}

\subsection{Comparison functors}\ 

The two simplest ways to interrogate a category are to study specific objects 
and to study the network of functors which relate it to other categories.
While the previous subsection provided preperatory background on specific objects, this subsection sets up the suite of functors which we will use to produce and study more general objects.
In the specific case of $\R$ this means studying the various ways we can move between $\SHR$, $\SHC$, $\Sp_{C_2}$ and $\Sp$.



\begin{rec} \label{rec:fields}
  Given a finite extension of fields $i : k \to \ell$, there are two pairs of adjunctions and one extra functor coming from the six functor formalism
  \[ i^* : \SHk \rightleftarrows \SH(\ell) : i_*, \qquad i_! : \SH(\ell) \rightleftarrows \SHk : i^! \quad \mathrm{ and } \quad i_{\sharp} : \SH(\ell) \to \SHk \]
  where $i^*$ is symmetric monoidal.
  Since $i$ is smooth, proper and unramified we have equivalences $i_! \simeq i_* \simeq i_{\sharp}$ and $i^! \simeq i^*$ \cite[Theorems 6.9 and 6.18]{hoyois}.
  These equivalences tell us that
  \begin{enumerate}
  \item $i^*$ and $i_*$ both commute with all limits and colimits.
  \item If $X$ is a smooth $\ell$-scheme, then $i_*X \simeq X$, where the second copy of $X$ is considered as a $k$-scheme.
  \end{enumerate}
  \todo{Restrictions ?}
\end{rec}

In \Cref{app:boilerplate}, as an example of the techniques showcased there, we show that for fields of characteristic zero
there is an equivalence of presentably symmetric monoidal categories
\[ \SH(\ell) \simeq \Mod ( \SHk ; \Spec(\ell) ), \]
and $i_*i^*(-) \simeq \Spec(\ell) \otimes -$.
This equivalence tells us that no new information enters the picture when we pass to a field extension.
Specializing to the case of $\C/\R$ the above equivalence becomes
\[ \SHC \simeq  \Mod ( \SHR ; \Spec(\C) ). \]

Using the descriptions of $i_*$ and $i_*i^*$ we can conclude that both $i^*$ and $i_*$ restrict to the full subcategory of Artin--Tate objects. Thus we obtain a similar description of the Artin--Tate category of a field extension:
\[ \SH(\ell)^{\at} \simeq \Mod ( \SHkat ; \Spec(\ell) ). \]
Specializing to the case of $\C/\R$ we will sometimes denote $\Spec(\C)$ by $Ca$ since it is (by the definition of $a$) the cofiber of the map $a: \Ss^{0,-1,0} \to \Ss^{0,0,0}$.
The above equivalence becomes
\[ \SHC^{\at} \simeq \Mod ( \SHR^{\at} ; Ca ). \]

At this point we turn to studying the case of $\R$ more closely.
Though some of the things we do after this point have obvious analogs in other cases,
many of our key maneuvers implicitly rely on the fact that $\R$ has a finite (and well-understood) absolute Galois group.
In particular, we will now assume that the reader is familiar with $C_2$-equivariant homotopy theory as in \cite{HHR}.
Equivariant homotopy theory first enters the picture through the Betti realization functors of \cite[Section 3.3]{MorelVoevodsky}.

\begin{rec}

  There is a commutative diagram of symmetric monoidal left adjoints:

  \begin{center}
    \begin{tikzcd}[column sep=tiny, row sep=tiny]
      
      \SHR \ar[rrrrrrr,"(-)_\C"] \ar[ddddddd,"\b"] & & & & & & & \SHC \ar[ddddddd,"\b"] \\
      & \Ss^{p,q,w} \ar[rrrrr,mapsto] \ar[ddddd,mapsto] & & & & & \Ss^{p+q,w} \ar[ddddd,mapsto] \\
      & & & & & & & \\
      & & & & & & & \\
      & & & & & & & \\
      & & & & & & & \\
      & \Ss^{p+q\sigma} \ar[rrrrr,mapsto] & & & & & \Ss^{p+q} \\
      \Sp_{C_2} \ar[rrrrrrr,"\Phi^e"] & & & & & & & \Sp.

    \end{tikzcd}
  \end{center}
  

  In the above diagram, $(-)_\C$ is the base change functor, $\Phi^e$ is the underlying functor and $\b$ are the Betti realization functors, induced by the assigment $X \mapsto X(\C)$.
  The inner square describing these functors on Picard objects can be verified by considering $Q$, $\G_m$ and $S^1$ directly \footnote{The observation that spurred the authors to begin this project was that upon restricting to categories of Tate objects the induced square on Picard groups is not a pullback, but with Artin--Tate objects it is in fact a pullback.}.
\end{rec}

\begin{rmk}
  Since the Betti realization of $\Ss^{p,q,w}$ is $\Ss^{p+q\sigma}$, we think of $w$ as recording the motivic weight and $p,q$ as providing a copy of $RO(C_2)$ in each weight.
  The Tate spheres are those of the form $\Ss^{p,q,q}$, so in the bigraded world the number of $\sigma$'s in the $RO(C_2)$-grading must equal the motivic weight.
  This restriction blinds one to the existence of a weight shifting element $\ta$ in $RO(C_2)$ grading $0$, which is the true analog of the $\C$-motivic $\tau$.
\end{rmk}

Having discussed Betti realization, we now introduce the less well-known functors $c$ and $c_{\C/\R}$ which provide sections of Betti realization.


\begin{rec} \label{rec:c}
  There are symmetric monoidal left adjoints $c$ and $c_{\C/\R}$,
  which fit into the commutative diagram
  \begin{center}
    \begin{tikzcd}[sep=huge]
      \Ss^{p+q\sigma} \ar[r, mapsto] & \Ss^{p,q,0} \ar[r, mapsto] & \Ss^{p+q\sigma} \\
      \Sp_{C_2} \ar[rr, "\mathrm{id}_{\Sp_{C_2}}", bend left=20] \ar[r, "c_{\C/\R}"] \ar[d,"\Phi^e"] & \SHR \ar[r, "\b"] \ar[d,"(-)_\C"] & \Sp_{C_2} \ar[d,"\Phi^e"] \\
      \Sp \ar[rr, "\mathrm{id}_{\Sp}"', bend right=20] \ar[r, "c"] & \SHC \ar[r, "\b"] & \Sp \\
      \Ss^{p} \ar[r, mapsto] & \Ss^{p,0} \ar[r, mapsto] & \Ss^p .
    \end{tikzcd}
  \end{center}

  The functor $c$ is the unique symmetric monoidal left adjoint coming from the fact that $\Sp$ is the unit of $Pr^{L, \mathrm{stab}}$ (see \Cref{exm:Sp-unit}).
  The functor $c_{\C/\R}$ is constructed in \cite{HellerOrmsby}, by beginning with the functor
  \[ \left\{ \mathrm{finite} \ \ \  C_2\mathrm{-sets} \right\} \to \SHR \]
  from \Cref{exm:galois-corr} and then extending it to all of $\Sp_{C_2}$.
\end{rec}

Since they are symmetric monoidal left adjoints, it is easy to see that the composites $\b \circ c : \Sp \to \Sp$ and $\b \circ c_{\C/\R} : \Sp_{C_2} \to \Sp_{C_2}$ are equivalent to the identity. In fact, upon restricting to the appropriate target category we uncover something more interesting:

\begin{thm}[{\cite{LevineComparison,HellerOrmsby,HellerOrmsbyII}}] \label{thm:c-ff}
  The symmetric monoidal functors $c$ and $c_{\C/\R}$ factor through the respective categories of Artin objects and provide equivalences,
  \[ c : \Sp \xrightarrow{\simeq} \SHCa \qquad \text{ and } \qquad c_{\C/\R} : \Sp_{C_2} \xrightarrow{\simeq} \SHRa. \]
\end{thm}

\begin{cor} \label{cor:betti-iso}
  The induced maps,
  \[ \pi_p\Ss \xrightarrow{c} \pi^\C _{p,0}\Ss \xrightarrow{\b} \pi_p\Ss  \qquad \text{and} \qquad \pi^{C_2}_{p+q\sigma}\Ss \xrightarrow{c_{\C/\R}} \pi^\R _{p,q,0}\Ss \xrightarrow{\b} \pi^{C_2} _{p+q\sigma}\Ss \]
  are isomorphisms for all $p$ and $q$.
\end{cor}

\begin{rmk}
  This corollary provides an identification of $\pi_{0,0,0}^{\R}\Ss$ with the Burnside ring of $C_2$.
  It is isomorphic to $\Z \oplus \Z$, with generators $1$ and $[C_2]$ and the relation $[C_2]^2 = 2[C_2]$.

  Morel's identification of $\pi_{n,n}\o_k$ in terms of Milnor--Witt K-theory provides another way of assigning names to elements of $\pi^\R_{0,0,0}\Ss$ \cite{Morelpizero}.
  In these terms $\pi_{0,0,0}^\R\Ss$ is generated by $1$ and $\eta [-1]$, subject to the relation $(\eta [-1])^2 = -2 \eta [-1]$. We denote by $\rho$ the element $-[-1]$.
  The translation between the two bases described above is given by $\eta \rho = 2 - [C_2]$.
\end{rmk}



\begin{rmk}
  When working over a general field one might think to replace the Betti realization by some variant of etale localization (see \cite{EldenJay}). The analog of this theorem would be that the category of Artin objects is already etale local.
  However, in many examples (such as finite fields) one consequence of the Morel connectivity theorem \cite{MorelConn} is that this is not true. An important precursor to answering \Cref{qst:main} will be producing the variant of equivariant homotopy theory which appears as the category of Artin objects.
\end{rmk}
  
\begin{rmk}
  Using the functors $c$ and $c_{\C/\R}$ the mapping spaces in $\SHC$ and $\SHR$ can be upgraded into mapping spectra and mapping $C_2$-spectra respectively. 
\end{rmk}



\begin{summary}\label{prop:comp}
  There are commuting diagrams
  \begin{center}
  \begin{tikzcd}
    \SHR \ar[r, "(-)_\C"] \ar[d,"\b"] &
    \SHC \ar[d, "\b"] &
    \Sp_{C_2} \ar[rr, "\mathrm{id}_{\Sp_{C_2}}", bend left=20] \ar[r, "c_{\C/\R}"'] \ar[d,"\Phi^e"] &
    \SHR \ar[r, "\b"'] \ar[d,"(-)_\C"] &
    \Sp_{C_2} \ar[d,"\Phi^e"] \\
    \Sp_{C_2} \ar[r,"\Phi^e"] &
    \Sp &
    \Sp \ar[rr, "\mathrm{id}_{\Sp}"', bend right=20] \ar[r, "c"] &
    \SHC \ar[r, "\b"] &
    \Sp ,
  \end{tikzcd}
  \end{center}
  
  and various functors considered in this section enjoy the following properties:
  \begin{enumerate}
  \item Each of $(-)_{\C}$, $\b$, $\Phi^e$, $c$ and $c_{\C/\R}$ is a symmetric monoidal left adjoint.
  \item Both $(-)_{\C}$ and $\Phi^e$ are right adjoints as well.
  \item All of the functors restrict to the categories of Artin--Tate objects and retain the properties listed in (1) and (2). In later sections we will almost exclusively deal with these restrictions so we will not use distinct notation for them.
  \item On Picard elements the functors in the digrams above behave as follows:
  \end{enumerate}  
  \begin{center}
    \begin{tabular}{llll}
      $ \b(\o_\R^{p,q,w}) \cong \o_{C_2}^{p+q\sigma} $ &
      $ \b(\o_\C^{s,w}) \cong \Ss^s$ &
      $ c_{\C/\R}(\o_{C_2}^{p+q\sigma}) \cong \o_\R^{p,q,0}$ \\
      $ c(\Ss^s) \cong \o_\C^{s,0}$  &
      $ (\o_\R^{p,q,w})_\C \cong \o_\C^{p+q,w}$ &
      $ \Phi^e(\o_{C_2}^{p+q\sigma}) \cong \Ss^{p+q}$ &
    \end{tabular}                  
  \end{center}
\end{summary}


\begin{proof}
  The only one of these which we have not already discussed is (2).
  The functor $(-)_\C$ can be described as tensoring up to $Ca$.
	Since $Ca$ is dualizable this functor commutes with all limits and colimits. Similarly, we may conclude that $\Phi^e$ commutes with all limits and colimits.
\end{proof}

\subsection{$\R$-motivic spectra as a deformation}\ 

In this subsection we refine the statement of \Cref{thm:main} into a sequence of precise claims which we will verify across the remainder of the paper. We summarized both \Cref{thm:C-main} and \Cref{thm:main} by saying that the category of Artin--Tate motivic spectra is a 1-parameter deformation of a purely topological category. To start, we clarify this, first over $\C$ and then over $\R$.

\begin{enumerate}
\item There is a distinguished element of the $p$-complete homotopy groups over $\C$,
  $\tau \in \pi_{0,-1}\Ss_p$, which maps to $1$ under Betti realization \footnote{The Betti realization map $\pi^{\C}_{0,-1}\Ss_p \to \pi_0\Ss_p$ is an isomorphism so the latter property uniquely identifies $\tau$.}. \cite[Lemma 23]{HKO}
\item The Betti realization functor factors through the category of $\tau$-local objects and provides a symmetric monoidal equivalence
  \[ \Mod ( \SHCat ; \Ss_p[\tau^{-1}] ) \simeq \Sp_{ip}. \]
  The idea for this goes back to \cite[Section 2.6]{DI}, but was first proven in \cite{Pstragowski}.  
\item The category $\SHCat$ can be equipped with the structure of a $\Sp_{ip}^{\Fil}$-algebra.
  More explicitly, this means we have a symmetric monoidal left adjoint
  \[ i^* : \Sp_{ip}^{\Fil} \to \SHCat, \]
  which sends the shift map in $\Sp_{ip}^{\Fil}$ to $\tau$. As a consequence of this $C\tau$ acquires the structure of a commutative algebra, a fact originally proven by Gheorghe \cite{Ctau}.
\item The category of modules over $C\tau$ is equivalent to a renormalization of the derived category of even $\BP_*\BP$-comodules. More specifically we have
  \[ \Mod ( \SHCat ; C\tau ) \simeq \IndCoh ( \Mfg )_{ip}. \]
  This equivalence sends $C\ta \otimes \Ss^{s,w}$ to $\Sigma^{s-2w} \omega_{\mathbb{G}/\Mfg} ^{\otimes w}$, 
  and on the level of homotopy groups it induces an equivalence
    \[ \pi_{s,w}^\C(C\tau) \cong \Ext_{\BP_*\BP}^{2w-s,w}(\BP_*, \BP_*). \] \todo{There is a small $p$-completion issue here...}
    The equivalence of categories is due to Gheorghe--Wang--Xu \cite{ctauactalol}, though the above isomorphism of groups was first proven by Isaksen \cite[Proposition 6.2.5]{StableStems} and then upgraded to an isomorphism of rings with all higher structure by Gheorghe \cite{Ctau}.
  \item There is an equivalence of symmetric monoidal categories between $\SHCat$ and $\Syn_{\MU, ip}^{\even}$, where the latter is the category of \emph{even} $\MU$-\emph{synthetic spectra} constructed in \cite{Pstragowski}. Notably, Pstr\k{a}gowski's construction uses no algebraic geometry, and requires only knowledge of the homotopy commutative ring object $\MU$ in $\Sp$. The comparison between these two categories proceeds via the close relationship between $\MGL$ and $\MU$. 
\item The adjunction $i$ is affine in the sense that there is a commutative algebra $R^{\C}_\bullet$ in $\Sp_{ip}^{\Fil}$ and an equivalence of symmetric monoidal categories under $\Sp_{ip}^{\Fil}$
  \[  \SHCat \simeq \Mod ( \Sp_{ip}^{\Fil} ; R^{\C}_\bullet ). \]
  Moreover, the commutative algebra $R^\C_\bullet$ admits an explicit construction which uses no algebraic geometry. It is given by
  \[ R_\bullet^{\C} \coloneqq \mathrm{Tot}^* \left( \tau_{\geq 2\bullet} \MU_p^{\otimes * + 1} \right). \]
    This construction uses only the commutative algebra $\MU$ in $\Sp$, and again the comparison proceeds via the close relationship between $\MGL$ and $\MU$. This approach is due to Gheorghe--Isaksen--Krause--Ricka \cite{cmmf}.
\end{enumerate}

In fact, every aspect of the picture over $\C$ extends to $\R$ in the simplest reasonable way.
The reader who has previously studied $\R$-motivic spectra may find this rather surprising (as the authors did), and we suggest that this highlights the primacy of the Artin--Tate category over the Tate category.

\begin{enumerate}
\item There is a distinguished element of the $p$-complete homotopy groups over $\R$,
  $\ta \in \pi_{0,0,-1}\Ss_p$, which maps to $1$ under Betti realization \footnote{As above, the Betti realization map $\pi^{\R}_{0,0,-1}\Ss_p \to \pi^{C_2}_0\Ss_p$ is an isomorphism so $\ta$ is uniquely identified.}. We will construct $\ta$ in \Cref{sec:ta}.
\item The Betti realization functor factors through the category of $\ta$-local objects, and provides a symmetric monoidal equivalence
  \[ \Mod ( \SHRat ; \Ss_2[\ta^{-1}] ) \simeq \Sp_{C_2,i2}. \]
  This will be proven as the main theorem of \Cref{sec:homology}.  
\item The category $\SHRat$ can be equipped with the structure of a $\Sp_{i2}^{\Fil}$-algebra.
  More explicitly, this means we have a symmetric monoidal left adjoint
  \[ i^* : \Sp_{i2}^{\Fil} \to \SHRat \]
  which sends the shift map in $\Sp_{i2}^{\Fil}$ to $\ta$.
  As a consequence of this $C\ta$ acquires the structure of a commutative algebra.
  We may also regard $\SHRat$ as a $\Sp_{C_2}$-algebra via the functor $c_{\C/\R}$.
  Tensoring these two functors together we obtain a symmetric monoidal left adjoint
  \[ i^* : \Sp_{C_2, i2}^{\Fil} \to \SHRat. \]
    These statements will be proven in \Cref{sec:top} as part of \Cref{prop:top-diagram}.
\item The category of modules over $C\ta$ is equivalent to a renormalization of the derived category of an abelian category of equivariant $\BP_*\BP$-comodules. More precisely, we have an equivalence of presentably symmetric monoidal categories
  \[ \Mod ( \SHRat ; C\ta ) \simeq \Mod(\Sp_{C_2} ; \uZt) \otimes_{\Z} \IndCoh ( \Mfg ). \]
  This equivalence sends $C\ta \otimes \Ss^{p,q,w}$ to $\Sigma^{(p-w)+(q-w)\sigma} \uZt \otimes \omega_{\mathbb{G}/\Mfg} ^{\otimes w}$ and on the level of tri-graded rings of homotopy groups it induces an isomorphism \footnote{If one wants the tensor product can be moved outside the Ext, but then it must be taken in a derived sense.}
  \[ \pi_{p,q,w}^\R C\ta \cong \bigoplus_{w+a-s = p} \Ext_{\BP_*\BP}^{s,2w}(\BP_*, \BP_* \otimes \pi_{a + (q-w) \sigma}^{C_2} \uZt). \]
  This will be proven as the main theorem of \Cref{sec:cta}.
\item The adjunction $i$ is affine in the sense that there is a commutative algebra
  $R^{\R}_\bullet$ in $\Sp_{C_2,i2}^{\Fil}$ and an equivalence of presentably symmetric monoidal categories under $\Sp_{C_2,i2}^{\Fil}$
  \[ \SHRat \simeq \Mod ( \Sp_{C_2,i2}^{\Fil} ; R^{\R}_\bullet ). \]
  Moreover, the commutative algebra $R^\R_\bullet$ admits an explicit construction which uses no algebraic geometry. It is given by
  \[ R_\bullet^{\R} \coloneqq \mathrm{Tot}^* \left( P_{2\bullet} \MU_{\R,2}^{\otimes * + 1} \right), \]
  where $P_{n}$ is the functor which takes the $n^{\mathrm{th}}$ slice cover of a $C_2$ spectrum.
    This construction uses only the commutative algebra $\MU_\R$ in $\Sp_{C_2}$ introduced by Landweber \cite{LandweberI, LandweberII}. 
  The comparison proceeds via understanding the close relationship between $\MGL$ and $\MU_\R$ over $\R$.
  This will be proven as the main theorem of \Cref{sec:top}.
\end{enumerate}




Using the second special element of the tri-graded homotopy groups of the sphere, $a \in \pi_{0,-1,0}^\R \Ss$, we can delve further into the structure of $\SHRat$. In order to do this we begin with a digression on the element $a_{\sigma}$ in $C_2$-equivariant homotopy theory.

\begin{prop}
  The category $\Sp_{C_2}$ can be viewed as a 1-parameter family with coordinate $a_\sigma$, special fiber $\Sp$ and generic fiber $\Sp$ in the sense that:
  \begin{itemize}
  \item The category of $a$-local objects can be identified with spectra, i.e.
    there is an equivalence of presentably symmetric monoidal categories
    \[ \Mod ( \Sp_{C_2} ; \Ss[ a_{\sigma}^{-1} ] ) \simeq \Sp. \]
  \item The cofiber of $a_\sigma$, which we will denote $Ca_\sigma$, can be endowed with a commutative algebra structure and there is an equivalence of presentably symmetric monoidal categories
    \[ \Mod ( \Sp_{C_2} ; Ca_{\sigma} ) \simeq \Sp. \]
    \item There is a monoidal left adjoint
      \[ i^* : \Sp^{\Gr} \to \Sp_{C_2}, \]
      which sends $\o(1)$ to $\Ss^{\sigma}$ \footnote{The authors will return to the question of whether this functor can be upgraded to a monoidal functor in from filtered spectra in a future work.}. Moreover, this adjunction is affine in the sense that we have an equivalence of categories
      \[ \Sp_{C_2} \simeq \Mod ( \Sp^{\Gr} ; R_\bullet^{C_2}), \]
      where $R_n^{C_2} \simeq \Sigma\R \mathrm{P}^{-n-1}_{-\infty}$.
  \end{itemize}  
\end{prop}

This result is certainly well known to experts. 
We give a proof in \Cref{app:boilerplate} as Examples \ref{exm:c2-underlying} and \ref{exm:c2-graded}.
Since therein the pair of functors $\Phi^e$ and $\Phi^{C_2}$ out of $\Sp_{C_2}$ become identified with modding out by $a_\sigma$ and inverting $a_\sigma$ respectively, this proposition produces a diagram of symmetric monoidal left adjoints
\begin{center}
  \begin{tikzcd}
    \Sp & & \Sp_{C_2} \ar[ll, "\Phi^{C_2}"']  \ar[dl, "{(-)[a_\sigma^{-1}]}"] \ar[dr, "Ca \otimes -"'] \ar[rr, "\Phi^e"] & & \Sp \\
     & \Mod ( \Sp_{C_2} ; \Ss[a_\sigma^{-1}] ) \ar[ul, "\simeq"'] & & \Mod ( \Sp_{C_2} ; Ca ) \ar[ur, "\simeq"] . & 
  \end{tikzcd}
\end{center}

We are now free to use the functor $c_{\C/\R}$ to push our description of $\Sp_{C_2}$ as a 1-parameter deformation into $\SHRat$ and obtain a description of that category as a 2-parameter deformation of $\Sp_{i2}$ with coordinates $\ta$ and $a$. For example, we can give $Ca$ the structure of a commutative algebra since $c_{\C/\R}(a_\sigma) = a$. Note that by \Cref{prop:comp} the commutative algebra structure obtained in this way is the same as the one coming from the equivalence $Ca \simeq \Spec(\C)$. 

While the 1-parameter deformations considered up to now had only a special and a generic fiber, when considered as a 2-parameter deformation $\SHRat$ has various ``limiting behavoirs'' which we presently study. Since each parameter can be set to either $0$ or $1$,
or left unspecified, there are 9 total categories of interest.
We summarize what we know in the following table:

{\renewcommand{\arraystretch}{1.4}
\begin{center}
  \begin{tabular}{|c|c|c|c|c|} \hline
    & unspecified & $\ta = 0$ & $\ta = 1$ \\\hline
    unspecified & $\SHRat$ & $\Mod(\Sp_{C_2, i2} ; \uZt) \otimes_{\Z} \IndCoh ( \Mfg )$ & $\Sp_{C_2,i2}$ \\\hline
    $a = 0$ & $\SHC^{\at}_{i2}$ & $\IndCoh ( \Mfg )_{i2}$ & $\Sp_{i2}$ \\\hline
    $a = 1$ & ? & $\Mod(\Sp ; \F_2[u_{2\sigma}]) \otimes_\Z \IndCoh ( \Mfg )$ & $\Sp_{i2}$ \\\hline
  \end{tabular}
\end{center}
}


The only identification in this table which we have not yet discussed is the one in the bottom middle.
However, one can easily obtain this from the identification of $C\ta$-modules and \Cref{lem:tensor-mod}:
\begin{align*}
  \Mod( \SHRat ; C\ta[a^{-1}] ) &\simeq \Sp_{i2} \otimes_{\Sp_{C_2,i2}} \Mod( \SHRat ; C\ta ) \\
  &\simeq \Sp_{i2} \otimes_{\Sp_{C_2,i2}} \Mod(\Sp_{C_2,i2} ; \uZt) \otimes_{\Z} \IndCoh ( \Mfg ) \\
  &\simeq \Mod(\Sp_{i2} ; \Phi^{C_2} \uZt) \otimes_{\Z} \IndCoh ( \Mfg ) \\
  &\simeq \Mod(\Sp ; \F_2[u_{2\sigma}]) \otimes_{\Z} \IndCoh ( \Mfg ) ,
\end{align*}
where the final step is just the identification of $\Phi^{C_2}\uZ$ with $\F_2[u_{2\sigma}]$ where $u_{2\sigma}$ is a polynomial generator in degree 2. As a corollary, we can give a description of the tri-graded homotopy groups of several objects in $\SHRat$ in terms of better understood categories.

\begin{cor}
  There are isomorphisms of rings:
  \begin{enumerate}
  \item $\pi_{p,q,w}^\R(Ca) \cong \pi_{p+q,w}^\C\Ss $,
  \item $\pi_{p,q,w}^\R(Ca \otimes C\ta) \cong \Ext_{\BP_*\BP}^{2w-p-q,w}(\BP_*, \BP_*) $,
  \item $\pi_{p,q,w}^\R(Ca[\ta^{-1}]_p) \cong \pi_{p+q}\Ss_p $,
  \item $\pi_{p,q,w}^\R(C\ta) \cong \bigoplus_{w+a-s = p} \Ext_{\BP_*\BP}^{s,2w}(\BP_*, \BP_* \otimes \pi_{a + (q-w) \sigma}^{C_2} \uZt)$,
  \item $\pi_{p,q,w}^\R(C\ta[a^{-1}]) \cong \bigoplus_{w+2a-s = p} \left( \F_2\{u^{2a}\} \otimes_{\F_2} \Ext_{\BP_*\BP}^{s,2w}(\BP_*, \BP_*/2)\right) $,    
  \item $\pi_{p,q,w}^\R(\Ss_2[\ta^{-1}]) \cong \pi_{p,q}^{C_2}\Ss_2$,
  \item $\pi_{p,q,w}(\Ss_2[a^{-1}, \ta^{-1}]) \cong \pi_{p}\Ss_2$.
  \end{enumerate}  
\end{cor}

We now turn to the category of $a$-local objects, only labelled ``?'' in the table above. It is one of the most mysterious actors in the story told in this paper and we shall devote \Cref{sec:a-local} to its study. As we shall see, it records a deformation of the category of spectra sitting between that which corresponds to the Adams--Novikov spectral sequence and that which corresponds to the classical Adams spectral sequence. The role it plays in mediating between these two spectral sequences remains to be understood.

Since $a = c_{\C/\R}(a_\sigma)$ and $a_\sigma = \b (a) $, we can construct a diagram of symmetric monoidal left adjoints
\begin{center} \begin{tikzcd}
    \SHR \ar[r, "{(-)[a^{-1}]}"] \ar[d,"\b"] &
    \Mod ( \SHR ; \Ss[a^{-1}] ) \ar[d,"\b"] \\    
    \Sp_{C_2} \ar[r, "\Phi^{C_2}"] &
    \Sp .
\end{tikzcd} \end{center}
As the diagram shows, the category $\Mod(\SHR; \Ss_2 [a^{-1}])$ can be viewed as the target category of an ``$\R$-motivic geometric fixed points'' functor.
Since $\Phi^{C_2} (\MU_\R) \simeq \MO$, one might guess that $\Mod(\SHRat; \Ss_2 [a^{-1}])$ is related to Pstragowski's synthetic category $\Syn_{\MO}$ \footnote{The spectrum $\MO$ is a sum of copies of $\F_2$, and as we shall see this implies $\Syn_{\MO} \simeq \Syn_{\F_2}$.}. Indeed, we construct a comparison functor.

\begin{prop}
  There is a symmetric monoidal left adjoint
  \[\Re_{\F_2} : \Mod(\SHRat; \Ss_2 [a^{-1}]) \to \Syn_{\F_2}\]
  which sends $\Ss^{p,q,w}$ to $\Ss^{p,w}$.
\end{prop}

On the other side, an interesting functor into $\Mod(\SHRat; \Ss_2 [a^{-1}])$ may be constructed using the Bachmann--Hoyois motivic norm functors.

\begin{rec} \label{rec:norms}
  In \cite{norms}, Bachmann and Hoyois construct symmetric monoidal norm functors along finite \'etale maps.
  These norm functors may be thought of as an indexed tensor product and in the case of $\R$ they fit into the following diagram showing compatability with the Hill--Hopkins--Ravenel norm functors in equivariant homotopy theory \cite[Section 11]{norms}:
  \begin{center}
  \begin{tikzcd}
  \SHC \ar[r, "\mathrm{Nm}_{\C}^\R"] \ar[d, "\b"] & \SHR \ar[d, "\b"] \\
  \Sp \ar[r, "\mathrm{Nm}_e^{C_2}"] & \Sp_{C_2}.
  \end{tikzcd}
  \end{center}
  From \cite[{Example 3.5, Lemma 4.4, Example 4.10}]{norms} we can conclude that on bigraded spheres $\mathrm{Nm}_\C^{\R}(\Ss_\C^{s,w}) \simeq \Ss^{s,s,2w}$. The functor $\mathrm{Nm}_{\C}^\R$ is polynomial of degree 2 and when applied to a direct sum we have a formula \cite[Corollary 5.13]{norms}
  \[ \mathrm{Nm}_{\C}^\R(X \oplus Y) \simeq \mathrm{Nm}_{\C}^\R(X) \oplus i_*(X \otimes Y) \oplus \mathrm{Nm}_{\C}^\R(Y). \]
  Since $\mathrm{Nm}_{\C}^\R$ commutes with sifted colimits, it follows from the above that it restricts to a functor on the categories of Artin--Tate objects
  \[\mathrm{Nm}_{\C}^\R : \SHR^{\at} \to \SHC^{\at}.\]
\end{rec}

While the norm functor $\mathrm{Nm}_{\C}^\R$ is not exact, it becomes exact after $a$ is inverted.

\begin{prop}
  The composite
  \[ \SHC_{i2}^{\at} \xrightarrow{\mathrm{Nm}_\C^\R} \SHRat \xrightarrow{(-)[a^{-1}]} \Mod ( \SHRat ; \Ss_2[a^{-1}] ) \]
  is a symmetric monoidal left adjoint.
  On Picard elements this composite sends $\Ss_2^{s,w}$ to $\Ss_2^{s,*,2w}[a^{-1}]$.
  The map $\tau$ is sent to $\ta^2$.
\end{prop}

In conclusion, we have a pair of functors 
\[\SHC^{\at}_{i2} \xrightarrow{\mathrm{Nm}_\C^\R(-)[a^{-1}]} \Mod( \SHRat; \Ss_2 [a^{-1}]) \xrightarrow{\Re_{\F_2}} \Syn_{\F_2},\]
which are each the identity on spectra after inverting $\ta$ and $\tau$.
It would be very interesting to obtain a better computational understanding of the $a$-local category and the behavior of these two functors. More ambitiously, we ask:

\begin{qst}
  Is there a good description of the full subcategory of $a$-local objects in $\SHRat$?
  What geometric information does the $a$-localization of a smooth projective variety $X$ remember?
\end{qst}


The careful reader will have noticed that we switched from $p$-completing when discussing $\C$ to $2$-completing when discussing $\R$. The reason for this is that at an odd prime the category $\SHR^{\at}_{ip}$ admits the following simple description in terms of $\SHCat$.

\begin{prop}
  There is an equivalence of categories
  \[ \SHR^{\at}_{ip} \simeq \Sp_{ip} \times \SHCat \times \SHCat. \]
\end{prop}

This will be proven in \Cref{sec:odd}.

\subsection{Potpourri}\

In this subsection, we collect a number of additional results about the category of Artin--Tate $\R$-motivic spectra that are worth highlighting.


\subsubsection{Trigraded $\R$-motivic homology and Steenrod algebra}
From a computational point of view, an important first step in studying the trigraded homotopy groups of $\R$-motivic spectra is the computation of the trigraded homology of a point and the trigraded dual Steenrod algebra. Leaning on known bigraded computations, we make these computations in \Cref{sec:homology}.
%

\begin{prop}
As trigraded commutative rings, with $|\ta|=(0,0,-1)$, there are isomorphisms
$$\pi^{\R}_{p,q,w} \HZt \cong \left(\pi^{C_2}_{p+q\sigma}\uZt\right)[\ta],$$
$$\pi^{\R}_{p,q,w} \HFt \cong \left(\pi^{C_2}_{p+q\sigma}\uFt\right)[\ta].$$
Here, $\pi^{C_2}_{p+q\sigma}\uZt$ and $\pi^{C_2}_{p+q\sigma}\uFt$ are placed in degrees $(p,q,0)$
\end{prop}

\begin{rmk}\label{rmk:ta-explanation}
  There are well-known elements $\tau \in \pi_{1,-1-1} ^{\R} \HFt$ and $u_{\sigma} \in \pi_{1-\sigma} ^{C_2} \uFt$. Write $u \in \pi_{1,-1,0} ^{\R} \HFt$ for the corresponding element.
  Then there is a relation (see \Cref{cor:rels})
  \[\tau = \ta \cdot u,\]
  i.e.
  \[\mathrm{tau} = \mathrm{ta} \cdot u.\]
  This was the original motivation for our choice of the character $\ta$.
\end{rmk}

\begin{prop} 
The trigraded dual Steenrod algebra $\pi^{\R}_{*,*,*} \left(\HFt \otimes \HFt\right)$ is isomorphic to
$$\left(\HFt\right)_{***} [\tau_0,\tau_1,\cdots,\xi_1,\xi_2,\cdots]/(\tau_i^2 = \ta (a \tau_{i+1} + u \xi_{i+1} + a \tau_0 \xi_{i+1}))$$
Here, $|\tau_i|=(2^i,2^{i}-1,2^{i}-1)$ and $|\xi_{i}|=(2^i-1,2^i-1,2^i-1)$.
\end{prop}

\begin{rmk}
The reader may notice that the `negative cone' in the $C_2$-equivariant homology of a point appears in the above formulas.  This is a feature of the Artin--Tate category, as opposed to the Tate category.
\end{rmk}

We expect that a computer could be coaxed into computing the $\HFt$-Adams spectral sequence for the tri-graded homotopy of $C\ta$, and that the natural maps between this spectral sequence and the $\HFt$-Adams spectral sequence for the tri-graded homotopy of the sphere would provide many of the differentials in the latter. This technique would likely be the most efficient way of computing both tri-graded $\R$-motivic stems and (in the long-run) the bi-graded $C_2$-equivariant homotopy groups. However, as the quad-graded nature of this computation would begin to tax our ability to visualize data it is unlikely that such a computation could be accomplished without the aid of a well-developed software suite for manipulating spectral sequence data.

%
%

\subsubsection{The effective slice filtration}
In \Cref{sec:eff}, we study Voevodsky's effective slice filtration. 
The effective slice filtration associates to any object $E \in \SHk$ a tower
\[ \dots \to f_{2} E \to f_1 E \to f_0 E \to f_{-1} E \to f_{-2} E \to \dots \to E.\]
The main result of the section is the following proposition, proven as \Cref{thm:eff-fil} which identifies the Betti realization of this tower.

\begin{prop}
  The functor $i_* : \SHRat \to \Sp_{C_2, i2}^{\Fil}$ sending $E$ to the tower
  \[ \dots \to \Map_{\SHRat} (\Ss_2 ^{0,0,1}, E) \xrightarrow{\ta} \Map_{\SHRat} (\Ss_2 ^{0,0,0}, E) \xrightarrow{\ta} \Map_{\SHRat} (\Ss_2 ^{0,0,-1}, E) \to \dots \]
  is equivalent to the functor $\b \circ f_\bullet : \SHRat \to \Sp_{C_2, i2}^{\Fil}$ taking $E \in \SHRat$ to the Betti realization of its effective slice tower.
\end{prop}

\begin{rec}
  Voevodsky has defined the ``rigid homotopy groups'' \cite[Definition 5.1]{VoeOpenProbs} of an $\R$-motivic spectrum $X$ to be
  \[\pi_{p,q,w} ^{\R,rig} X \coloneqq \pi_{p,q,w} ^{\R} s_w X,\]
  where $s_w X$ is $n$th effective slice of $X$, i.e. the cofiber of $f_{n+1} X \to f_n X$.
\end{rec}

Using the notion of rigid homotopy groups, he observed the phenomenon of \emph{algebraic degeneration} in motivic homotopy theory.
Since the associated graded of this tower also computes the homotopy groups of $X \otimes C\ta$, it follows from the proposition that $\pi_{p,q,w} ^{\R,rig} X \cong \pi_{p,q,w} ^{\R} X \otimes C\ta$ when $X \in \SHRat$.
This shows that Voevodsky's ideas about algebraic degeneration line up with notions coming from the cofiber of $\ta$ (or $\tau$) philosophy.

The proposition also implies that the so-called $C_2$-effective spectral sequence (see \cite{Kong}) is interchangeable with the $\ta$-Bockstein spectral sequence.


\subsubsection{The functor $\nur$}
%
%
In \Cref{sec:nu}, we will construct a lax symmetric monoidal functor
\[\nur : \Sp_{C_2, i2} \to \SHRat,\]
which is a section of the Betti realization functor and sends $\Ss_2^{n\rho}$ to $\Ss_2^{n,n,n}$.
This functor is defined similarly to Pstragowski's synthetic analog functor \cite[Definition 4.3]{Pstragowski}.
Unlike the section $c_{\C/\R}$ of Heller--Ormsby, the functor $\nur$ is not exact.
However, it is better adapted to the construction of interesting $\R$-motivic spectra, as the following result shows.

\begin{prop}
  There are equivalences of commutative algebras,
  \begin{center}
  \begin{tabular}{lll}
    $ \nur \uFt \simeq \HFt,  \qquad\qquad$ &
    $ \nur \uZt \simeq \HZt, \qquad\qquad$ &    
    $ \nur \MU_{\R,2} \simeq \MGL_2, $ \\
    $ \nur \ku_{\R,2} \simeq \kgl_2, $ &
    $ \nur \ko_{C_2,2} \simeq \kq_2.$
  \end{tabular}
  \end{center}
\end{prop}

On the one hand, this proposition suggests that each of the $2$-complete $\R$-motivic homology theories above is not particularly far from being purely topological.
On the other hand, it suggests that one may profitably define $\R$-motivic analogs of $2$-complete $C_2$-equivariant homology theories by applying $\nu_\R$.
For example, if we had a good definition of $2$-complete $C_2$-equivariant connective topological modular forms $\tmf_{C_2, 2}$, then we could define a spectrum of $\R$-motivic modular forms as $\nu_{\R} \tmf_{C_2, 2}$. Unfortunately, no such definition is currently known.

\subsubsection{Further deformations}

One of our motivations for this project was the hope that a deformation-theoretic description of Artin--Tate $\R$-motivic stable homotopy theory would suggest other profitable deformations of classical stable homotopy theories.

In \Cref{app:def-const}, we show how a $1$-parameter deformation of a stable presentably symmetric monoidal category $\CC$ may be associated to a choice of object $E \in \CC$ and filtration $\{\CC_{\geq k}\}$ of $\CC$.
The case of $\SHCat$ is associated to the case where $E = \MU_{p}$, $\CC = \Sp_{ip}$ and $\CC_{\geq k}$ consists of the $2k$-connective objects. 
The case of $\SHRat$, on the other hand, is associated to the case where $E = \MU_{\R, 2}$, $\CC = \Sp_{C_2, i2}$ and $\CC_{\geq k}$ consists of the regular slice $2k$-connective objects.
Note that while the completions are necessary for the link to motivic homotopy theory, for the purpose of studying classical homotopy theory one may just as well work with integral deformations.

This suggests several other deformations that may be profitable to study.
\begin{enumerate}
\item We can take the deformation of $C_{2^n}$-equivariant homotopy theory with respect to the norm $\mathrm{N}_{C_2} ^{C_{2^n}} \MU_\R$ and the even slice filtration.
\item At odd primes, one could try to construct the deformation associated to the hypothetical spectrum $\mathrm{BP}_{\mu_p}$ and the even regular slice filtration.
\item One could take the deformation of $\Sp_{C_p}$ associated to $\underline{\mathbb{F}}_p$ and the regular slice filtration.
At the prime $2$, this is connected to a variant of the $\underline{\mathbb{F}}_2$-Adams spectral sequence.
Dylan Wilson has suggested that this variant may be more tractable at odd primes than the $\underline{\mathbb{F}}_p$-Adams spectral sequence, considering his forthcoming work with Krishanu Sankar proving that $\underline{\mathbb{F}}_p \otimes \underline{\mathbb{F}}_p$ is not a free $\underline{\mathbb{F}}_p$-module.
\item The deformation of $\Syn_{\MU}$ based on $\nu \F_p$ and an appropriately chosen filtration would likely shed much light on the motivic Adams spectral sequence over $\C$. Such a category might be called a ``bisynthetic''.
\end{enumerate}

%


\subsection*{Notations and conventions}\ 



Throughout this paper the term category will refer to an $\infty$-category as developed by Joyal and Lurie. In some places we use the term $1$-category, which is short-hand for 1-truncated $\infty$-category. We will also assume the reader is familiar with higher algebra as developed in \cite{HA}.

Throughout this paper, filtered and graded objects will be ubiquitous.
We adopt the convention that a filtered object in a category $\CC$ is a diagram of the form
  \[ \cdots \to C_2 \to C_1 \to C_0 \to C_{-1} \to C_{-2} \to \cdots, \]
i.e. that the maps decrease the index variable. We provide a more complete introduction to filtered objects in \Cref{app:fil}, where we set up notation for several standard constructions.


Let $\CC$ denote a stable, presentably symmetric monoidal category.
We will let $\CC_p$ denote the category of $p$-complete objects in $\CC$; this category acquires a symmetric monoidal structure through the completion of the symmetric monoidal structure on $\CC$.
Similarly, if $X$ is an an object of $\CC$, we let $X_p$ denote the $p$-completion of $X$.
If $\CC$ has a set of compact generators, then we let $\CC_{ip}$ denote the ind-completion of the full subcategory of $\CC$ generated under finite colimits and retracts by the $p$-completions of compact objects. In the situation where the unit is compact and all compact objects are dualizable, this definition admits the following simplification: $\CC_{ip} \simeq \Mod ( \CC ; \o_p)$. The tensor product can then be described as the relative tensor product over $\o_p$. The reason for this equivalence is that tensoring with a dualizable object commutes with limits, so for $X$ dualizable $X_p \simeq X \otimes \o_p$.

Our reason for using $\CC_{ip}$ over $\CC_p$ is that, even if the unit in $\CC$ is compact, the unit in $\CC_p$ may not be compact. On the other hand, compactness of the unit in $\CC$ implies compactness of the unit in $\CC_{ip}$. This will make certain arguments more direct in the $ip$ case.
The reader who strongly prefers the usual notion of $p$-completion will be relieved to know that in convenient cases (such as those in which we work throughout this paper) $(\CC_{ip})_p \simeq \CC_p$. Therefore, all the main theorems above admit $p$-complete analogs. 

When $\CC$ is a stable category and $X,Y \in \CC$ are objects, we will typically let $\Map_{\CC} (X,Y)$ denote the spectrum of maps from $X$ to $Y$.
One exception is when $\CC$ comes with a natural enrichment to $C_2$-spectra, such as when $\CC = \SHR$. In that case, we will regard $\Map_{\CC} (X,Y)$ as a $C_2$-spectrum.
When we want to access the underlying space of maps, we shall use the notation $\Omega^{\infty} \Map_{\CC} (X,Y)$.









\subsection*{Acknowledgements}\

We thank
Tom Bachmann,
Mark Behrens,
Dan Isaksen,
Piotr Pstr\k{a}gowski and
Zhouli Xu
for helpful conversations.
During the course of this work, the middle author was supported by NSF grant DMS-1803273,
and the last author by an NSF GRFP fellowship under Grant No. 1745302.

\section{Constructing the element ta} \label{sec:ta}
The map $\tau$ and its properties are the most striking feature of the category of $p$-complete motives over $\C$. In this section we construct a map $\ta$ which plays the role of $\tau$ over $\R$ and verify its basic property: its Betti realization is $1$. 

\begin{thm}\label{thm:ta}
  For all primes $p$, there is a class $\ta \in \pi_{0,0,-1} ^\R \Ss_p$ with the following properties:
  \begin{enumerate}
    \item Under Betti realization, $\ta$ goes to $1 \in \pi_{0} ^{C_2} \Ss_p$.
    \item Under base change to $\C$, $\ta$ goes to $\tau \in \pi_{0,-1} ^\C \Ss_p$.
  \end{enumerate}
\end{thm}

Surprisingly, as with the element $\tau$ in previous work, we will not need to use any information about the construction of $\ta$ besides properties (1) and (2) outside this section.
Since the construction of $\ta$ will be completely analogous to the construction of $\tau$ over $\C$, we begin by recalling that construction. The construction below is inspired by \cite[Remark on p. 22]{HKO} and \cite[Section 4]{TomEtaleII}.

\subsection{Constructing $\tau$ over $\C$}\

We begin by fixing a compatible sequence of primitive $(p^k)^{\mathrm{th}}$ roots of unity
$\{\zeta_{p^k}\}_{k \geq 0}$. 
We may then construct the following diagram of varieties over $\C$
\begin{center}
  \begin{tikzcd}
    \Spec(\C) \coprod \Spec(\C) \ar[d,"{(\zeta_{p^k},1)}"] \ar[r] & \Spec(\C) \ar[d,"1"] \\
    \G_m \ar[r,"{[p^k]}"] & \G_m .
  \end{tikzcd}
\end{center}
After passing to the associated diagram of motivic spaces we can add a third column by taking cofibers:
\begin{center}
  \begin{tikzcd}
    \Spec(\C) \coprod \Spec(\C) \ar[d,"{(\zeta_{p^k},1)}"] \ar[r] & 
    \Spec(\C) \ar[d, "1"] \ar[r] & 
    S^1 \ar[d,"\tau_k"] \\
    \G_m \ar[r,"{[p^k]}"] & \G_m \ar[r] & \G_m /{[p^k]}.
  \end{tikzcd}
\end{center}
Using the compatibility of the chosen primitive $(p^k)^{\mathrm{th}}$ roots of unity,
the maps $\tau_k$ so defined are compatible in the sense that there are commutative diagrams
\begin{center}
  \begin{tikzcd}
    S^1 \ar[r,"\mathrm{id}"] \ar[d,"\tau_k"] & S^1 \ar[d,"\tau_{k-1}"] \\
    \G_m /{[p^k]} \ar[r] & \G_m /{[p^{k-1}]}.
  \end{tikzcd}
\end{center}


Stabilizing and using the fact that $\Sigma^\infty [p^k] : \Ss^{1,1} \to \Ss^{1,1}$ is equivalent to $p^k$ over $\C$,\footnote{This follows from \Cref{cor:betti-iso}, since the Betti realization of $[p^k]$ is clearly $p^k$.} we find that we have constructed a compatible system of elements $\tau_k \in \pi_{0,-1} ^\C \Ss/p^k$. Passing to the limit, we obtain the element $\tau \in \pi_{0,-1} ^\C \Ss_p$.

\begin{prop} \label{prop:tau-1}
  For some choice of $\{\zeta_{p^k}\}_{k \geq 0}$, we have $\b(\tau) = 1$. \footnote{It seems likely that the choice of roots of unity which yields $1$ is $ \exp\left( \frac{2\pi i}{p^k} \right)$.}
\end{prop}

\begin{proof}
  We will show that $\b(\tau) \in \pi_0 \Ss_p = \Z_p$ is a $p$-adic unit.
  Then, using the action of $Aut(\mu_{p^\infty}) = \Z_p^\times$ on systems of $(p^k)^{\mathrm{th}}$ roots of unity, we may change $\tau$ by multiplication by any $p$-adic unit.

It now suffices to show that the Betti realizations of the unstable maps $\tau_k : S^1 \to \G_m / [p^k]$ are surjective on $\pi_1$. Since $[p^k] : \C^\times \to \C^\times$ is a principal $\mu_{p^k}$-fibration, it is classified by a map $\C^\times \to B\mu_{p^k}$ and we can form the following diagram:
\begin{center}
  \begin{tikzcd}
    S^0 \ar[d, "{(\zeta_{p^k},1)}"] \ar[r] & 
    S^0 \ar[d, "{(\zeta_{p^k},1)}"] \ar[r] & 
    * \ar[d] \ar[r] & 
    S^1 \ar[d] \\
    \mu_{p^k} \ar[r] &
    \C^\times \ar[r,"{[p^k]}"] & 
    \C^\times \ar[d] \ar[r] & 
    \C^\times / [p^k] \ar[dl, dashed] \\
    & & B\mu_{p^k} 
  \end{tikzcd}
\end{center}

The dashed map $\C^\times / [p^k] \to B\mu_{p^k}$ is an isomorphism on $\pi_1$ after Betti realization, so it suffices to show that the composite map $S^1 \to \C^\times / [p^k] \to B\mu_{p^k}$ is surjective on Betti-realized $\pi_1$. This follows from the fact that the map $S^1 \to B\mu_{p^k}$ is adjoint to the map $S^0 \xrightarrow{(\zeta_{p^k},1)} \mu_{p^k}$.
\end{proof}

\subsection{Constructing $\ta$ over $\R$}\

We now imitate the construction of $\tau$ above to construct $\ta$ and prove \Cref{thm:ta}.
The key point is that there are two forms of $\G_m$ over $\R$.
While the roots of unity $\zeta_n$ do not lie in
\[ \{\pm 1\} \subset \G_m (\R) \subseteq \G_m (\C) = \C^\times, \]
the other form of $\G_m$ over $\R$ (which is the ``algebraic circle'', $Q$, given by $\{x^2 + y^2 = 1\}$) has
\[ S^1 = Q(\R) \subseteq Q(\C) \cong \G_m (\C) = \C^\times. \]
As such, we are free to imitate the construction of $\tau$ above with $\G_m$ replaced by the algebraic circle $Q$.

Let $\{\zeta_{p^k}\}_{k \geq 0}$ be the system of primitive $(p^k)^{\mathrm{th}}$ roots of unity satisfying the conclusion of \Cref{prop:tau-1}. As before, we form the diagram of $\R$-motivic spaces:
\begin{center}
  \begin{tikzcd}
    \Spec(\R) \coprod \Spec(\R) \ar[d,"{(\zeta_{p^k},1)}"] \ar[r] & 
    \Spec(\R) \ar[d, "1"] \ar[r] & 
    S^1 \ar[d,"\ta_k"] \\
    Q \ar[r,"{[p^k]}"] &
    Q \ar[r] &
    Q /{[p^k]}.
  \end{tikzcd}
\end{center}
Before finishing the construction and proving \Cref{thm:ta} we need a short lemma which identifies the map $[p^k]$.

\begin{lem}\label{lem:sus-pk}
  The map $\Sigma^\infty [p^k] : \Ss^{1,0,1} \to \Ss^{1,0,1}$ is homotopic to $p^k : \Ss^{1,0,1} \to \Ss^{1,0,1}$.
\end{lem}

\begin{proof}
  Since the map $\pi_{0,0,0} ^\R \to \pi_{0} ^{C_2}$ induced by Betti realization is an isomorphism by \Cref{cor:betti-iso}, it suffices to show this after Betti realization.
  Upon Betti realization, this is a map $\Ss^1 \to \Ss^1$ which induces multiplication by $p^k$ on both geometric fixed points and on the underlying spectrum.
   The map $\pi_0^{C_2} \to \Z \oplus \Z$ given by (underlying, geometric fixed points) is injective.
%
\end{proof}

\begin{proof}[Proof of \Cref{thm:ta}]
  Stabilizing and applying \Cref{lem:sus-pk},
  we obtain a compatible system of classes $\ta_k \in \pi_{0,0,-1} ^\R \Ss/p^k$
  which give rise to a class $\ta \in \pi_{0,0,-1} ^\R \Ss_p$.
  It follows immediately from the definition that the base change of $\ta$ to $\C$ is $\tau$.
  Since the Betti realization of $\tau$ is $1$, we find that the underlying of the Betti realization of $\ta$ is $1$. 

  In order to show that the Betti realization of $\ta$ is $1$,
  it therefore suffices to show that it is multiplication by some $p$-adic integer.
  To do this, it suffices to do so modulo $p^k$ for all $k$.
  Modulo $p^k$, the Betti realization of $\ta$ arises as the stabilization of a map $S^1 \to S^1 / p^k$ of $C_2$-equivariant spaces, and all such maps are given by multiplication by a $p$-adic integer.
\end{proof}

\section{The ta-local category} \label{sec:homology}
In this section, we show that Betti realization identifies the category of $\ta$-local Artin-Tate $\R$-motivic spectra with the category of $C_2$-spectra. Indeed, we know from \Cref{thm:ta} that $\b(\ta) = 1$, so that $\b$ factors through the category of $\ta$-local objects. The main theorem of this section is that the induced functor is an equivalence.

\begin{dfn} \label{dfn:ta-inverted-betti}
  Let $\Ss_2[\ta^{-1}]$ denote the commutative algebra given by
  \[ \Ss_2 [\ta^{-1}] \coloneqq \colim\left( \Ss_2^{0,0,0} \stackrel{\ta}{\to} \Ss_2^{0,0,1} \stackrel{\ta}{\to} \Ss_2^{0,0,2} \stackrel{\ta}{\to} \cdots \right). \]
  The category of modules over $\Ss_2 [\ta^{-1}]$ in $\SHgc$ is equivalent to the catgoery of $\ta$-local objects in $\SHgc$.\footnote{Here we use the fact that one may invert an element in the Picard-graded homotopy of a commutative algebra and the associated description of the module category as a localization from the proof of \cite[Proposition 4.3.17]{ECII}.}
  Since the Betti realization of $\ta$ is $1$, Betti realization factors through $\ta$-localization providing a symmetric monoidal functor
  \[ \b : \Mod ( \SHgc ; \Ss_2 [\ta^{-1}] ) \to \Sp_{C_2, i2}. \]
\end{dfn}

\begin{thm} \label{thm:comparison-invert-ta}
  Betti realization induces an equivalence of symmetric monoidal categories
  \[ \b : \Mod ( \SHgc ; \Ss_2 [\ta^{-1}] ) \xrightarrow{\simeq} \Sp_{C_2, i2} \]
  with inverse equivalence given by $Y(-) \coloneqq c_{\C/\R}(-)[\ta^{-1}]$.
  Simply put, inverting $\ta$ in $\SHgc$ recovers $\Sp_{C_2, i2}$. 
\end{thm}

The main step in our proof of this theorem is showing that $Y$ is fully-faithful.
Since $c_{\C/\R}$ is fully-faithful, as proved by Heller and Ormsby (see \Cref{thm:c-ff}), this reduces to studying the interaction of $\ta$ and $c_{\C/\R}$. Our method of proof is to prove the appropriate statement at the level of $\HFt$ by direct computation, extend to dualizable objects by descent, then extend to all objects by colimits. This stategy requires knowledge of the tri-graded homology of a point and the strucutre of the tri-graded dual Steenrod algebra as an input. Since this is of independent interest we have separated it out as its own subsection. 

\subsection{The homology of a point}\ 

In this subsection we compute the tri-graded homology of a point and the structure of the dual Steenrod algebra. Although we only use coarse information extracted from these computations in this paper, we hope that they are useful to readers more interested in computations. The reader who has heard that $\R$-motivic computations are simplified by the absence of the ``negative cone'' may be surprised to learn that it is present in the tri-graded picture.

\begin{prop} \label{thm:pthomology}
As tri-graded commutative rings, with $|\ta|=(0,0,-1)$, there are isomorphisms
\[ \pi^{\R}_{p,q,w} \HZt \cong \left(\pi^{C_2}_{p+q\sigma}\uZt\right)[\ta] \quad \text{ and } \quad \pi^{\R}_{p,q,w} \HFt \cong \left(\pi^{C_2}_{p+q\sigma}\uFt\right)[\ta], \] 
where $\pi_{p+q\sigma} ^{C_2} \uZt$ and $\pi_{p+q\sigma} ^{C_2} \uFt$ are each placed in tri-degree $(p,q,0)$.
\end{prop}

Before we move on to the proof of \Cref{thm:pthomology},
we remind the reader of the structure of $\pi_{p+q\sigma} ^{C_2} \uZt$ and $\pi_{p+q\sigma} ^{C_2} \uFt$.
We also explain how to translate between the older names of elements in the bi-graded homology of a point and the names arising from this proposition.

\begin{rec} \label{thm:eqpthom} \todo{This needs a citation}
  The bi-graded homotopy groups of $\uFt$ are given by
  \[\pi_{p+q\sigma} ^{C_2} \uFt \cong \F_2 [a_\sigma, u_{\sigma}] \oplus \F_2 \left\{ \frac{\theta}{a_\sigma ^k u_{\sigma} ^{n}} \vert k,n \geq 0 \right\},\]
  where $|a_\sigma| = -\sigma$, $|u_{\sigma}| = 1-\sigma$, and $|\theta| = 2\sigma - 2$.
   The multiplicative structure is that of a trivial square-zero extension over $\F_2 [a_\sigma, u_\sigma]$: more informally, it is given by $\theta^2 = 0$.
  This is pictured in \Cref{F2-fig}.
  
  The bi-graded homotopy groups of $\uZt$ are given by
  \[\pi_{p+q\sigma} ^{C_2} \uZt \cong \Z_2 [a_\sigma,u_{2\sigma}]/(2a_\sigma) \oplus \Z_2 \left\{ \frac{2}{u_{2\sigma} ^{n}} \vert n \geq 1 \right\} \oplus \F_2 \left\{ \frac{\theta}{a_\sigma ^k u_{2\sigma} ^{n}} \vert k,n \geq 0 \right\}\]
  where $|a_\sigma| = -\sigma$, $|u_{2\sigma}| = 2-2\sigma$, and $|\theta| = 3\sigma - 3$.
  Again, the term involving $\theta$ sits in a trivial square-zero extension with rest of the ring.
  This is pictured in \Cref{uZ-fig}.
\end{rec}

\begin{figure}[t]
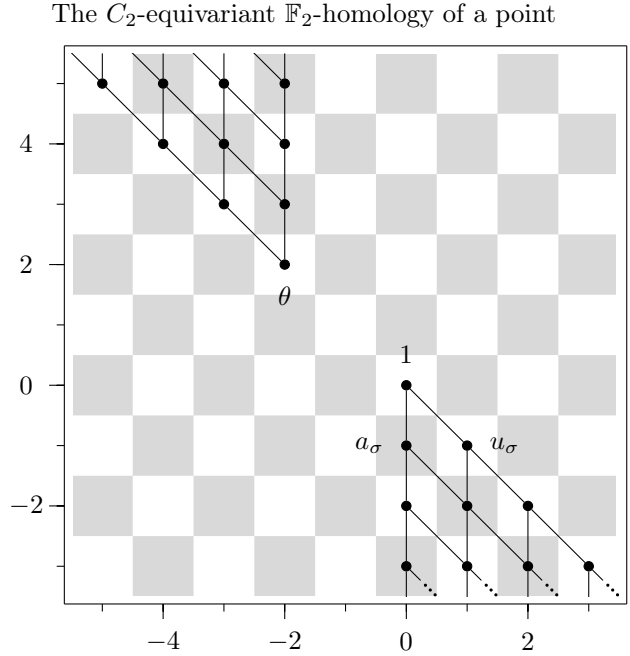

  \centering
  The $C_2$-equivariant $\F_2$-homology of a point \\\vspace{6pt}
  \begin{sseqpage}[ xscale=0.8, yscale=0.8, x range = {-5}{3}, y range = {-3}{5}, tick step = 2, major tick step=2, minor tick step= 1, class labels = {left}, classes = fill, grid = chess, axes type = frame ]
    \class["1" above](0,0)
    \class["a_\sigma"](0,-1) \structline
    \class(0,-2) \structline
    \class(0,-3) \structline
    \class(0,-4) \structline
    
    \class["u_\sigma" right](1,-1) \structline(0,0)
    \class(1,-2) \structline \structline(0,-1)
    \class(1,-3) \structline \structline(0,-2)
    \class(1,-4) \structline \structline(0,-3)
    
    \class(2,-2) \structline(1,-1)
    \class(2,-3) \structline \structline(1,-2)
    \class(2,-4) \structline \structline(1,-3)
    
    \class(3,-3) \structline(2,-2)
    \class(3,-4) \structline \structline(2,-3)
    
    \class(4,-4) \structline(3,-3)
    
    \class["\theta" below](-2,2)
    \class(-2,3) \structline
    \class(-2,4) \structline
    \class(-2,5) \structline
    \class(-2,6) \structline
    
    \class(-3,3) \structline(-2,2)
    \class(-3,4) \structline \structline(-2,3)
    \class(-3,5) \structline \structline(-2,4)
    \class(-3,6) \structline \structline(-2,5)
    
    \class(-4,4) \structline(-3,3)
    \class(-4,5) \structline \structline(-3,4)
    \class(-4,6) \structline \structline(-3,5)
    
    \class(-5,5) \structline(-4,4)
    \class(-5,6) \structline \structline(-4,5)
    
    \class(-6,6) \structline(-5,5)
    
  \end{sseqpage}
  \caption{ {\footnotesize
      The bigraded homotopy groups of $\underline{\F}_2$ depicted on a $(p,q)$ coordinate system.
      Vertical lines going down denote multiplication by $a_\sigma$.
      Diagonal lines going down-right denote multiplication by $u_{\sigma}$.}}
  \label{F2-fig}
\end{figure}

\begin{figure}[t]
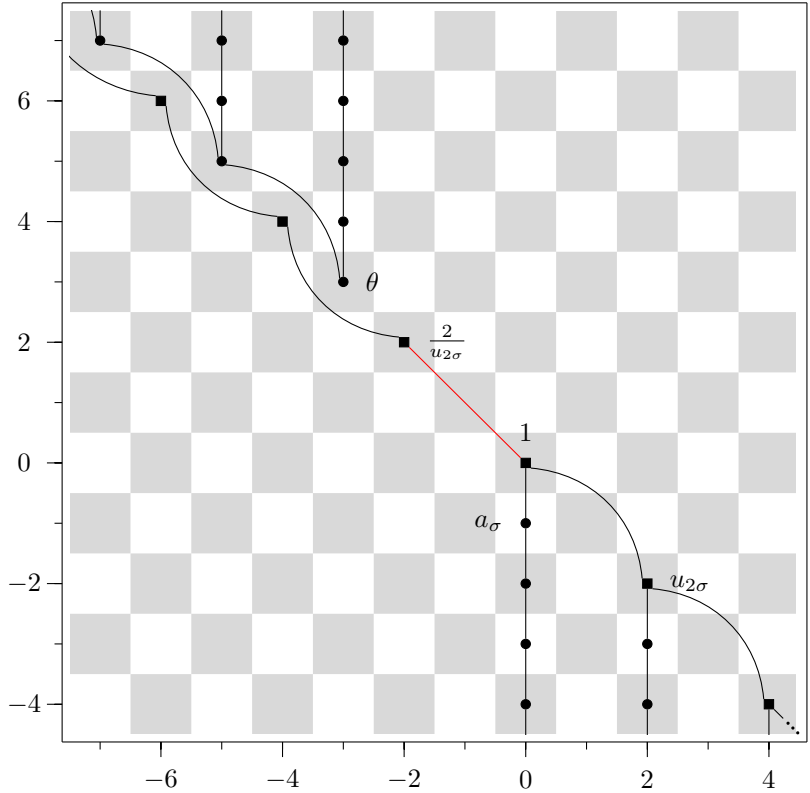

  \centering
  The $C_2$-equivariant integral homology of a point \\\vspace{6pt}
  \begin{sseqpage}[ xscale=0.8, yscale=0.8, x range = {-7}{4}, y range = {-4}{7}, tick step = 2, major tick step=2, minor tick step= 1, class labels = {left}, classes = fill, grid = chess, axes type = frame ]
    \class["1" above, rectangle](0,0)
    \class["a_\sigma"](0,-1) \structline
    \class(0,-2) \structline
    \class(0,-3) \structline
    \class(0,-4) \structline
    \class(0,-5) \structline
    
    \class["u_{2\sigma}" right, rectangle](2,-2) \structline[bend left = 40](0,0)(2,-2)
    \class(2,-3) \structline 
    \class(2,-4) \structline 
    \class(2,-5) \structline 

    \class[rectangle](4,-4) \structline[bend left = 40](4,-4)(2,-2)
    \class(4,-5) \structline
    \class[rectangle](6,-6) \structline(6,-6)(4,-4)

    \class["\frac{2}{u_{2\sigma}}" right, rectangle](-2,2) \structline[red](0,0)
    \class[rectangle](-4,4) \structline[bend right = 40](-4,4)(-2,2)
    \class[rectangle](-6,6) \structline[bend right = 40](-4,4)(-6,6)
    \class[rectangle](-8,8) \structline[bend right = 40](-8,8)(-6,6)
    
    \class["\theta" right](-3,3) 
    \class(-3,4) \structline 
    \class(-3,5) \structline 
    \class(-3,6) \structline
    \class(-3,7) \structline
    \class(-3,8) \structline 
    
    \class(-5,5) \structline[bend left = 40](-3,3)
    \class(-5,6) \structline
    \class(-5,7) \structline
    \class(-5,8) \structline 
    
    \class(-7,7) \structline[bend left = 40](-5,5)
    \class(-7,8) \structline
    
    \class(-9,9) \structline[bend left = 40](-7,7)  
    
  \end{sseqpage}
  \caption{ {\footnotesize
      The bigraded homotopy groups of $\uZt$ depicted on a $(p,q)$ coordinate system.
      Squares denote copies of $\Z_2$ and circles denotes copies of $\F_2$.
      Vertical lines going down denote multiplication by $a_\sigma$.
      Diagonal lines going down-right denote multiplication by $u_{2\sigma}$. 
      The red line indicates a $u_{2\sigma}$ multiplication that hits twice a generator.
      Not all $u_{2\sigma}$ multiplications are shown.}}
  \label{uZ-fig}
\end{figure}

\begin{rec} \label{rec:elements}
  Under Betti realization, $\HZt$ and $\HFt$ are sent to $\uZt$ and $\uFt$ \cite{HellerOrmsby}.
  The following are some commonly encountered homotopy elements and their Betti realizations:
  \begin{itemize}
  \item Stabilizing the inclusion of the fixed points $S^0 \to S^{\sigma}$, we obtain $a_\sigma \in \pi_{-\sigma} ^{C_2} \Ss$. This class maps to the corresponding class in $\uZt$ and $\uFt$. 
  \item We let $u \in \pi_{1,-1,0} ^\R \HFt$ denote the element corresponding to $u_\sigma \in \pi_{1-\sigma} ^{C_2} \uFt$ under the isomorphism of \Cref{thm:pthomology}.
  \item The class $\rho \in \pi_{0,-1,-1} ^{\R} \Ss$ is defined to be the stabilization of the inclusion $\{\pm 1\} \hookrightarrow \G_m$. Under Betti realization, $\rho$ goes to $a_\sigma$.
  \item The element $\tau \in \pi_{1,-1,-1}^\R \HFt$ corresponds to $-1$ under the 
    isomorphism $\pi_{1,-1,-1} ^\R \HFt \cong \mu_2 (\R)$. \todo{Is B-L needed for this one, or is this an easy case?} Under Betti realization, $\tau$ goes to $u_\sigma$.
  \end{itemize}
\end{rec}

\begin{cor}\label{cor:rels}
  There are relations $\rho = \ta \cdot a$ in $\pi_{0,-1,-1} ^\R \Ss_2$ and $\tau = \ta \cdot u \in \pi_{1,-1,-1} ^{\R} \HFt$.
\end{cor}

\begin{proof}
  In \Cref{prop:weak-one-sided}, we will show that Betti realization induces an isomorphism
  $\pi_{0,-1,-1} ^\R \Ss_2 \to \pi_{-\sigma}^{C_2} \Ss_2$,
  therefore for the first relation it suffices to note that $\b (\rho) = 1 \cdot a_\sigma = \b ( \ta \cdot a)$. Similarly, by \Cref{thm:pthomology} and \Cref{thm:eqpthom}, we know that $\pi_{1,-1,-1} ^\R \HFt \cong \F_2 \{\ta \cdot u\}$. Since $\b (\tau) = u_\sigma \neq 0$, the second relation follows.
\end{proof}

Assuming the tri-graded homology of a point, the tri-graded dual Steenrod algebra can easily by deduced from results of Voevodsky.

\begin{thm} \label{thm:Steenrod}
The tri-graded $\R$-motivic dual Steenrod algebra $\pi^{\R}_{*,*,*} \left(\HFt \otimes \HFt\right)$ is isomorphic to
$$\left(\HFt\right)_{*,*,*} [\tau_0,\tau_1,\cdots,\xi_1,\xi_2,\cdots]/(\tau_i^2 = \ta a \tau_{i+1} + \ta u \xi_{i+1} + \ta a \tau_0 \xi_{i+1}),$$
where $\abs{\tau_i}=(2^i,2^{i}-1,2^{i}-1)$ and $\abs{\xi_{i}}=(2^i-1,2^i-1,2^i-1)$.
\end{thm}

\begin{proof}
%
  We will deduce this from the classical bigraded computation of the $\R$-motivic Steenrod algebra. The key input is the fact that $\HFt \otimes \HFt$ decomposes as a direct sum of $\R$-motivic spectra of the form $\Sigma^{p,q,q} \HFt$. This is stated as \cite[Theorem 1.1 (3)]{HoyoisSteenrod}, which in the case of a characteristic zero base field such as $\R$ follows from work of Voevodsky \cite{Voe03, Voe10}.

  The decomposition of $\HFt \otimes \HFt$ given in \cite[Theorem 1.1]{HoyoisSteenrod} shows that, as a $\left(\HFt\right)_{*,*,*}$-module, $\pi_{*,*,*} ^{\R} (\HFt \otimes \HFt)$ is freely generated by monomials in the $\xi_i$ and $\tau_i$.  Since each $\tau_i$ and $\xi_i$ is represented by a map of the form $\Ss^{a} \otimes \G_m^{b} \to \HFt \otimes \HFt$ (with no copies of $\Ss^{\C}$ in the domain), we may read off the formulas for products of $\tau_i$ and $\xi_i$ from the standard bigraded computations \cite[Theorem 12.6]{Voe03}.
  To translate the formulas into tri-graded notation, we need only make use of the relations $\rho = \ta \cdot a$ and $\tau = \ta \cdot u$ of \Cref{cor:rels}.
\end{proof}

We now proceed to the proof of \Cref{thm:pthomology}.
The main steps of the proof are split across the next several lemmas, the key input being our knowledge of the bigraded $2$-complete motivic cohomology of both $\R$ and $\C$. 

\begin{lem} \label{lem:HomologyPointC}
  For integers $p,w \in \Z$, there are isomorphisms
  \[ \pi_{p,w,w}^{\R} \HZt \cong \Z_2[\tau^2, \rho] / (2\rho). \]
  For integers $p,q,w \in \Z$, there are isomorphisms
  \[ \pi^{\R}_{p,q,w} \left( \Spec(\C) \otimes \HZt \right) \cong \pi^{\C}_{p+q,w}(\HZt) \]
  and
  \[\pi^{\C} _{p,w} \HZt \cong \Z_2 [\tau].\]
\end{lem}

\begin{proof}
  The base change isomorphism is a corollary of the six functor formalism. \todo{Cite something from the intro here?}
  The computations of motivic cohomology are a corollary of Voevodsky's solution of the Bloch--Kato and Beilinson--Lichtenbaum conjectures \cite{Voe03b}. 
  %
%
\end{proof}

\begin{lem} \label{lem:Z-positive-w}
  The groups $\pi_{p,q,w}^{\R} \HZt$ and $\pi_{p,q,w}^{\R} \HFt$ are zero for $w > 0$.
\end{lem}

\begin{proof}
  Considering the cofiber sequence $\HZt \xrightarrow{2} \HZt \to \HFt$, we reduce to the case of $\HZt$.
  Using \Cref{lem:HomologyPointC}, we see that this lemma is true for tri-degrees of the form $(p,w,w)$.
  For each integer $n$, consider the cofiber sequence
$$\Spec(\C) \otimes (\bS^{\C})^{\otimes n-1} \to (\bS^{\C})^{\otimes n-1} \to (\bS^{\C})^{\otimes n}.$$
Tensoring this with $\HZt$ and applying $\pi_{p,w,w}$, we obtain a long exact sequence
$$\pi^{\C}_{p+w-n+1,w}  \HZt \to \pi^{\R}_{p,w-n+1,w} \HZt \to \pi^{\R}_{p,w-n,w} \HZt \to \pi^{\C}_{p+w-n,w} \HZt.$$

When $w>0$, \Cref{lem:HomologyPointC} thus ensures an isomorphism between $\pi^{\R}_{p,w-n+1,w} \HZt$ and $\pi^{\R}_{p,w-n,w} \HZt$.  As $n$ varies, using the base case $n=0$, we conclude that these groups always vanish.
\end{proof}

\begin{lem} \label{lem:Zcalc}
  When $w \le 0$, the Betti realization maps
  $ \pi^{\R}_{p,q,w} \HZt \to \pi^{C_2}_{p+q\sigma} \uZt $
  and 
  $ \pi^{\R}_{p,q,w} \HFt \to \pi^{C_2}_{p+q\sigma} \uFt $
  are isomorphisms.
\end{lem}

\begin{proof}
    Considering the cofiber sequence $\HZt \xrightarrow{2} \HZt \to \HFt$, we reduce to the case of $\HZt$. We first check that this true when $w=q$.
  \Cref{thm:eqpthom}, \Cref{rec:elements} and \Cref{lem:HomologyPointC} imply that the Betti realization map is an isomorphism for degrees of the form $(p,w,w)$ with $ w \leq 0$.
  As in the previous lemma we tensor the cofiber sequence
  \[ \Spec(\C) \otimes (\bS^{\C})^{\otimes n-1} \to (\bS^{\C})^{\otimes n-1} \to (\bS^{\C})^{\otimes n} \]
  with $\HZt$ and take homotopy groups in order to spread out to other cases. 
  Using $\beta$ as notation for Betti realization we obtain maps of of exact sequences of abelian groups
  \begin{center}
    \begin{tikzcd}
      \pi^{\C}_{p+w-n+1,w}  \HZt \arrow{r} \arrow{d}{\cong} & \pi^{\R}_{p,w-n+1,w} \HZt \arrow{r} \arrow{d}{\beta_{p,n-1}} & \pi^{\R}_{p,w-n,w} \HZt \arrow{r} \arrow{d}{\beta_{p,n}} & \pi^{\C}_{p+w-n,w} \HZt \arrow{d}{\cong} \\
    \pi_{p+w-n+1} \Zt \arrow{r} & \pi^{C_2}_{p+(w-n+1)\sigma} \uZt \arrow{r} & \pi^{C_2}_{p+(w-n)\sigma} \uZt \arrow{r} & \pi_{p+w-n} \Zt,
    \end{tikzcd}
  \end{center}
  with our goal being to prove that the Betti realization maps $\beta_{p,n}$ are isomorphisms for all $p,n \in \Z$.

  Using the 5-lemma we conclude that ($\beta_{p,n-1}$ is iso) + ($\beta_{p-1,n-1}$ is iso) implies ($\beta_{p,n}$ is iso). Similarly, the 5-lemma also implies that ($\beta_{p,n}$ is iso) + ($\beta_{p+1,n}$ is iso) implies ($\beta_{p,n-1}$ is iso). As noted above, we have already learned that $\beta_{p,0}$ is an isomorphism for all $p$, so we may now induct outwards from this case to conclude.
\end{proof}



\begin{proof}[Proof of \Cref{thm:pthomology}]
By \Cref{lem:Zcalc}, the map
\[\pi^\R _{p,q,0} \HZt \to \pi^{C_2}_{p+q\sigma}\uZt,\]
induced by Betti realization is an isomorphism.
Taking the inverse, we obtain a ring map
\[\pi^{C_2}_{p+q\sigma}\uZt \to \pi^\R _{p,q,0} \HZt.\]
This extends to a ring map
\[\left(\pi^{C_2}_{p+q\sigma} \uZt \right) [\ta] \to \pi_{p,q,*} \HZt.\]
  It follows from \Cref{thm:ta} and Lemmas \ref{lem:Z-positive-w} and \ref{lem:Zcalc} that this map is an isomorphism.
The same proof works for $\HFt$.
%
%
%
%
\end{proof}

\subsection{The proof of \Cref{thm:comparison-invert-ta}} \label{subsec:ccr}\ 

Recall that we are proving that $\b$ is an equivalence on the $\ta$-local category with inverse $Y$.
We will check directly that $Y$ is an equivalence of categories by showing that it is fully faithful and essentially surjective. That $\b$ is the inverse of $Y$ follows from the fact that $\b \circ Y = \b \circ c_{\C/\R}$ is the identity on $\Sp_{C_2,i2}$. Before we can make progress on this goal we will need a pair of lemmas.

\begin{lem} \label{lem:Cta-neg-w}
  The tri-graded homotopy of $C\ta$ is zero in negative weights.
\end{lem}

\begin{proof}
  Our method of proof will be to apply the motivic Adams spectral sequence to $C\ta$.
  Since this spectral sequence converges strongly for $\Ss_2$ by \cite{MAdamsConv}, it also converges strongly for $C\ta$. This spectral sequence takes the form
  $$E_1^{s,t} = \pi_{t,q,w}(C\ta \otimes \HFt^{\otimes {s+1}}) \implies \pi_{t-s,q,w} C\ta.$$
  It suffices therefore to check at the level of the $E_1$-page that $\pi_{t,q,w} \HFt^{\otimes {s+1}} = 0$ for $w<0$.
	
  For this, we use the known \cite{Voe03,Voe10,HoyoisSteenrod} description of $\HFt \otimes \HFt$ in $\SHR$.  One has that $\HFt \otimes \HFt \simeq \oplus_{(x_i,y_i,y_i)} \Sigma^{x_i,y_i,y_i} \HFt$, where the $x_i$ and $y_i$ range over non-negative integers.  The result now follows immediately from \Cref{thm:pthomology}.
\end{proof}

\begin{lem} \label{prop:weak-one-sided}
  Let $n$ denote a non-negative integer.  Given a pair of $C_2$-spectra $X, Y \in \Sp_{C_2, i2}$, the natural map induced by $\ta^n$
  \[ \Map_{\SHRat}(c_{\C/\R} X, c_{\C/\R} Y) \to \Map_{\SHRat}(c_{\C/\R} X, \Ss_2^{0,0,n} \otimes c_{\C/\R} Y) \]
  is an equivalence.
\end{lem}

\begin{proof}
  Since $c_{\C/\R} $ commutes with colimits it will suffice to prove the proposition as $X$ ranges through a family of compact generators of $\Sp_{C_2, i2}$. In particular, it will suffice to assume $X \simeq \Ss_2^{p+q\sigma}$ is a representation sphere.  Since $c_{\C/\R} X \simeq \Ss^{p,q,0} _2$ is compact, it furthermore suffices to prove the proposition as $Y$ ranges through a family of compact generators. In particular, it suffices to assume that $Y \simeq \Ss^{a+b\sigma}_2$ is also a representation sphere. 
	
  At this point, the proposition reduces to a claim about 
  $$\Map_{\SHRat}(c_{\C/\R} \Ss_2^{p+q\sigma}, \Ss_2^{0,0,n} \otimes c_{\C/\R} \Ss_2^{a+b\sigma}) \simeq \Map_{\SHRat}(\Ss_2^{p-a,q-b,-n},\Ss_2),$$
  or in other words a claim about the tri-graded stable stems $\pi^{\R}_{*,*,*} \Ss_2$.  We must show that multiplication by $\ta^n$ is an isomorphism $\pi^{\R}_{*,*,0} \Ss_2 \to \pi^{\R}_{*,*,-n} \Ss_2$.  Equivalently, we must check that $\ta:\pi^{\R}_{*,*,-n} \Ss_2 \to \pi^{\R}_{*,*,-n-1} \Ss_2$ is an isomorphism for each $n \ge 0$.  Examining the cofiber sequence
  \[\Sigma^{-1,0,0} C\ta \to \Ss_2^{0,0,-1} \stackrel{\ta}{\to} \Ss_2 \to C\ta, \]
  we may use \Cref{lem:Cta-neg-w} to conclude.
\end{proof}


We are now ready to complete the proof of \Cref{thm:comparison-invert-ta}.

  To check that $Y$ is fully faithful, we must prove that for any pair $A, B \in \Sp_{C_2,i2}$ the composite
\[ \Map_{\Sp_{C_2,i2}}(A,B) \to \Map_{\SHRat}(c_{\C/\R}A, c_{\C/\R}B) \to \Map_{\SHRat}(c_{\C/\R}A [\ta^{-1}], c_{\C/\R}B [\ta^{-1}]) \]
is an equivalence. The first map is an equivalence as a consequence of the fully-faithfulness of $c_{\C/\R}$ as proven by Heller-Ormsby, see \Cref{thm:c-ff}. The second map factors as
\begin{align*}
  \Map_{\SHRat}(c_{\C/\R}A, c_{\C/\R}B)
  &\xrightarrow{\simeq} \Map_{\SHRat}(c_{\C/\R}A, c_{\C/\R}B [\ta^{-1}]) \\
  &\xrightarrow{\simeq} \Map_{\SHRat}(c_{\C/\R}A [\ta^{-1}], c_{\C/\R}B [\ta^{-1}])
\end{align*}
where first map is an equivalence as a consequence of \Cref{prop:weak-one-sided}
and the second map is an equivalence since inverting $\ta$ is a localization.

Since $Y$ is now fully-faithful and colimit preserving, to check that it is essentially surjective it suffices to check that its image contains a family of compact generators. One such family consists of the objects $\Ss^{p,q,w}_2[\ta^{-1}]$ as $p,q,w$ range over the integers, since $\Ss_2^{p,q,w}[\ta^{-1}] \simeq \Ss_2^{p,q,0}[\ta^{-1}] \simeq Y(\Ss_2^{p+q\sigma})$.

\section{Galois reconstruction} \label{sec:top}
In this section, we provide a Galois reconstruction of $\SHRat$. In other words, we show how to reconstruct $\SHRat$ from $C_2$-equivariant homotopy theory.
As in the case of $\C$, understanding the close connection between $\MGL$ and its Betti realization is the essential step in reconstruction.
In \Cref{app:def}, we have set up a general framework for reconstruction results of this kind. We will rely heavily on the work there, so we suggest the reader familiarize themself with the material and notations therein before proceeding.

Before we state the main theorem, let us recall the definition of the regular slice filtration of $C_2$-spectra.

\begin{dfn}
  We say that a $C_2$-spectrum $X$ is regular slice $n$-connective if $\Phi^e X$ is $n$-connective and $\Phi^{C_2} X$ is $\lceil \frac{n}{2} \rceil$-connective.\footnote{Here we work with the regular slice filtration because of its good multiplicative properties.}

  We let $\Sp_{C_2} ^{\geq n}$ denote the full subcategory of $\Sp_{C_2}$ consisting of the regular slice $n$-connective $C_2$-spectra. This is a coreflective subcategory, and we let $P_{n} : \Sp_{C_2} \to \Sp_{C_2} ^{\geq n}$ denote the right adjoint to the inclusion. We let $P^n$ denote the $n^{\mathrm{th}}$ regular slice truncation functor and let $P^n _n$ denote the $n^{\mathrm{th}}$ regular slice functor.
\end{dfn}

\begin{cnstr}
  Since the functors $\Phi^e$ and $\Phi^{C_2}$ are monoidal, the hypotheses of \Cref{cnstr:trunc-tower} are satisfied and we may assemble the categories $\Sp_{C_2}^{\geq n}$ into a coreflective symmetric monoidal subcategory $\Sp_{C_2, \geq 0}^{\Fil, 2\mathrm{slice}} \subset \Sp_{C_2}^{\Fil}$ which consists of those $X_\bullet$ which are regular slice $2n$-connective in position $n$. This provides a lax symmetric monoidal connective cover functor $\tau_{\geq 0}^{2\mathrm{slice}}$ which takes the $(2n)^{\mathrm{th}}$-slice cover at position $n$. The composition $\tau_{\geq 0}^{2\mathrm{slice}}(Y(-))$ is equivalent to even slice tower functor $P_{2\bullet}$ and demonstrates that this functor is lax symmetric monoidal.
\end{cnstr}


The fundamental construction of the section is the following commutative algebra in filtered $C_2$-spectra:
\[ R_\bullet \coloneqq \mathrm{Tot}^* \left( P_{2\bullet} \MU_{\R,2}^{\otimes * + 1} \right). \]

\begin{thm}[Galois reconstruction] \label{thm:filt-model}
  There is an equivalence of presentably symmetric monoidal categories under $\Sp_{C_2, i2} ^{\Fil}$,
  \[\SHRat \simeq \Mod(\Sp^{\Fil}_{C_2, i2}; R_\bullet),\]
  where $\Sp_{C_2, i2}$ acts on the left through $c_{\C/\R}$.
\end{thm}

The commutative algebra $R_\bullet$ can be called the ``d\'ecalage of the $\MU_{\R,2}$-Adams tower with respect to the even slice filtration''. The commutative algebra $R_\bullet$ is also the image of the unit in $\Sp_{C_2,i2}$ under a certain lax symmetric monoidal functor.

\begin{cnstr} \label{cnstr:gamma-star}
  Applying \Cref{cnstr:shear}, we obtain a lax symmetric monoidal functor
  \[ \mathrm{Sh}(P_{2\bullet} ; \MU_{\R,2}) : \Sp_{C_2,i2} \to \Sp^{\Fil}_{C_2, i2}. \]
  By construction, $R_\bullet \simeq \mathrm{Sh}(P_{2\bullet} ; \MU_{\R,2})(\Ss_2)$, so this functor factors through the category of modules over $R_\bullet$. Composing with the equivalence of \Cref{thm:filt-model}, this defines a lax symmetric monoidal functor,
  \[ \Gamma_* :  \Sp_{C_2, i2} \to \SHRat. \]  
\end{cnstr}

This functor is analogous to the functor $\Gamma_* : \Sp_{2} \to \SHC_{2}$ studied in \cite{cmmf}, and may also be compared with the synthetic analogue functor of Pstragowski \cite{Pstragowski}. We will study this functor more closely in \Cref{sec:t-structures}.

The proof of \Cref{thm:filt-model} will be carried out in two steps: first we will prove that there is an equivalence
  \[\SHRat \simeq \Mod(\Sp^{\Fil}_{C_2, i2}; i_* (\Ss_2))\]
for some lax symmetric monoidal functor $i_* : \SHRat \to \Sp^{\Fil} _{C_2, i2}$.
We will then construct an equivalence $i_* (\Ss_2) \simeq R_\bullet$ of commutative rings in $\Sp^{\Fil} _{C_2, i2}$.
The key step in the construction of this equivalence is the identification of $i_* (\MGL_2)$ as a commutative algebra in $\Sp^{\Fil} _{C_2, i2}$.
In order make this identification, we will show that $i_*$ admits a description in terms of Voevodsky's effective slice filtration.



\subsection{The filtered model}\

In this subsection we prove the first half of Galois reconstruction,
namely that there is a filtered model for $\SHRat$.

\begin{prop} \label{prop:top-diagram}
  There is a diagram of symmetric monoidal left adjoints
  \begin{center}
    \begin{tikzcd}
      {\Sp_{C_2,i2}} \ar[r] \ar[d, "\mathrm{Id}"] &
      {\Sp_{C_2,i2}^{\Fil}} \ar[r, " - \otimes i_*\Ss"] \ar[dr, "i^*"] &
      {\Mod ( \Sp_{C_2,i2}^{\Fil} ; i_*\Ss_2)} \ar[r, "\Re"] \ar[d, "\simeq"] &
      {\Sp_{C_2,i2}} \ar[d, "\mathrm{Id}"] \\
      {\Sp_{C_2,i2}} \ar[rr, "c_{\C/\R}"] & &
      \SHRat \ar[r, "\b"] &
      {\Sp_{C_2,i2}}
    \end{tikzcd}
  \end{center}
  such that $i^*(\Ss_2^{p+q\sigma}(w)) \simeq \Ss_2^{p,q,w}$.
\end{prop}

Note that \Cref{prop:filt-mod} produces a diagram of this type,
so in order to prove the proposition we only need to endow $\SHRat$ and $\Sp_{C_2,i2}$ with the structure of a deformation pair in the sense of \Cref{dfn:def-pair}.

\begin{proof}
  We begin with the diagram
  \begin{center} \begin{tikzcd}
      & \SHRat \ar[dr, "\b"] & \\
      {\Sp_{C_2,i2}} \ar[ur, "c_{\C/\R}"] \ar[rr, "\mathrm{Id}"] & &
      {\Sp_{C_2,i2}}.
  \end{tikzcd} \end{center}
  In order to make $i^*$ behave as desired on Picard elements, we pick
  $i_0(w) = \Ss_2^{0,0,w}$. Since $\b (\Ss_2^{0,0,w}) \simeq \Ss_2^{0}$, this factors through the kernel of the map on Picard groups induced by $\b$.

  To conclude, we now need to verify the two conditions in the definition of a deformation pair.
  The first condition is implied by \Cref{prop:weak-one-sided}.  
  To verify the second condition, we note that the representation spheres $\Ss^{p+q\sigma} _2$ form a set of compact dualizable generators for $\Sp_{C_2,i2}$ and the tri-graded spheres $\Ss^{p,q,w} _2$ form a set of compact dualizable generators for $\SHRat$.
\end{proof}

\begin{rmk}
  Unraveling the definitions in \Cref{app:def}, we find that for $X \in \SHRat$ there is a natural identification
  \[ i_* (X)_n \simeq \Map_{\SHRat} (\Ss_2 ^{0,0,n}, X), \]
  and that the natural maps
  \[ \Map_{\SHRat} (\Ss_2 ^{0,0,n}, X) \simeq i_* (X)_n \to i_* (X)_{n-1} \simeq \Map_{\SHRat} (\Ss_2 ^{0,0,n-1}, X) \]
  are induced by $\ta : \Ss^{0,0,n-1} _2 \to \Ss^{0,0,n} _2$.
\end{rmk}

The task of identifying $i_*\Ss_2$ with $R_\bullet$ will occupy us for the remainder of the section.

\subsection{The effective slice filtration} \label{sec:eff}\ 

%


In this section, we relate the functor $i_*$ defined in the previous subsection to Voevodsky's effective slice filtration.
We begin by recalling the definition of the effective slice filtration \cite{VoeOpenProbs}.

\begin{dfn} \label{dfn:eff-slice}
  Let $\Sm / \R$ denote the category of smooth and separated $\R$-schemes of finite type.
  We let $\SHR^{\eff} _{i2, \geq n} \subset \SHR_{i2}$ denote the full subcategory generated under small colimits by the collection $\{\Ss_2^{p,q,q} \otimes X_+ \vert X \in \Sm / \R, p,q \in \Z, q \geq n\}$.
  We denote the right adjoint of the inclusion by $f_n : \SHR_{i2} \to \SHR^{\eff} _{i2, \geq n}$. There are natural transformations $f_{n+1} \to f_{n}$ and we let $s_n$ denote the cofiber of this map.

  Since $\SHR^{\eff} _{i2, \geq n} \subset \SHR_{i2}$ is a compactly generated stable subcategory, the functors $f_n$ and $s_n$ preserve all colimits. Moreover, the tensor product of an $n$-effective object with an $m$-effective object is $(m+n)$-effective, so that the effective slice tower functor $f_\bullet : \SHR_{i2} \to \SHR_{i2} ^{\Fil}$ is lax symmetric monoidal.
\end{dfn}

The main result of this section is the following:

\begin{prop} \label{thm:eff-fil}
  The lax symmetric monoidal functor $i_* : \SHRat \to \Sp_{C_2, i2}^{\Fil}$ of the previous subsection is equivalent to the lax symmetric monoidal functor $\b \circ f_\bullet : \SHRat \to \Sp_{C_2, i2}^{\Fil}$ taking $E \in \SHRat$ to the Betti realization of its effective slice tower
  \[\dots \to \b (f_{2} E) \to \b (f_1 E) \to \b (f_0 E) \to \b (f_{-1} E) \to \b (f_{-2} E) \to \dots.\]
\end{prop}
%
%
%

As a first step, we rephrase \Cref{dfn:eff-slice} to be more natural in our trigraded context:

\begin{lem}
  The full subcategory $\SHR^{\eff} _{i2, \geq n} \subset \SHR_{i2}$ is generated under small colimits by the collection $\{\Ss_2 ^{p,q,w} \otimes X_+ \vert X \in \Sm / \R, p,q,w \in \Z, w \geq n\}$.
  As a consequence, the suspension functors provide equivalences,
  \[ \Sigma^{p,q,w} : \SHR^{\eff} _{i2, \geq n} \to \SHR^{\eff} _{i2, \geq n+w}. \]
\end{lem}

\begin{proof}
  Since $\Sigma^{0,1,1} : \SHR^{\eff} _{i2, \geq n} \to \SHR^{\eff} _{i2, \geq n+1}$ is clearly an equivalence of categories, it suffices to prove the generation statement in the case that $n = 0$.
  Since $\SHR^{\eff} _{i2, \geq 0}$ is closed under tensor products, it suffices to show that $\Ss_2^{0,-1,0} \in \SHR^{\eff} _{i2, \geq 0}$.

  Now, $\SHR^{\eff} _{i2, \geq 0}$ is clearly closed under $\Sigma^{-1,0,0}$, so it suffices to show that $\Ss_2 ^{1,-1,0} \in \SHR^{\eff} _{i2, \geq 0}$.
  This follows from the existence of a cofiber sequence
  \[ \Ss^{0,0,0} \to \Spec \C \to \Ss^{1,-1,0}. \]
  The second statement is a clear consequence of the generation statement.
\end{proof}

We now define an alternative filtration on $\SHRat$ which is easier to analyze; following an argument of Heard \cite{SliceComp}, itself an adaptation of an argument of Pelaez \cite{Pelaez}, we will prove that this filtration is in fact equivalent to the effective slice filtration.

\begin{dfn}
  We let $\SHR^{\at, \gceff} _{i2, \geq n}$ denote the full subcategory generated under small colimits by the collection $\{\Ss_{2} ^{p,q,w} \vert p,q,w \in \Z, w \geq n\}$.
  We let
  \[ f_n ^{\att} : \SHRat \to \SHR^{\at, \gceff}_{i2, \geq n} \]
  denote the right adjoint of the inclusion. There are natural transformations $f_{n+1} ^{\att} \to f_{n} ^{\att}$, and we denote the cofiber by $s_n ^{\att}$.

  Since $\SHR^{\at, \gceff}_{i2, \geq n} \subset \SHRat$ is a compactly generated stable subcategory, the functors $f_n ^{\att}$ and $s_n ^{\att}$ preserve colimits.
\end{dfn}

It is clear that $\SHR^{\at,\gceff} _{i2, \geq n} \subset \SHRat \cap \SHR^{\eff} _{i2, \geq n}$.
%

\begin{lem} \label{lem:susp-slice}
  Given $E \in \SHR_{i2}$, there are natural equivalences $f_k \Sigma^{p,q,w} E \simeq \Sigma^{p,q,w} f_{k-w} E$ and $s_k \Sigma^{p,q,w} E \simeq \Sigma^{p,q,w} s_{k-w} E$.
  If $E \in \SHRat$, the analagous fact holds for $f_k ^{\att}$ and $s_k ^{\att}$.
\end{lem}

\begin{proof}
  This follows directly from the fact that $\Sigma^{p,q,w} : \SHR^{\eff} _{i2, \geq k} \to \SHR^{\eff}_{i2, \geq k +w}$ and $\Sigma^{p,q,w} : \SHR^{\at, \gceff} _{i2, \geq k} \to \SHR^{\at, \gceff}_{i2, \geq k +w}$ are equivalences of categories.
\end{proof}

\begin{lem} \label{prop:eff-slice-at-compare}
  Given  $E \in \SHRat$, there are natural equivalences $f_n E \simeq f_n ^{\att} E$ for all $n \in \Z$.
\end{lem}

\begin{proof}
  It is easy to see that the categories $\SHR^{\eff} _{i2, \geq n} \subset \SHR_{i2}$ and $\SHR^{\at, \gceff} _{i2, \geq n} \subset \SHRat$ define slice filtrations in the sense of \cite[Definition 2.1]{SliceComp}.
  The result will therefore follow from \cite[Theorem 2.20]{SliceComp} if we can verify three conditions. Let $\iota : \SHRat \hookrightarrow \SHR_{i2}$ denote the inclusion. Then we must show that the following hold for all $E \in \SHRat$:
  \begin{enumerate}
    \item The natural map $\varinjlim_{n} \iota(f_n ^{\att} E) \to \iota(\varinjlim_{n} f_n ^{\att} E))$ is an equivalence.
    \item $\iota(f_n ^{\att} E) \in \SHR^{\eff} _{i2, \geq n}$.
    \item $\Map_{\SHR_{i2} } (X, \iota(s_n ^{\att} E)) \simeq 0$ for all $X \in \SHR^{\eff} _{i2, \geq n+1}$.
  \end{enumerate}
  Condition (1) is clear from the fact that $\SHRat$ is closed under colimits in $\SHR_{i2}$, and condition (2) follows from the fact that $\SHR^{\at, \gceff} _{i2, \geq n} \subset \SHR^{\eff} _{i2, \geq n}$.

  To prove condition (3), we note that, since $s_n ^{\att}$ commutes with filtered colimits and $\SHR^{\eff} _{i2, \geq n+1}$ is compactly generated, it suffices to prove the statement for generators of $\SHR^{\at, \gceff} _{i2, \geq n}$, namely the trigraded spheres $\Ss_2 ^{p,q,w}$ where $w \geq n$.
  Since $s_n ^{\att} \Ss_2^{p,q,w} \simeq \Sigma^{p,q,w} s_{n-w} ^{\att} \Ss_2^{0,0,0}$ by \Cref{lem:susp-slice}, it suffices to show this for $\Ss_2^{0,0,0}$.

  This follows from the equivalence $s_{n} ^{\att} \Ss_2 ^{0,0,0} \simeq s_{n} \Ss_2^{0,0,0}$, which may be proved exactly as in \cite[Theorem 3.15]{SliceComp}.
\end{proof}

We are now free to use $f_n$ and $f^{\att} _n$ interchangeably. The following proposition gives us the needed control over $f^{\att} _n$:

\begin{prop} \label{prop:eff-condition}
  Let $E \in \SHRat$. Then $E \in \SHR^{\at, \gceff} _{i2, \geq n}$ if and only if
  \[ \pi^{\R} _{p,q,w} (C\ta \otimes E) = 0 \]
  for all $w < n$, i.e. if and only if $\ta : \pi^{\R} _{p,q,w+1} E \to \pi^{\R} _{p,q,w} E$ is an isomorphism for all $w < n$.
\end{prop}

\begin{proof}
  To show that $\pi^{\R} _{p,q,w} (C\ta \otimes E) = 0$ for all $w < n$ if $E \in \SHR^{\at, \gceff} _{i2,\geq n}$, it suffices to prove this when $E  = \Ss^{p,q,w}$ for $w \geq n$, which follows from \Cref{prop:weak-one-sided}.

  On the other hand, suppose that $E \in \SHRat$ satisfies $\pi^{\R} _{p,q,w} (C\ta \otimes E) = 0$ for all $w < n$.
  We will show that $f_n ^{\att} E \to E$ is an equivalence.
  First, we note that it is an equivalence on $\pi^{\R} _{p,q,w}$ for all $w \geq n$ by definition.
  By assumption, $\ta : \pi^{\R}_{p,q,w+1} E \to \pi^{\R} _{p,q,w} E$ is an isomorphism for all $w < n$.
  On the other hand, by the above $\ta : \pi^{\R}_{p,q,w+1} (f_n ^{\att} E) \to \pi^{\R} _{p,q,w}  (f_n ^{\att} E)$ is also an isomorphism for all $w < n$.
  This implies that $f_n ^{\att} E \to E$ in fact induces an isomorphism on all $\pi_{p,q,w} ^{\R}$, as desired.
\end{proof}
%
%

Finally, we are ready to prove \Cref{thm:eff-fil}.

\begin{proof}[Proof of \Cref{thm:eff-fil}]
  Given a bifiltered object $X_{\bullet, *}$, we let $\diag(X_{\bullet, *}) = X_{\bullet, \bullet}$ denote the filtered object obtained by restricting along the diagonal map $\Z^{\Fil} \hookrightarrow \Z^{\Fil} \times \Z^{\Fil}$.
  Then there is a span of lax symmetric monoidal functors:
  \begin{center}
    \begin{tikzcd}
      \diag \circ i_* \circ f_\bullet \ar[r] \ar[d] & i_* \\
      \b \circ f_\bullet, &
    \end{tikzcd}
  \end{center}
  where the horizontal map is induced by the natural transformation $f_\bullet \to Y$ and the vertical map is induced by the natural transformation $i_* \to Y \circ \b$. (Here as in \Cref{app:fil}, $Y$ is the functor taking an object to its constant filtered object.)

  Applied to $E \in \SHRat$, in filtration $n$ this span looks like
  \begin{center}
    \begin{tikzcd}
      \Map_{\SHRat} (\Ss_2 ^{0,0,n}, f_n E) \ar[r] \ar[d] & \Map_{\SHRat} (\Ss_2 ^{0,0,n}, E) \\
      \b (f_n E). &
    \end{tikzcd}
  \end{center}
  The horizontal map is an equivalence by $n$-effectivity of $\Ss_2 ^{0,0,n}$, and the vertical map is an equivalence by \Cref{thm:comparison-invert-ta} and \Cref{prop:eff-condition}.
\end{proof}
%
%
%

\subsection{Identification of $i_* \MGL_2$}\ 

In this subsection, we will identify the commutative algebra in filtered $C_2$-spectra given by $i_* \MGL_2$.
This requires two main inputs.
The first is the description of the underlying filtered object in terms of the effective slice filtration from the previous subsection.
The second is a theorem of Heard which relates the effective slice filtration of $\MGL$ to the regular slice filtration of its Betti realization $\MU_\R$.

\begin{prop} \label{prop:P-MGL}
  There is an equivalence of commutative algebras in filtered $C_2$-spectra
  \[ i_* \MGL_2  \simeq P_{2\bullet} \MU_{\R,2}. \]
\end{prop}     

We prove this as a special case of a slightly stronger result where we allow tensor powers of $\MGL_2$. This generalization will be useful in the sequel.

\begin{prop} \label{prop:P-MGL-power}
  The objects $i_* \MGL_2^{\otimes k}$ are regular slice $2n$-connective in position $n$.
  This yields a natural factorization of the map to the constant filtered object
  \begin{center} \begin{tikzcd}
      & P_{2\bullet} \MU_{\R,2}^{\otimes k} \ar[dr] & \\
      i_* \MGL_2^{\otimes k} \ar[ur, "\simeq"] \ar[rr] & & Y \MU_{\R,2}^{\otimes k},
  \end{tikzcd} \end{center}
  where the indicated map
  is an equivalence of commutative algebras in filtered $C_2$-spectra.
\end{prop}     



The result of Heard that we will need is the following:

\begin{thm}[\cite{SliceComp}] \label{thm:slice-comp}
  Under Betti realization, the slice tower of $\MGL ^{\otimes k}$ goes to the even part of the regular slice tower of $\MUR^{\otimes k}$.
  More precisely, there is a commutative diagram
  \begin{center} \begin{tikzcd}[sep=small]
      \cdots \ar[r] &
      \b (f_{2} \MGL^{\otimes k}) \ar[r] \ar[d, "\simeq"] &
      \b (f_1 \MGL^{\otimes k}) \ar[r] \ar[d, "\simeq"] &
      \b (f_0 \MGL^{\otimes k}) \ar[r] \ar[d, "\simeq"] &
      \cdots \b (\MGL^{\otimes k}) \ar[d, "\simeq"] \\
      \cdots \ar[r] &
      P_{4}  (\MU_{\R}^{\otimes k}) \ar[r] &
      P_2  (\MU_{\R}^{\otimes k}) \ar[r] &
      P_0  (\MU_{\R}^{\otimes k}) \ar[r] &
      \cdots \MU_\R^{\otimes k} .
  \end{tikzcd} \end{center}  
  Moreover, the odd regular slices of $\MUR$ vanish.
\end{thm}

Applying \Cref{thm:complete-compare}, we are able to deduce the $2$-completed analogue of 
\Cref{thm:slice-comp}. \todo{ehhh}

\begin{proof}
  For notational brevity we set $E \coloneqq \MGL_2^{\otimes k}$.
  Note that it is a consequence of \Cref{thm:complete-compare} that $\b (E) $ is equivalent as a commutative algebra to $\MU_{\R,2}^{\otimes n}$.
  Using \Cref{thm:eff-fil}, we can conclude that $i_*E$ is given by the filtered $C_2$-spectrum, 
  \[\cdots \to \b (f_{2} E) \to \b (f_1 E) \to \b (f_0 E) \to \b (f_{-1} E) \to \b (f_{-2} E) \to \cdots.\]
  By \Cref{thm:slice-comp}, this is equivalent to the tower
  \[ \cdot \to P_{4} \b (E) \to P_2 \b (E) \to P_0 \b (E) \to P_{-2} \b (E) \to P_{-4} \b (E) \to \cdots.\]
  Now, consider the natural map of commutative algebras $i_*E \to Y (\b (E))$.
  Looking at the explicit description of $i_*E$ we can conclude it lies in the coreflective subcategory $\Sp_{C_2, \geq 0}^{\Fil, 2\mathrm{slice}}$. Thus we obtain a diagram of commutative algebras
 \begin{center} \begin{tikzcd}
      & P_{2\bullet} (\b (E)) \ar[dr] & \\
      i_*E \ar[ur, "\simeq"] \ar[rr] & &  Y(\b (E)),
  \end{tikzcd} \end{center}
  where the first map is an equivalence by the above.
\end{proof}

\subsection{Identification of $i_* (\Ss_2)$}\

In order to finish the proof of \Cref{thm:filt-model}, we need to prove the following proposition:

\begin{prop}
  There is an equivalence of commutative algebras in $\Sp_{C_2,i2}^{\Fil}$ between $i_*\Ss_2$ and $R_\bullet$.
\end{prop}


\begin{proof}
  Recall that $\mathrm{cb}$ is the functor which sends a commutative algebra $A$ to the cosimplicial commutative algebra given by its cobar complex $A^{\otimes(\ast+1)}$. Consider the natural map
  \[ \mathrm{cb}(\MGL_2) \to \mathrm{cb}(\MGL_2[\ta^{-1}]) \]
  in $\SHRat$. After applying $i_*$, we may use $i_*((-)[\ta^{-1}]) \simeq Y(\b (-))$ and the equivalence $\b (\MGL_2) \simeq \MU_{\R,2}$ to obtain a map
  \[ i_*\mathrm{cb}(\MGL_2) \to Y(\mathrm{cb}(\MU_{\R,2})). \]
  Applying \Cref{prop:P-MGL-power}, we obtain a factorization
  \[ i_*\mathrm{cb}(\MGL_2) \xrightarrow{\simeq} \tau_{\geq 0}^{2\mathrm{slice}} Y(\mathrm{cb}(\MU_{\R,2})) \to Y(\mathrm{cb}(\MU_{\R,2})). \]
  Taking totalizations, we obtain
  \begin{align*}
    i_*(\Ss_2)
    &\simeq i_*((\Ss_2)_{\MGL_2}^{\wedge})
    \simeq i_* (\mathrm{Tot}( \mathrm{cb}(\MGL_2)) )
    \simeq \mathrm{Tot}( i_* (\mathrm{cb}(\MGL_2)) ) \\
    &\simeq \mathrm{Tot}( \tau_{\geq 0}^{2\mathrm{slice}} Y(\mathrm{cb}(\MU_{\R,2})))
    \simeq \mathrm{Sh}(P_{2\bullet} ; \MU_{\R,2})(\o)
    \simeq R_\bullet 
  \end{align*} 
  where the first equivalence is the $\MGL_2$-completeness of $\Ss_2$ \footnote{One argument for this is that $\HFt$-completeness implies $\MGL_2$-completeness and $\Ss_2$ is $\HFt$-complete by \cite{MAdamsConv}.} and the final equivalence is the definition of $R_\bullet$.
\end{proof}

We close with a simple corollary which we will make use of in \Cref{sec:t-structures}.

\begin{cor} \label{cor:gamma-MGL}
  There is an equivalence $\MGL_2 \simeq \Gamma_*(\MU_{\R,2})$ of commutative algebras in $\SHRat$.
\end{cor}

\begin{proof}
  By \Cref{thm:filt-model}, it suffices to show that there is an equivalence $i_* \MGL_2 \simeq \mathrm{Sh}(P_{2\bullet} ; \MU_{\R,2}) (\MU_{\R,2})$ of commutative algebras over $i_* \Ss_2$.
  Unraveling the construction of the equivalence $i_*\Ss_2 \simeq \mathrm{Sh}(P_{2\bullet} ; \MU_{\R,2}) (\Ss_2)$ we can build a diagram of commutative algebras in $\Sp_{C_2, i2} ^{\Fil}$
  \begin{center} \begin{tikzcd}
      i_*\Ss_2 \ar[r, "\simeq"] \ar[d] &
      \mathrm{Sh}(P_{2\bullet} ; \MU_{\R,2})(\Ss_2) \ar[r] \ar[d] &
      \mathrm{Sh}(P_{2\bullet} ; \MU_{\R,2})(\MU_{\R,2}) \ar[dl, dashed, "\simeq"] \\
      i_*\MGL_2 \ar[r, "\simeq"] &
      P_{2\bullet}\MU_{\R,2} ,
  \end{tikzcd} \end{center}
  where the dashed equivalence comes from \Cref{exm:E-collapse}.
\end{proof}

\section{Modules over the cofiber of ta} \label{sec:cta}
In this section, we study the category of $C\ta$-modules.
Our main theorem states that this category admits an explicit, algebraic description.
Below, we use $\mathrm{IndCoh}(\mathcal{M}_{\mathrm{fg}})$ to denote $\mathrm{Ind}$ of the thick subcategory of $\mathrm{MU}_*\mathrm{MU}$-comodules generated by even shifts of the unit. For more details on this category, see \Cref{dfn:IndCoh}.

\begin{thm} \label{thm:Cta}
  There is an equivalence of presentably symmetric monoidal categories under $\Sp_{C_2, i2}$
  \begin{align} \Mod ( \SHR^{\at} _{i2} ; C\ta ) \simeq \Mod ( \Sp_{C_2, i2}; \uZt) \otimes_{\Z} \mathrm{IndCoh}(\mathcal{M}_{\mathrm{fg}}), \label{eqn:cta} \end{align}
  where $\Sp_{C_2, i2}$ acts on the left-hand-side through $c_{\C/\R}$ and on the left factor of the right-hand-side. 
	On Picard elements this equivalence takes
  \[ C\ta \otimes \Ss_2^{p,q,w} \mapsto \Sigma^{(p-w)+(q-w)\sigma} \uZt \otimes \omega_{\mathbb{G}/\Mfg} ^{\otimes w}. \]
  As a consequence of this equivalence of categories, there is an isomorphism of tri-graded commutative rings:\footnote{The tensor product may be moved outside the Ext, but then it must be taken in the derived sense.}
  \[ \pi_{p,q,w}^\R(C\ta) \cong \bigoplus_{w+a-s = p} \Ext_{\MU_*\MU}^{s,2w}(\MU_*, \MU_* \otimes \pi_{a + (q-w) \sigma}^{C_2} \uZt). \]
\end{thm}

Over $\C$ the corresponding result is the main theorem of \cite{ctauactalol}.
Upon tensoring with $\Spec(\C)$, our theorem recovers theirs. 

\begin{ntn}
  In the above theorem and throughout the section, we follow the convention of writing $\CC \otimes_R \D$ for the tensor product of two presentable $R$-linear categories (instead of $\CC \otimes_{\Mod_R} \D$).
\end{ntn}

\begin{rmk}
  In the statement of the above theorem, we use the $\Z$-linear structure on $\Mod(\Sp_{C_2, i2};\uZt)$ coming from the symmetric monoidal functor $(\_) \otimes_{\uZ} \uZt : \Mod_\Z \to \Mod(\Sp_{C_2, i2};\uZt)$, where $\_ : \Mod_{\Z} \to \Mod(\Sp_{C_2}; \uZ)$ is the symmetric monoidal functor discussed in \Cref{rmk:underline} below.
\end{rmk}

\begin{rmk}
  We will prove in \Cref{lem:Cta-t-structure} that the category $\Mod ( \Sp_{C_2, i2}; \uZt) \otimes_{\Z} \mathrm{IndCoh}(\mathcal{M}_{\mathrm{fg}})$ admits a $t$-structure with heart the category of comodules in $C_2$-Mackey functors over the Hopf algebroid $(\piu_{*\rho} ^{C_2} (\MU_{\R,2}), \piu_{*\rho}^{C_2} (\MU_{\R, 2} \otimes \MU_{\R, 2}))$.
  
  We may therefore view any such comodule as an element of $\Mod ( \Sp_{C_2, i2}; \uZt) \otimes_{\Z} \mathrm{IndCoh}(\mathcal{M}_{\mathrm{fg}})$.
\end{rmk}

In fact, we are able to explicitly describe the trigraded homotopy groups of $Y \otimes C\ta$ for a more general collection of $Y$ than just trigraded spheres.

\begin{dfn}
  We say that $X \in \Sp_{C_2, i2}$ is \emph{$\MU_{\R,2}$-projective}
  if $\MU_{\R,2} \otimes X$ is a retract (as an $\MU_{\R,2}$-module) of
  $\bigoplus_\alpha \Sigma^{n_\alpha \rho} \MU_{\R,2}$
  for some set of integers $n_\alpha$.
\end{dfn}

Examples of $\MU_{\R,2}$-projective $X$ include $\Ss_2 ^0$ and $\MU_{\R,2}$, as well as any object of $\Sp_{C_2, i2}$ built out of cells of the form $\Ss_2^{n\rho}$.

\begin{thm} \label{thm:Cta-homotopy}
  If $X \in \Sp_{C_2, i2}$ is $\MUR$-projective,
  then under the equivalence of \Cref{thm:Cta}  
  $C\ta \otimes \Gamma_*(X)$ corresponds to the $\piu_{*\rho} ^{C_2} (\MU_{\R,2} \otimes \MU_{\R,2})$-comodule
  $\piu_{*\rho} ^{C_2} (\MU_{\R,2} \otimes X)$.

  Moreover, there is an isomorphism of trigraded groups,
  \[ \pi_{p,q,w} ^\R (\Gamma_* (X) \otimes C\ta) \cong \bigoplus_{w+a-s = p} \Ext_{(\MU_2)_* \MU_2}^{s,2w}((\MU_2)_*, (\MU_2)_{*} (\Phi^{e} (X)) \otimes_{\Z_2} \pi_{a + (q-w) \sigma}^{C_2} \uZt), \]
  compatible with the $\pi^{\R} _{p,q,w}C\ta$-module structure in the expected way.
  
%
\end{thm}


\todo{There's a commented rmk here that should be exported to somewhere else.}

The proofs of Theorems \ref{thm:Cta} and \ref{thm:Cta-homotopy} will be quite long, and correctly handling the symmetric monoidal structures involved requires us to take a rather circuitous route. For this reason, before proceeding we provide a sketch of the argument.

The proof of the equivalence
\[ \Mod ( \SHR^{\at} _{i2} ; C\ta ) \simeq \Mod ( \Sp_{C_2, i2}; \uZt) \otimes_{\Z} \mathrm{IndCoh}(\mathcal{M}_{\mathrm{fg}}) \]
has three main steps and one significant subtlety.
First, we 
produce symmetric monoidal functors into each side of \Cref{eqn:cta} from $\Sp_{C_2,i2}^{\Fil}$.
Second, we show that there are commutative algebras $R_1$ and $R_2$ in $\Sp_{C_2, i2}^{\Fil}$ such that the left-hand-side is equivalent to $\Mod ( \Sp_{C_2, i2}^{\Fil} ; R_1)$ and the right-hand-side is equivalent to $\Mod ( \Sp_{C_2, i2}^{\Fil} ; R_2)$.
Third, we examine $R_1$ and $R_2$ directly and find that they are in fact equivalent.

The functor into $C\ta$-modules from $\Sp_{C_2,i2}^{\Fil}$ is the composite of the functor $i_*$ from \Cref{sec:top} with $C\ta \otimes -$.
The functor into the right-hand-side is more delicate to construct. A first guess would be to use the equivalence $\Sp_{C_2,i2}^{\Fil} \cong \Sp_{C_2,i2} \otimes \Sp^{\Fil}$ and tensor the following pair of maps:
\[ \Sp_{C_2, i2} \to \Mod ( \Sp_{C_2, i2} ; \uZt ), \]
\[ \Sp^{\Fil} \to \Mod ( \Sp^{\Fil} ; \Z) \xrightarrow{\Gr} \Mod ( \Sp^{\Gr} ; \Z ) \to \mathrm{IndCoh}(\mathcal{M}_{\mathrm{fg}}) . \]
In fact, this does not produce the correct functor; in $C\ta$-modules the periodicity class $v_1$ lives in $C_2$-degree $\rho$, but the given tensor product functor puts $v_1$ in $C_2$-degree 0. In order to fix this, we twist by a functor
\[ \Mod (\Sp_{C_2,i2}^{\Gr}; \uZt) \xrightarrow{\mathrm{tw}^{\rho}} \Mod (\Sp_{C_2,i2}^{\Gr}; \uZt)\]
that has the effect of tensoring with $\Ss^{n\rho}$ on the $n^{\mathrm{th}}$ graded piece.
The presence of this twist, and its interaction with the symmetric monoidal structure, is the main subtlety of the argument.
\todo{The absence of this twist in the GIKR argument indicates a lack of rigor.}

The second step is straightforward, and demonstrates the power of higher algebra in proving ``Koszul duality'' statements as a corollary of Barr--Beck--Lurie monadicity. At this point we have two commutative algebras $R_1$ and $R_2$, and all we need to do is show they are equivalent.

The third step relies on special properties of $R_1$ and $R_2$.
Both commutative algebras come equippped with a preferred presentation as a totalization of a cosimplicial diagram of commutative algebras. What we do is show that these cosimplicial diagrams are levelwise equivalent. Each cosimplicial diagram takes values in in the heart of a $t$-structure, and in particular is determined by $1$-categorical data. Thus, it is genuinely possible to write down a comparison map by hand and check that it is an equivalence, without concern for the potential infinity of higher coherence data.


\begin{enumerate}
  \item In \Cref{subsec:cats}, we review the categories which appear in the right-hand-side of \Cref{thm:Cta}. We also construct a $t$-structure on the right-hand-side of \Cref{thm:Cta} which will play an important role in our argument.
\item In \Cref{subsec:twists}, we collect some material on twisted $t$-structures on categories of graded and filtered objects, as well as on twisting isomorphisms between them.
\item In \Cref{subsec:slice}, we relate twisted $t$-structures to the slice filtration in $C_2$-equivariant homotopy theory. This finishes the construction of the comparison functors.
\item In \Cref{subsec:reds}, we use Koszul duality to reduce the proof of Theorems \ref{thm:Cta} and \ref{thm:Cta-homotopy} to understanding a specific pair of commutative algebras.
\item In \Cref{subsec:main-lem}, we prove \Cref{lem:main-lem}, which is the key result that allows us to compare the commutative algebras $R_1$ and $R_2$ above. Using this lemma, we then complete the proof of Theorems \ref{thm:Cta} and \ref{thm:Cta-homotopy}.
\end{enumerate}
\todo{This list still needs some attention}



\subsection{Categories of interest} \label{subsec:cats}\ 

In this subsection we set notation for working in the various categories of interest.
In particular, we aim for a practical understanding of the category
\[ \Mod ( \Sp_{C_2, i2}; \uZt) \otimes_{\Z} \mathrm{IndCoh}(\mathcal{M}_{\mathrm{fg}}) \]
that appears on the right-hand-side of \Cref{thm:Cta}. Proceeding from inside out,
we start by fixing notation for the category of $\Z$-modules:

\begin{dfn}
  Let $\Abh$ denote the abelian category of discrete abelian groups.
  Let $\Ab$ denote the category of $\Z$-modules with its standard $t$-structure,
  so that $\Abh$ is the heart of $\Ab$.
\end{dfn}



Next, we describe the $C_2$-equivariant analog of $\Ab$:

\begin{dfn} \label{dfn:mackey}
  Let $M(C_2)^{\heartsuit}$ denote the abelian category of discrete Mackey functors for the group $C_2$,
  which may be explicitly described as follows.
  A Mackey functor $B \in M(C_2) ^{\heartsuit}$ consists of the following:
  \begin{itemize}
    \item abelian groups $B(C_2)$ and $B(\ast)$,
    \item an involution $\sigma : B(C_2) \to B(C_2)$,
    \item homomorphisms $r : B(\ast) \to B(C_2)$ and $t : B(C_2) \to B(\ast)$,
     which satisfy the relations $\sigma \circ r = r$, $t \circ \sigma = t$ and $r \circ t = 1 + \sigma$.
  \end{itemize}
  We will somtimes write $[C_2]$ for the endomorphism $t \circ r$ of $B(*)$.
  We let $M(C_2)$ denote the unbounded derived category of $M(C_2)^{\heartsuit}$ equipped with its standard $t$-structure.\footnote{As defined, for example, in \cite[Definition 1.3.5.8]{HA}.}
\end{dfn}

In our setting, $M(C_2)$ arises naturally because it is equivalent to the category of $\piu_0 ^{C_2} \Ss$-modules in $\Sp_{C_2}$ \cite[Theorem 5.10]{SpectrumMackey}.
This description equips $M(C_2)$ with a symmetric monoidal structure compatible with the $t$-structure. The induced symmetric monoidal structure on $M(C_2)^{\heartsuit}$ is known as the box product and admits an explicit description as in \cite{CpBox}.
\todo{Probably there's a better cite for this, but this one has a nice modern proof.}
\todo{rwb: Its always unnerved me that I don't know an explicit description of the monoidal structure on Mackey functors.}

\begin{exm}
  Given a discrete abelian group $A$, we may consider the constant Mackey functor $\underline{A} \in M(C_2)^{\heartsuit}$. This Mackey functor is determined by $\underline{A} (C_2) = \underline{A} (\ast) = A$, $r = \sigma = \mathrm{id}_{A}$ and $t=2 \cdot \mathrm{id}_{A}$.
  This construction provides an exact colimit-preserving functor
  \[\_ : \Abh \to M(C_2)^{\heartsuit}.\]
  Examining the explicit description of the box product in \cite{CpBox}, we find that $\_$ acquires the structure of a (non-unital) tensor-product preserving lax symmetric monoidal functor.
  In particular, we obtain a commutative algebra $\uZ$ in $M(C_2)^{\heartsuit}$.
\end{exm}
\todo{rwb: Given we describe nothing about the monoidal structure, why is this functor lax-mon?}

\begin{dfn}
  Let $\underline{\mathrm{Ab}}$ denote the the symmetric monoidal category of $\uZ$-modules in $M(C_2)$ (or, equivalently, $\Sp_{C_2}$), equipped with its standard $t$-structure. This is equivalent to the derived category of $\Mod(M(C_2)^{\heartsuit}; \uZ)$; in particular, we have $\uAbh = \Mod(M(C_2) ^{\heartsuit}; \uZ)$.
\end{dfn}

\begin{rmk} \label{rmk:underline}
  Since $\Z$ is the unit of $\Abh$, we learn that $\_$ factors through an exact colimit-preserving symmetric monoidal functor
  \[\_ : \Abh \to \uAbh.\]
  Since $\_$ is exact, colimit-preserving and symmetric monoidal it extends uniquely to a colimit-preserving symmetric monoidal functor $\_ : \Ab \to \uAb$ of derived categories. Moreover, the pair of functors $A \to A(*)$ and $A \to A(C_2)$ extend to limit-preserving functors on the level of the derived category. Since this pair of functors is jointly conservative, we can use the fact that $\underline{A}(*) \cong A$ and $\underline{A}(C_2) \cong A$ to conclude that $\_$ preserves limits as well.
 
\end{rmk}

Before proceeding to explicitly describe $\IndCoh(\Mfg)$, we record a few lemmas for later use. 

\begin{lem} \label{lem:const-mackey}
  The functor $\_ : \Abh \to \uAbh$ is fully faithful and left adjoint to the functor
  $B \mapsto B(\ast)$.
  Its essential image is therefore a coreflective subcategory of $\uAbh$, with coreflector given by the counit map $\underline{B(\ast)} \to B$.
  In particular, for any $B$ in the essential image of $\_$, there are canonical isomorphisms
  \[B \cong \underline{B(\ast)} \cong \underline{B(C_2)}.\]
\end{lem}

\begin{proof} Clear.
\end{proof}

\begin{lem} \label{lem:uabh-image}
  The forgetful functor $\uAbh \to M(C_2)^{\heartsuit}$ is fully faithful, with image spanned by the Mackey functors $B$ for which $[C_2]=2$ as endomorphisms of $B(\ast)$.
\end{lem}

\begin{proof}
  The unit of $M(C_2)^{\heartsuit}$ is the Burnside Mackey functor $\underline{\mathcal{A}}$, whose values are given by $\underline{\mathcal{A}} (C_2) = \Z$ and $\underline{\mathcal{A}} (\ast) = \Z[[C_2]]/([C_2]^2 - 2[C_2])$.
  As a commutative algebra in $M(C_2)^{\heartsuit}$, $\uZ$ is the quotient of $\underline{\mathcal{A}}$ by the relation $[C_2]=2$.
  The result follows.
\end{proof}

The following definition and construction sum up what we need to know about categories of quasicoherent and ind-coherent sheaves on $\Mfg$:

\begin{dfn} \label{dfn:IndCoh}
  We let $\Mfg$ denote the moduli stack of formal groups and let $\QCoh (\Mfg)$ denote the category of quasicoherent sheaves on $\Mfg$.
  It is equivalent to the derived category of evenly graded $\MU_* \MU$-comodules \cite[Remarks 2.38 and 3.14]{GoerssMfg}.
%

  Let $\mathcal{D} \subset \QCoh(\Mfg)$ denote the thick subcategory generated by the sheaves $\omega_{\mathbb{G}/\Mfg} ^{\otimes k}$ for $k \in \Z$, where $\omega_{\mathbb{G} / \Mfg}$ is the sheaf of invariant differentials on the universal formal group $\mathbb{G} / \Mfg$. We define $\IndCoh(\Mfg) \coloneqq \Ind(\mathcal{D})$.
  The category $\IndCoh(\Mfg)$ is equivalent to Hovey's category of evenly graded stable comodules $\mathrm{Stable}_{\MU_* \MU}$, c.f. \cite[Remark 4.30]{LocDual}.
  %

  Both $\QCoh(\Mfg)$ and $\IndCoh(\Mfg)$ are naturally stable presentably symmetric monoidal categories and come equipped with compatible $t$-structures whose hearts are equivalent to the $1$-category of $\MU_* \MU$-comodules. \todo{Come back to this and remark that D = Perf and justify that Indcoh = Indperf here.}
\end{dfn}

\begin{cnstr}
  Since $\QCoh(\Mfg)$ is equivalent to the derived category of a Grothendieck abelian category, it is equipped with the structure of a colimit-preserving symmetric monoidal functor $\Ab \to \QCoh(\Mfg)$.
  Letting $\Ab ^{\mathrm{fin}}$ denote the full subcategory of compact objects, this restricts to an exact symmetric monoidal functor $\Ab ^{\mathrm{fin}} \to \D$, where $\D$ is defined as in \Cref{dfn:IndCoh}.
  Taking $\Ind$ and using the fact that $\Ab$ is compactly generated, we obtain a colimit-preserving symmetric monoidal functor
  \[\Ab \simeq \Ind(\Ab ^{\mathrm{fin}}) \to \Ind(\D) = \IndCoh(\Mfg).\]

  Using this functor and the functor $\Ab \xrightarrow{\_} \uAb \xrightarrow{- \otimes_{\uZ} \uZt} \uAb_{i2}$, we are able to define the stable presentably symmetric monoidal category
  \[\uAb_{i2} \otimes_{\Z} \IndCoh(\Mfg),\]
  which appeared in \Cref{thm:Cta} with slightly different notation.
\end{cnstr}

The core of our understanding of $\uAb_{i2} \otimes_{\Z} \IndCoh(\Mfg)$ rests on the following lemma, where we put a $t$-structure on this category whose heart we can describe explicitly.

\begin{lem} \label{lem:Cta-t-structure}
  There exists a $t$-structure on $\uAb_{i2} \otimes_{\Ab} \IndCoh(\Mfg)$ with the following properties:
  \begin{enumerate}
    \item It is compatible with the symmetric monoidal structure, accessible, compatible with filtered colimits and right complete.
    \item The natural functor
      \[\uAb_{i2, \geq 0} \otimes _{\Ab_{\geq 0}} \IndCoh(\Mfg)_{\geq 0} \to \uAb_{i2} \otimes_{\Ab} \IndCoh(\Mfg)\]
      is fully faithful with essential image the connective objects.
    \item The natural functor
      \[\uAb_{i2} ^{\heartsuit} \otimes_{\Ab^{\heartsuit}} \IndCoh(\Mfg)^{\heartsuit} \to \uAb_{i2} \otimes_{\Ab} \IndCoh(\Mfg)\]
      is fully faithful with essential image the heart of $\uAb_{i2} \otimes_{\Ab} \IndCoh(\Mfg)$.
  \end{enumerate}
\end{lem}

Using the explicit description of $\uAbh$ given by \Cref{dfn:mackey} and \Cref{lem:uabh-image} and the fact that $\IndCoh(\Mfg)^{\heartsuit}$ is equivalent to the category of evenly-graded $\MU_* \MU$-comodules, we obtain the following description of
\[\uAb_{i2}^{\heartsuit} \otimes_{\Ab^{\heartsuit}} \IndCoh(\Mfg)^{\heartsuit}:\]
this is the category of pairs of $(\MU_2)_*\MU_2$-comodules $A(C_2)$ and $A(\ast)$ equipped with following structures (all of which are compatible with the comodule structure):
\begin{itemize}
  \item maps $r : A(\ast) \to A(C_2)$ and $t : A(C_2) \to A(\ast)$, along with an involution $\sigma : A(C_2) \to A(C_2)$,
  \item which satisfy the relations $\sigma \circ r = r$, $t \circ \sigma = t$, $t \circ r = 2$ and $r \circ t = 1 + \sigma$.
\end{itemize}

In other words, $\uAb^{\heartsuit}_{i2} \otimes_{\Ab^{\heartsuit}} \IndCoh(\Mfg)^{\heartsuit}$ is the category of evenly graded comodules over the Hopf algebroid $(\underline{(\MU_2)_{*}}, \underline{(\MU_2)_* \MU_2})$ in graded $C_2$-Mackey functors.

\begin{rmk}
  Using the fact that $$((\underline{\MU_{\R,2}})_{*\rho}, (\underline{\MU_{\R,2}})_{*\rho} \MU_{\R,2}) \cong (\underline{(\MU_2)_{2*}}, \underline{(\MU_2)_{2*} \MU_2}),$$ see \cite[Theorems 2.25 and 2.28]{HuKriz}, this category may equally well be described as the category of graded comodules over the Hopf algebroid $((\underline{\MU_{\R,2}})_{*\rho}, \underline({\MU_{\R,2}})_{*\rho} \MU_{\R,2})$ in graded $C_2$-Mackey functors.
\end{rmk}

The proof of \Cref{lem:Cta-t-structure} is not difficult, but it will rely on material from \cite[Appendix C]{SAG} which we presently recount. We begin with a simple method for producing $t$-structures. 

\begin{cnstr}
  Given a pointed presentable category $\CC$, we let 
  $\Sp \otimes \CC \simeq \Sp(\CC)$
  denote the category of spectrum objects in $\CC$. It is a stable presentable category.

  The category $\Sp \otimes \CC$ admits a $t$-structure with $(\Sp \otimes \CC)_{\geq 0}$ equal to the essential image of the functor $\Sigma^{\infty} : \CC \to \Sp \otimes \CC$.
  This $t$-structure is accessible and right complete.

  Note that if $\CC$ is endowed with the structure of a presentably symmetric monoidal category, then $\Sp \otimes \CC$ naturally inherits this structure and the $t$-structure defined above is compatible with it.
\end{cnstr}

In fact, the following lemma, which follows from \cite[Remark C.3.1.5]{SAG}, demonstrates that this construction is the universal way to produce a (well-behaved) $t$-structure on a presentable stable category. 

\begin{lem} \label{lem:conn-enough}
  Suppose that $(\CC, \CC_{\geq 0})$ is a stable presentably symmetric monoidal category with compatible $t$-structure. Suppose that the $t$-structure is accessible, compatible with filtered colimits and right complete.

  Then there is a natural equivalence of symmetric monoidal categories with compatible $t$-structure
  \[(\CC, \CC_{\geq 0}) \simeq (\Sp \otimes \CC_{\geq 0}, (\Sp \otimes \CC_{\geq 0})_{\geq 0}).\]
\end{lem}


In conclusion, we may as well work with the categories of connective objects. For $\Ab, \uAb_{i2}$, and $\IndCoh(\Mfg)$, these are all Grothendieck prestable categories in the sense of \cite[Definition C.1.4.2]{SAG}.

\begin{rec} \label{rec:groth-ab}
  We let $\Groth_{\infty}$ denote the category of Grothendieck prestable categories \cite[Definition C.3.0.5]{SAG}. By \cite[Theorem C.4.2.1]{SAG}, $\Groth_{\infty}$ is a symmetric monoidal category with tensor product given by the usual tensor product of presentable categories and unit $\Sp_{\geq 0}$.

  Moreover, let $\Groth_{1}$ denote the category of Grothendieck abelian $1$-categories. It is a symmetric monoidal category with tensor product given by the usual tensor product of presentable categories and unit $\Abh$ \cite[Corollary C.5.4.19]{SAG}. There is a functor $\tau_{\leq 0} : \Groth_{\infty} \to \Groth_{1}$ sending a Grothendieck prestable category to its subcategory of discrete objects.
  By \cite[Remark C.5.4.20]{SAG}, the functor $\tau_{\leq 0}$ is symmetric monoidal.

  We summarize this in the following span of symmetric monoidal categories 
  \begin{center}
    \begin{tikzcd}
      & \Groth_{\infty} \ar[dl, "\Sp \otimes -"] \ar[dr, "\tau_{\leq 0}"] \\
      \mathrm{Pr}^{L,\mathrm{Stab}} & & \Groth_{1}.
    \end{tikzcd}
  \end{center}
\end{rec}

\begin{proof}[Proof (of \Cref{lem:Cta-t-structure}).]
  Using \Cref{lem:conn-enough} and \cite[Remark C.4.2.3]{SAG}, we have
  \[\Sp \otimes \left(\uAb_{i2, \geq 0} \otimes_{\Ab_{\geq 0}} \IndCoh(\Mfg)_{\geq 0} \right) \simeq \uAb_{i2} \otimes_{\Ab} \IndCoh(\Mfg).\]
  Moreover, since $\uAb_{i2, \geq 0} \otimes_{\Ab_{\geq 0}} \IndCoh(\Mfg)_{\geq 0}$ is prestable, the $\Sigma^{\infty}$ functor identifies it with $\left(\Sp \otimes \left(\uAb_{i2, \geq 0} \otimes_{\Ab_{\geq 0}} \IndCoh(\Mfg)_{\geq 0} \right)\right)_{\geq 0}$.

  This immediately implies the first two properties, with the exception of compatibility with filtered colimits. This follows from the fact that $\uAb_{i2, \geq 0} \otimes_{\Ab_{\geq 0}} \IndCoh(\Mfg)_{\geq 0}$ is compactly generated, which holds becuase each one of $\uAb_{i2, \geq 0}$, $\Ab_{\geq 0}$ and $\IndCoh(\Mfg)_{\geq 0}$ is compactly generated \cite[Corollary C.6.2.3]{SAG}.

  The third property follows from the fact that $\CC^{\heartsuit} = \tau_{\leq 0} \left( \CC_{\geq 0} \right)$ and the fact that $\tau_{\leq 0}: \Groth_{\infty} \to \Groth_1$ is symmetric monoidal.
\end{proof}

\subsection{Twistings I: $t$-structures on filtered objects} \label{subsec:twists}\ 

%
We begin this subsection with a discussion of several natural $t$-structures on categories of graded and filtered objects, obtained by twisting a given $t$-structure by a Picard element.
In the filtered case our constructions are a straightforward generalization of a construction of Beilinson. Although these constructions are simple, they have the effect of modifying the symmetric monoidal structure on the heart of a category.
Next, we explicitly describe this modification at the level of the heart, where it admits a description in terms of an Euler characteristic. Read another way, there is an Euler characteristic obstruction to producing symmetric monoidal twisting functors. 
We close the subsection by producing symmetric monoidal twisting functors whenever this Euler characteristic obstruction vanishes.

\begin{dfn} \label{dfn:L-twist}
  For the remainder of this subsection, we fix the following data:
  a stable, presentably symmetric monoidal category $\mathcal{C}$
  equipped with a compatible $t$-structure\footnote{Here compatible means that a tensor product of objects which are $\geq 0$ is $\geq 0$ itself, and that the unit is $\geq 0$.}
  and a Picard element $\mathcal{L}$.

  We then define the $\mathcal{L}$-twisted $t$-structure on $\mathcal{C}^{\Gr}$ and $\mathcal{C}^{\Fil}$ as follows:\todo{Cite something to show this allowed?}
  \begin{itemize}
  \item An object $X_\bullet$ of $\mathcal{C}^{\Gr}$ is $\geq 0$ in the $\mathcal{L}$-twisted $t$-structure if, for all $n$, $\mathcal{L}^{\otimes -n} \otimes X_n$ is $\geq 0$ in the original $t$-structure.
  \item Similarly, an object $X_\bullet$ of $\mathcal{C}^{\Fil}$ is $\geq 0$ in the $\mathcal{L}$-twisted $t$-structure if $\mathcal{L} ^{\otimes -n} \otimes X_n$ is $\geq 0$ in the original $t$-structure, for all $n$.
  \end{itemize}
  We denote the hearts of these $t$-structures by $\CC^{\Gr, \mathcal{L}, \heartsuit}$ and $\CC^{\Fil, \mathcal{L}, \heartsuit}$ respectively.
  We denote the $i^{\mathrm{th}}$ $\mathcal{L}$-twisted $t$-structure homotopy objects by $\piu_{i} ^{\mathcal{L}, \Gr}$ and $\piu_{i} ^{\mathcal{L}, \Fil}$ respectively.
  When $\mathcal{L} = \o$, we frequently omit it from the notation.
\end{dfn}

We summarize the basic properties of this definition in the following lemma.

\begin{lem} \label{prop:twist-grade} \label{prop:twist-filt} \label{prop:twist-1-conn}
  In the graded case:
  \begin{enumerate}
  \item The $\mathcal{L}$-twisted $t$-structure on $\CC^{\Gr}$ is compatible with the symmetric monoidal structure.
  \item An object $X_\bullet$ is $< 0$ in the $\mathcal{L}$-twisted $t$-structure if and only if, for all $n$, $\mathcal{L}^{\otimes -n} \otimes X_n$ is $< 0$ in the original $t$-structure.
  \item The $t$-structure homotopy groups are determined by the following formula:
\[ (\underline{\pi}_0^{\mathcal{L}, \Gr} X_\bullet )_n \cong \mathcal{L}^{\otimes n} \underline{\pi}_0(\mathcal{L}^{\otimes -n} \otimes X_n). \]
  \end{enumerate}
  Assuming that $\mathcal{L} \geq 0$, we obtain similar results in the filtered case:
  \begin{enumerate}
  \item[(1')] The $\mathcal{L}$-twisted $t$-structure on $\CC^{\Fil}$ is compatible with the symmetric monoidal structure.
  \item[(2')] An object $X_\bullet$ is $< 0$ in the $\mathcal{L}$-twisted $t$-structure if and only if $\mathcal{L}^{\otimes -n} \otimes X_n < 0$.
  \item[(3')] The $t$-structure homotopy groups are given the by following formula:
  \[ (\underline{\pi}_0^{\mathcal{L}, \Fil} X_\bullet )_n \cong \mathcal{L}^{\otimes n} \underline{\pi}_0(\mathcal{L}^{\otimes -n} \otimes X_n). \]

  \end{enumerate}

  Under the stronger assumption that $\mathcal{L} \geq 1$, 
  the functor $\CC^{\Gr} \to \CC^{\Fil}$ which is right adjoint to the associated graded functor
  restricts to an equivalence of symmetric monoidal $1$-categories
  \[ \CC^{\Gr, \mathcal{L}, \heartsuit} \simeq \CC^{\Fil, \mathcal{L}, \heartsuit}. \]
\end{lem}

\begin{ntn}
  We write $- \otimes^{\heartsuit, \mathcal{L}} -$ for the tensor product induced on either $\CC^{\Gr, \mathcal{L}, \heartsuit}$ or $\CC^{\Fil, \mathcal{L}, \heartsuit}$. In the case that $\mathcal{L} = \o$, we simply write $- \otimes ^{\heartsuit} -$.
\end{ntn}

%

\begin{proof}
  The statements (1), (2) and (3) are all clear. We therefore restrct ourselves to the filtered case for the rest of the proof.

  From the expression of the tensor product as a Day convolution, we observe that a tensor product of connective objects has $n^{\mathrm{th}}$ term presented as a colimit over a diagram of connective objects tensored with $\mathcal{L}$ at least $n$ times. Since $\mathcal{L} \geq 0$ we may conclude that (1') holds.

  The expression for the homotopy objects in (3') follows from (2') in a straightforward way.
  We now prove (2').
  Suppose that $X_\bullet$ is an object such that $\mathcal{L}^{\otimes -n} \otimes X_n < 0$.
  We want to show that $X_\bullet < 0$, i.e. that it recieves no maps from $Y_\bullet \geq 0$.
  Using the condition that $\mathcal{L} \geq 0$, we learn that $[Y_n, X_m] = 0$ for all $n > m$, which is enough to imply that $[Y_\bullet, X_\bullet] = 0$.

  Suppose that $X_\bullet < 0$; then we would like to show that $\mathcal{L}^{\otimes -n} \otimes X_n < 0$.
  Associated to every $Y \geq 0$ in $\CC$ we can consider the filtered object $Y \otimes \mathcal{L}^{\otimes n}(n)$, which is depicted below:
  \[ \dots \to 0 \to 0 \to Y \otimes \mathcal{L} ^{\otimes n} \xrightarrow{\id} Y \otimes \mathcal{L} ^{\otimes n} \xrightarrow{\id} \dots \]
Here, the first nonzero object occurs at position $n$.
  Using the asssumption that $\mathcal{L} \geq 0$, we can conclude that $Y \otimes \mathcal{L}^{\otimes n}(n) \geq 0$. Now, on mapping out of this object we have
  \[ * \simeq \Omega^{\infty} \Map_{\CC^{\Fil}} ( Y \otimes \mathcal{L}^{\otimes n}(n), X) \simeq \Omega^\infty \Map_\CC ( Y \otimes \mathcal{L}^{\otimes n}, X_n), \]
  which implies that $\mathcal{L}^{\otimes -n} \otimes X_n < 0$ for all $n$, as desired.

  We now proceed to prove the final statement.
  Let $\tau : \o_{\CC^{\Fil}} (-1) \to \o_{\CC^{\Fil}}$ denote the shift map in $\CC^{\Fil}$. Then $C\tau$ admits an $\E_\infty$-algebra structure in $\CC^{\Fil}$ such that there is an equivalence of symmetric monoidal categories $\CC^{\Gr} \simeq \Mod(\CC^{\Fil}; C\tau)$ (cf. \Cref{prop:mod-fil}). Using the assumption that $\mathcal{L} \geq 0$, we learn that $\o_{\CC^{\Fil}}$ and $C\tau$ are $\geq 0$. As a consequence, we find that $\CC^{\Gr, \mathcal{L}, \heartsuit}$ may be identified with $\Mod(\CC^{\Fil, \mathcal{L}, \heartsuit}; \piu_0 ^{\mathcal{L}} C\tau)$.

  The proposition will therefore follow if we prove that $\piu_{0} ^{\mathcal{L}} \o_{\CC^{\Fil}} \to \piu_0 ^{\mathcal{L}} C\tau$ is an equivalence.
  By \Cref{prop:twist-filt}, we find that
  \[(\piu_{0} ^{\mathcal{L}, \Fil} \o_{\CC^{\Fil}})_n \cong \begin{cases} \mathcal{L}^{\otimes n} \pi_0 (\mathcal{L} ^{\otimes -n}) & \text{ if } n \leq 0, \\ 0 & \text{ otherwise.} \end{cases}\]
    Since $\mathcal{L} \geq 1$ by assumption, we find for $n < 0$ that $\mathcal{L}^{\otimes n} \pi_0 \mathcal{L} ^{\otimes -n} \cong 0$, whereas for $n=0$ we just get $\pi_0 \o_{\CC}$.

  On the other hand, we have that
  \[(\piu_0 ^{\mathcal{L}, \Fil} C\tau)_n \cong \begin{cases} \pi_0 \o_{\CC} & \text{ if } n = 0, \\ 0 & \text{ otherwise.} \end{cases}\]
  It follows from this that the map $\piu_{0} ^{\mathcal{L}, \Fil} \o_{\CC^{\Fil}} \to \piu_0 ^{\mathcal{L}, \Fil} C\tau$ is an equivalence, as desired.
\end{proof}

We can now profitably define the twist functors.

\begin{cnstr} \label{cnstr:twist}
  Given a symmetric monoidal Grothendieck abelian $1$-category $\mathcal{A}$ and an invertible object $a \in \mathcal{A}$, we can construct a monoidal functor
  $i_a : \Z \to \mathcal{A}$ which sends $1$ to $a$.\footnote{By symmetric monoidal Grothendieck abelian $1$-category, we mean a commutative algebra object of the symmetric monoidal category $\mathrm{Groth}_1$ discussed in \Cref{rec:groth-ab}. In other words, we assume that the tensor product commutes with colimits in each variable.}
  If the swap map $s_{a,a}$ is the identity, then we can make this functor symmetric monoidal.
  Tensoring up with $\Abh$, we obtain a monoidal (or symmetric monoidal) functor
  \[ i_a : \Ab^{\Gr, \heartsuit} \xrightarrow{\o(1) \mapsto a} \mathcal{A}. \]
  
  In our situation of interest, we take $\mathcal{A} = \CC^{\Gr, \mathcal{L}, \heartsuit}$ and $a = (\piu_0 \o) \otimes \mathcal{L}(1)$. Then we can define the functor $\twist^{\mathcal{L}}$ as the composite
  \[ \CC^{\Gr, \heartsuit} \cong \CC^\heartsuit \otimes \Ab^{\Gr, \heartsuit} \xrightarrow{\mathrm{Id} \otimes i_{(\piu_0 \o) \otimes \mathcal{L}(1)}} \CC^\heartsuit \otimes \CC^{\Gr, \mathcal{L}, \heartsuit} \xrightarrow{\otimes} \CC^{\Gr, \mathcal{L}, \heartsuit}. \]
  On objects this has the effect of sending $\{X_n\}$ to $\{X_n \otimes \mathcal{L}^{\otimes n}\}$.
  It is easy to see that this is a monoidal equivalence between $\CC^{\Gr, \heartsuit}$ and $\CC^{\Gr, \mathcal{L}, \heartsuit}$, and is further symmetric monoidal if $s_{a, a}$ is the identity. 
\end{cnstr}

The above construction succinctly explains how much the symmetric monoidal structure on $\CC^{\Gr, \mathcal{L}, \heartsuit}$ has been twisted in terms of the quantity $s_{a, a}$. We now give a description of this quantity in terms of the following standard definition:

\begin{dfn}
  Given a dualizable object $X$,
  with dual $X^\vee$, unit $\eta$, and counit $\epsilon$,
  the trace of an endomorphism $f: X \to X$ is the element $\mathrm{tr}(f) \in \pi_0 \End(\o)$ given by the composite
  \[ \o \xrightarrow{\eta} X \otimes X^\vee \xrightarrow{ f \otimes X^\vee} X \otimes X^\vee \xrightarrow{s_{X,X^\vee}} X^{\vee} \otimes X \xrightarrow{\epsilon} \o. \]
  The trace of the identity is called the Euler characteristic and denoted $\chi(X)$.
\end{dfn}

A simple diagram chase tells us that the Euler characteristic is multiplicative and defines an $\End(\o)$-linear map $\mathrm{tr} : \End(X) \to \End(\o)$. Moreover, one can compute that the trace of the swap map $s_{X,X}$ is equal to $\chi(X)$. In the case where $X$ is invertible the trace map is an isomorphism, and we can use this to conclude that the condition $s_{a,a} = \mathrm{Id}_{a \otimes a}$ in \Cref{cnstr:twist} is equivalent to asking that $\chi(a) = 1$. Specializing further to the case $a = (\piu_0 \o) \otimes \mathcal{L}(1)$, we note that the diagram computing the Euler characteristic of $(\piu_0 \o) \otimes \mathcal{L}(1)$ is $\piu_0^{\mathcal{L},\Gr}$ applied to the diagram computing the Euler characteristic of $\mathcal{L}(1)$ in the ambient category $\CC^{\Gr}$. This implies that
\[ \chi(a) = \chi(\mathcal{L}(1)) = \chi(\mathcal{L})\chi(\o(1)) = \chi(\mathcal{L}) \cdot 1. \]
Thus, we have proved:

\begin{lem} \label{prop:twist}
  The functor $\twist^{\mathcal{L}}$ can be made symmetric monoidal if $\chi(\mathcal{L}) = 1$.
\end{lem}




Conversely, if we assume that $\twist^{\mathcal{L}}$ is symmetric monoidal, then we can make the following Euler characteristic computation:
\[ 1 = \twist^{\mathcal{L}} (1) = \twist^{\mathcal{L}}( \chi (\o(1))) = \chi ( \twist^{\mathcal{L}}(\o(1))) = \chi ( \mathcal{L}(1) ) = \chi (\mathcal{L}). \]

\begin{rmk}
  Examining the failure of $\twist^{\mathcal{L}}$ to be symmetric monoidal, we find that in general the symmetric monoidal structure on $\CC^{\Gr, \mathcal{L}, \heartsuit}$ is twisted by the Euler characteristic $\chi(\mathcal{L})$.
\end{rmk}

\begin{exm}
  If we take $\mathcal{C} = \Sp$ with its usual $t$-structure, then the $\Ss^1$-twisted category $\Sp^{\Gr, \Ss^1, \heartsuit}$ is equivalent to the symmetric monoidal category of graded abelian groups with the Koszul sign convention, since $\chi(\Ss^1) = -1$.

  On the other hand, $\Sp^{\Gr, \Ss^2, \heartsuit}$ is symmetric monoidally equivalent to the category of graded abelian groups by \Cref{prop:twist}, since $\chi(\Ss^2) = 1$. \todo{Define $\chi$ above, note that it is $2$-torsion for a Picard element.}
\end{exm}
%

%

\begin{lem} \label{lem:D-twist}
  Suppose that $\CC$ is equivalent to the derived category of its heart $\CC^{\heartsuit}$ as a symmetric monoidal category with compatible $t$-structure. Then there exists an equivalence of monoidal categories $\twist^{\mathcal{L}} : \CC^{\Gr} \simeq \CC^{\Gr}$ making the following diagram of monoidal functors commute:
  \begin{center}
    \begin{tikzcd}
      \CC^{\Gr, \heartsuit} \ar[r, "\twist^{\mathcal{L}}"] \ar[d] & \CC^{\Gr, \mathcal{L}, \heartsuit} \ar[d] \\
      \CC^{\Gr} \ar[r, "\twist^{\mathcal{L}}"] & \CC^{\Gr}.
    \end{tikzcd}
  \end{center}
  If $\chi(\mathcal{L}) = 1$, then the above square naturally lifts to a diagram of symmetric monoidal functors.
  On objects, $\twist^{\mathcal{L}} : \CC^{\Gr} \to \CC^{\Gr}$ is given by $\{X_n\} \mapsto \{X_n \otimes \mathcal{L}^{\otimes n}\}$.
\end{lem}

\begin{proof}
  The assumption implies that $\CC^{\Gr}$, equipped with the usual $t$-structure, is equivalent as a symmetric monoidal category with compatible $t$-struture to $\mathcal{D}(\CC^{\Gr, \heartsuit})$.
  It also implies that $\CC^{\Gr}$, when equipped with the $\mathcal{L}$-twisted $t$-structure, is equivalent as a symmteric monoidal category with compatible $t$-struture to $\mathcal{D}(\CC^{\Gr, \mathcal{L}, \heartsuit})$.

  It follows that the equivalence of monoidal $1$-categories $\twist^{\mathcal{L}} : \CC^{\Gr, \heartsuit} \to \CC^{\Gr, \mathcal{L}, \heartsuit}$ of \Cref{prop:twist} determines a compatible equivalence of monoidal categories $\twist^{\mathcal{L}}: \CC^{\Gr} \to \CC^{\Gr}$, via taking derived categories.
  If $\chi(\mathcal{L}) = 1$, these are equivalences of symmetric monoidal categories by \Cref{prop:twist}.
\end{proof}

\subsection{Twistings II: the slice filtration} \label{subsec:slice}\ 

In this short subsection, we specialize the material from the previous subsection to the case of interest. This means we look at $\Sp_{C_2, i2}$ with Picard element $\Ss_2 ^\rho$.
In order to connect this with $\MU_{\R,2}$ and the material from \Cref{sec:top}, we relate the $\rho$-twisted $t$-structure to the regular slice filtration.

\begin{exm}
  If we take $\mathcal{C}= \Sp_{C_2}$, then the following Euler characteristic computations provide the three possible nontrivial twists on graded Mackey functors:
  \[ \chi(\Ss^{\sigma}) = 1 - [C_2], \qquad \chi(\Ss^1) = -1, \qquad \chi(\Ss^{\rho}) = [C_2] - 1. \]
  To verify this, one uses the compatibility of the Euler characteristic with the symmetric monoidal functors $\Phi^{e}$ and $\Phi^{C_2}$.
\end{exm}

\begin{exm}
  Write $\uAb_{i2}^{\Gr, \rho, \heartsuit}$ for $\uAb_{i2}^{\Gr, \Sigma^{\rho} \uZt, \heartsuit}$.
  Then, since the image of $\chi(\Ss^{\rho})$ in $\pi_0 \uZt$ is $1$,
  there is a symmetric monoidal equivalence
  \[\twist^{\rho} : \uAb_{i2}^{\Gr, \heartsuit} \xrightarrow{\simeq} \uAb_{i2}^{\Gr, \rho, \heartsuit},\]
  by \Cref{prop:twist}. Applying \Cref{lem:D-twist}, we can then upgrade this to a diagram of symmetric monoidal functors
  \begin{center}
    \begin{tikzcd}
      \uAb_{i2} ^{\Gr, \heartsuit} \ar[r, "\twist^{\rho}"] \ar[d] & \uAb_{i2} ^{\Gr, \rho, \heartsuit} \ar[d] \\
      \uAb_{i2} ^{\Gr} \ar[r, "\twist^{\rho}"] & \uAb_{i2} ^{\Gr}.
    \end{tikzcd}
  \end{center}
\end{exm}

\todo{We're gonna need to either explain more about how the slices work, cite somebody or write more shit here.}
\todo{We need to fix a convention does pi n live in deg n (composite of two truncation functors) or in deg 0.}

\begin{cnv}
  From this point on, our category $\CC$ will be one of $\Sp_{C_2, i2}$, $M(C_2)_{i2}$, $\uAb_{i2}$ or $\Ab_{i2}$ with its usual $t$-structure, and the Picard element we work with will be $\Ss^{\rho} _2, \Sigma^{\rho} \piu_{0} ^{C_2} \Ss_2, \Sigma^{\rho} \uZt$ or $\Sigma^{2} \Zt$, respectively. In order to reduce the notational burden, we denote the former three Picard objects by $\rho$ and the final Picard object by $2$ in superscripts. So, in the example of $\Sp_{C_2, i2}$, we will write $\tau_{\geq 0} ^\rho$ for the connective cover, $\piu_0 ^\rho$ for the $0^{\mathrm{th}}$ homotopy object and $\Sp_{C_2, i2} ^{\Fil, \rho, \heartsuit}$ for the heart of $\Sp_{C_2, i2} ^{\Fil}$.
\end{cnv}

This key result of this section is the following:

\begin{lem}\label{prop:twist-slice}
  Let $E \in \Sp_{C_2, i2}$. Then
  \[\tau_{\geq 0} ^{\rho} Y(E) \simeq \left( \dots \to P_4 E \to P_2 E \to P_0 E \to P_{-2} E \to P_{-4} E \to \dots \right), \]
  where $Y$ is the functor that sends an object to the associated constant filtered object. 
  If we further suppose that $\piu_{n\rho-1} ^{C_2} E = 0$ for all $n$, there is a natural equivalence
  \[\Gr (\tau_{\geq 0} ^{\rho} Y(E)) \simeq \piu_0 ^{\rho} Y(E) \cong \{\Sigma^{n\rho} \piu_{n\rho} ^{C_2} E\}.\]
  Finally, if $\piu_{n\rho} E$ is a constant Mackey functor for all $n$, then
  $ \piu_0 ^{\rho} Y(E) $
  factors through the full subcategory
  \[\uAb_{i2} ^{\Gr, \rho, \heartsuit} \subset M(C_2)_{i2} ^{\Gr, \rho, \heartsuit} \simeq \Sp_{C_2, i2} ^{\Fil, \rho, \heartsuit}.\]
\end{lem}

\begin{proof}
  This lemma has three statements.
  For the first statement, it suffices to note the following:
  \begin{enumerate}
    \item A $C_2$-spectrum is $t$-structure $0$-connective if and only if it is regular slice $0$-connective.
\item A $C_2$-spectrum $E$ is regular slice $2n$-connective if and only if $E \otimes \Ss^{-n\rho}$ is regular slice $0$-connective.
  \end{enumerate}

  For the second statement, since $\Ss^{\rho} _2 \geq 1$, we can use the final statement of \Cref{prop:twist-1-conn} to obtain a canonical map
  \[\Gr (\tau_{\geq 0} ^{\rho} Y(E)) \to \piu_0 ^{\rho} Y(E).\]
  On the one hand, we have
  $ \piu_0 ^\rho Y(E) \cong \{\Sigma^{n\rho} \piu_{n\rho} ^{C_2} E\} $
  by \Cref{prop:twist-grade}(3'). \todo{I dont 100\% follow this argument.}
  Since by assumption the odd slices of $E$ vanish, we learn that
  the $(2n)^{\mathrm{th}}$ slice of $E$ is equivalent to $\Sigma^{n\rho} \piu_{n\rho} ^{C_2} E$ and
  that there are cofiber sequences
  \[ P_{2n+2} E \to P_{2n} E \to \Sigma^{n\rho} \piu_{n\rho} ^{C_2} E. \]
  The desired equivalence now follows from the first statement.

  The third statement now follows from the fact that any discrete $C_2$-Mackey functor which is constant admits a (necessarily unique) $\uZ$-module structure.
\end{proof}

\subsection{Reductions} \label{subsec:reds}\ 

In this subsection, we express each side of the equivalence (\ref{eqn:cta}) of \Cref{thm:Cta} as the category of modules over some commutative algebra in $\Sp_{C_2, i2} ^{\Fil}$.
Almost everything in this subsection is a standard application of Lurie's theory of higher algebra.
We have extracted the necessary material in \Cref{app:boilerplate}, where we present it in a digested form. 

We begin with the left-hand-side of (\ref{eqn:cta}).

\begin{lem} \label{lem:gr-of-comparison}
  The cofiber of $\ta : \Ss_2^{0,0,-1} \to \Ss_2^{0,0,0}$ is a commutative algebra in $\SHR_{i2}$. Furthermore, there is an equivalence of symmetric monoidal categories
  \[ \mathrm{Mod}(\SHRat; C\ta) \simeq \mathrm{Mod}( \Sp_{C_2, i2}^{\Gr}; \Gr( R_\bullet ) ), \]
	where $R_{\bullet} =\Tot^{*}\left( P_{2\bullet} \MU_{\mathbb{R},2}^{\otimes *+1} \right)$ 
\end{lem}

\begin{proof}
  Recall that the shift map $\Ss_2(-1) \to \Ss_2$ in $\Sp_{C_2, i2}^{\Fil}$ is denoted by $\tau$.
    Tensoring the equivalence of \Cref{exm:fil-ctau} with $\Sp_{C_2, i2}$ and applying \Cref{lem:tensor-mod}, we see that there is an equivalence of presentably symmetric monoidal categories
    \[ \Sp_{C_2, i2} ^{\Gr} \simeq \Mod ( \Sp_{C_2, i2} ^{\Fil} ; C\tau ). \]

    Under the symmetric monoidal functor $i^* : \Sp_{C_2, i2} ^{\Fil} \to \SHRat$ constructed in \Cref{prop:top-diagram}, the commutative algebra $C\tau$ maps to $C\ta$. Therefore, using \Cref{lem:tensor-mod}, we have 
    \[ \mathrm{Mod}(\SHRat; C\ta) \simeq \Sp_{C_2, i2}^{\Gr} \otimes_{\Sp_{C_2, i2} ^{\Fil}} \SHRat. \]
    Using \Cref{thm:filt-model}, there is an equivalence of presentably symmetric monoidal categories under $\Sp_{C_2, i2}^{\Fil}$ between $\SHRat$ and $\mathrm{Mod}( \Sp^{\Fil}_{C_2,i2} ; R_\bullet )$. Tensoring down along the associated graded ring map and using \Cref{lem:tensor-mod}, we obtain equivalences
    \begin{align*}
      \Sp_{C_2, i2}^{\Gr} \otimes_{\Sp_{C_2, i2}^{\Fil}} \SHRat
      &\simeq \Sp_{C_2,i2}^{\Gr} \otimes_{\Sp_{C_2,i2}^{\Fil}} \mathrm{Mod}( \Sp_{C_2,i2}^{\Fil} ; R_\bullet ) \\
      &\simeq \mathrm{Mod}( \Sp_{C_2,i2}^{\Fil} ; \Gr(R_{\bullet}) ). \qedhere
    \end{align*}
      %

\end{proof}

Now we proceed to the right-hand-side of (\ref{eqn:cta}).

\begin{lem} \label{lem:omega-euler-one}
  The Euler characteristic of $\omega_{\G/\Mfg} \in \IndCoh(\Mfg)^{\heartsuit}$ is equal to 1.
  As a consequence, \Cref{cnstr:twist} provides a symmetric monoidal left adjoint
  \[ p^* : \Ab^{\Gr, \heartsuit} \to \IndCoh(\Mfg)^{\heartsuit} \]
  which sends $\Z(1)$ to $\omega_{\G/\Mfg}$.
\end{lem}

\begin{proof}
  Let $L$ denote the Lazard ring. Then the flat cover $\Spec(L) \to \Mfg$ determines a pullback map $\IndCoh(\Mfg) \to \Mod_{L}$. This determines an injective pullback map
  $\pi_0 \End(\mathcal{O}_{\Mfg}) \to \pi_0 \End(L)$.
  Since the pullback of $\omega_{\G/\Mfg}$ is equivalent to $L$ we learn that $\chi(\omega_{\G/\Mfg}) = 1$.
\end{proof}

Since $p^*$ sends a family of compact dualizable objects (the powers of $\Z(1)$) to a family of compact dualizable generators (the powers of $\omega_{\G/\Mfg}$), we may apply \Cref{prop:rigid} to obtain the following lemma:

\begin{lem}\label{lem:MU-koszul}
  There is an equivalence of presentably symmetric monoidal categories
  \[ \mathrm{IndCoh}(\mathcal{M}_{\mathrm{fg}}) \simeq \mathrm{Mod}(\Ab^{\Gr}; p_*\mathcal{O}_{\Mfg} ). \]
\end{lem}

  

Using \Cref{lem:MU-koszul} and \Cref{lem:tensor-mod}, we obtain the following corollary:

\begin{cor}\label{lem:uMU-koszul}
  There is an equivalence of presentably symmetric monoidal categories
  \[ \uAb_{i2} \otimes_{\Ab} \mathrm{IndCoh}(\mathcal{M}_{\mathrm{fg}}) \simeq \Mod ( \uAb_{i2}^{\Gr}; \underline{p_*\mathcal{O}_{\Mfg}} \otimes_{\uZ} \uZt ).  \]
\end{cor}


\subsection{The main lemma} \label{subsec:main-lem}\

In this subsection, we prove our main lemma, \Cref{lem:main-lem}, and use it to prove Theorems \ref{thm:Cta} and \ref{thm:Cta-homotopy}. This lemma gives us an explicit formula for $\Gr(i_*(\Gamma_*(-)))$ on a restricted class of objects. Once we have this formula the remaining work is relatively easy.
We begin with a definition and a couple of useful lemmas.

\begin{dfn}\ 
  \begin{itemize}
  \item Let $\Sp_{C_2, i2} ^{\mathrm{proj}}$ denote the full subcategory of $E \in \Sp_{C_2, i2}$ for which $\MU_{\R,2} \otimes E$ is a retract of a sum of pure suspensions of $\MU_{\R,2}$ (a suspension $\Sigma^{a+b\sigma} \MU_{\R,2}$ is said to be pure if $a=b$). 
  \item Let $\Sp_{i2} ^{\mathrm{proj}}$ denote the full subcategory of $E \in \Sp_{i2}$ for which $\MU_{2} \otimes E$ is a retract of a sum of even suspensions of $\MU_2$.
  \end{itemize}
  Since the underlying spectrum of $\MU_{\R, 2}$ is $\MU_2$, the underlying spectrum of an object of $\Sp_{C_2, i2} ^{\mathrm{proj}}$ is contained in $\Sp_{i2} ^{\mathrm{proj}}$.
  Note that both $\Sp_{C_2,i2} ^{\mathrm{proj}}$ and $\Sp_{i2} ^{\mathrm{proj}}$ contain units and are closed under tensor product, so each inherits the structure of a symmetric monoidal category.
\end{dfn}

\begin{lem} \label{lem:MUR-uZ-mod}
  Suppose that $E \in \Sp_{C_2, i2} ^{\mathrm{proj}}$. Then 
  \[\underline{(\MU_{\R,2})}_{n\rho} ^{C_2} (E) \cong \underline{(\MU_2)_{2*} (E)}.\]
  As a consequence, $\underline{(\MU_{\R,2})}_{n\rho} ^{C_2} (E)$ acquires the structure of a $\uZt$-module.
\end{lem}

\begin{proof}
  By definition of $\Sp_{C_2, i2} ^{\mathrm{proj}}$, it suffices to note that this is true for $E = \Ss_2$, which follows from \cite[Theorem 2.28]{HuKriz}.
\end{proof}

\begin{lem} \label{lem:main-lem}
  Let $h$ denote the symmetric monoidal functor $h : \Sp_{i2} ^{\mathrm{proj}} \to \IndCoh(\Mfg)_{i2} ^{\heartsuit}$
  which sends a spectrum to its associated $((\MU_2)_*, (\MU_2)_*\MU_2 )$-comodule.  
  There is a commutative diagram of lax symmetric monoidal functors
  \begin{center}
    \begin{tikzcd}
      \Sp_{C_2, i2} ^{\mathrm{proj}} \ar[rrrr,"i_* \circ \Gamma_*"] \ar[d, "\Phi^e"] & & & &
      \Sp_{C_2, i2} ^{\Fil} \ar[d, "\Gr"] \\
      \Sp^{\mathrm{proj}} _{i2} \ar[r, "h"] &
      \IndCoh(\Mfg)_{i2} ^{\heartsuit} \ar[r, "p_*"] &
      \Ab^{\Gr}_{i2} \ar[r, "\_"] &
      \uAb^{\Gr}_{i2} \ar[r, "\twist^\rho"] &
      \Sp_{C_2,i2} ^{\Gr}.
    \end{tikzcd}
  \end{center}

\end{lem}

\begin{proof}



  Our argument will rest on the existence of the following commuting diagram of lax symmetric monoidal functors.
  \begin{center}
    \begin{tikzcd}
      & \Sp_{i2} ^{\mathrm{proj}} \ar[ddl, "h"] \ar[d, "- \otimes \mathrm{cb}(\MU_2)"] & &
      \Sp_{C_2,i2} ^{\mathrm{proj}} \ar[d, "- \otimes \mathrm{cb}(\MU_{\R,i2})"] \ar[ddl, dashed, "f"'] \ar[ll, "\Phi^e"] \\
      & \Sp_{i2}^{\Delta} \ar[d, "\piu_0^2Y"] & &
      \Sp_{C_2,i2} ^{\Delta} \ar[ll, "\Phi^e"] \ar[r,"\tau_{\geq 0} ^\rho Y"'] \ar[d, "\piu_{0}^{\rho} Y"'] &
      \Sp_{C_2,i2} ^{\Fil, \Delta} \ar[d, "\Gr"] \\
      \IndCoh(\Mfg)_{i2}^{\heartsuit} \ar[ddr, "p_*"] & \Ab_{i2}^{\Gr,2,\heartsuit,\Delta} \ar[d, "\twist^{-2}"] &
      \uAb_{i2} ^{\Gr, \rho,\heartsuit, \Delta} \ar[l, "\Phi^e"] \ar[d, "\twist^{-\rho}"] \ar[r, "\mathrm{forget}"] &
      \Sp_{C_2,i2} ^{\Gr, \rho, \heartsuit, \Delta} \ar[r] & \Sp_{C_2,i2}^{\Gr, \Delta} \ar[dd, "\Tot"] \\
      & \Ab_{i2}^{\Gr,\heartsuit,\Delta} \ar[d, "\Tot"] &
      \uAb_{i2}^{\Gr, \heartsuit, \Delta} \ar[l, "\Phi^e"] \ar[d, "\Tot"] & & \\
      & \Ab_{i2}^{\Gr} &
      \uAb_{i2}^{\Gr} \ar[r, "\twist^{\rho}"] \ar[l, "\Phi^e"] &
      \uAb_{i2}^{\Gr} \ar[r,"\mathrm{forget}"] &
      \Sp_{C_2, i2} ^{\Gr}.
    \end{tikzcd}
  \end{center}

  The squares formed by the underlying functor, $\Phi^e$, commute since underlying commutes with limits and colimits, the underlying of $\Ss_2^\rho$ is $\Ss_2^2$ and the underlying of $\MU_{\R,2}$ is $\MU_2$.
  In grid position $(4,3)$ we have implicitly made use of the equivalence $\Sp_{C_2} ^{\Fil,\rho,\heartsuit,\Delta} \simeq \Sp_{C_2} ^{\Gr, \rho, \heartsuit, \Delta}$ to identify the target of $\piu_0 ^{\rho}Y$ with the latter. The upper-right square commutes by \Cref{prop:twist-slice}. The dashed arrow, labelled $f$, is unique (and lax symmetric monoidal) if it exists since the forgetful functor is fully-faithful. The dashed arrow then exists by \Cref{lem:MUR-uZ-mod}. The large bottom-right square commutes because $\twist^{\rho}$ commutes with $\Tot$. Finally, the existence of the factorization along the left side is a corollary of our ability to identify the $\mathrm{E}_2$-page of the Adams--Novikov spectral sequence (see \cite{Adams}). 

  Using the fact that $\Gr$ and $\Tot$ commute, we have equivalences
  \[ \Gr \circ i_* \circ \Gamma_* \simeq \Tot  \Gr ( \tau_{\geq 0}^\rho Y ( - \otimes \mathrm{cb}(\MU_{\R,2}))) \simeq \twist^{\rho} (\Tot (\twist^{-\rho}( f(-)))). \]
  By \Cref{lem:MUR-uZ-mod}, the functor $f$ (and its composition with the twist) lands in the full subcategory of the target which is spanned, in each grading and cosimplicial degree, by $\uZt$-modules for which the restriction map $r$ is an equivalence. Thus, we have an equivalence of lax symmetric monoidal functors
  \[ \twist^{\rho} (\Tot (\twist^{-\rho}( f(-)))) \simeq \twist^{\rho} (\Tot (\underline{\Phi^e \twist^{-\rho}( f(-))})). \]
  Since $\_$ commutes with limits we then have further equivalences,
  \[ \twist^{\rho} (\Tot (\underline{\Phi^e \twist^{-\rho}( f(-))}))
  \simeq \twist^{\rho} ( \underline{\Tot (\Phi^e \twist^{-\rho}( f(-))) } )
  \simeq \twist^{\rho} ( \underline{  p_* h \Phi^e(-) } ). \qedhere \]
\end{proof}

Finally, we are able to prove our main theorems:

\begin{proof}[Proof of \Cref{thm:Cta} and \Cref{thm:Cta-homotopy} (part 1).]
  We have a chain of symmetric monoidal equivalences:

  \begin{align*}
    \mathrm{Mod}(\SHRat; C\ta)
    &\xrightarrow[\ref{lem:gr-of-comparison}]{\simeq} \mathrm{Mod}( \Sp_{C_2, i2}^{\Gr}; \Gr(R_{\bullet}) ) 
    \xrightarrow[\ref{lem:main-lem}]{\simeq} \mathrm{Mod}( \uAb_{i2}^{\Gr}; \twist^{\rho} (\underline{p_*\mathcal{O}_{\Mfg}}) ) \\
    &\xrightarrow[\twist^{-\rho}]{\simeq} \mathrm{Mod}( \uAb_{i2}^{\Gr}; \underline{p_*\mathcal{O}_{\Mfg}} ) 
    \xrightarrow[\ref{lem:uMU-koszul}]{\simeq} \uAb_{i2} \otimes_{\Ab} \IndCoh(\Mfg)
  \end{align*}

  The claim about the $\Sp_{C_2,i2}$-algebra structure follows once we know $\Sp_{C_2,i2}$ acts through the ambient category at each step. 
  The key points here are
  that the equivalence from \Cref{lem:main-lem} was induced by an equivalence of algebras and
  that the twist functor is an $\uAb_{i2}$-algebra map.

  Using the $\Sp_{C_2,i2}$-linearity, the claim about Picard elements reduces to tracking $C\ta \otimes \Ss_2^{0,0,1}$ around the diagram. Under the first equivalence, $C\ta \otimes \Ss_2^{0,0,1}$ goes to $\o (1)$. The twist by $-\rho$ sends this to $\Sigma ^{\rho}\o (1)$. The final equivalence sends $\o(1)$ to $\uZt \otimes \omega_{\mathbb{G}/\Mfg}$. Altogether we learn that $C\ta \otimes \Ss_2^{0,0,1}$ is sent to $\Sigma^{-\rho}\uZ_2 \otimes \omega_{\mathbb{G}/\Mfg}$, as desired.

  Now suppose that $X \in \Sp_{C_2}^{\mathrm{proj}}$.
  Then, using \Cref{lem:main-lem}, we have
  \[\Gr(i_* (\Gamma_* X)) \simeq \twist^{\rho} (\underline{p_* h \Phi^e (X) }). \]
  Continuing along the sequence of equivalences above, we see that this object is sent to
  $\uZt \otimes h(\Phi^e(X))$, as desired. \qedhere

\end{proof}

In order to calculate the homotopy groups of $C\ta$-modules, we begin with a simple lemma.

\begin{lem} \label{lem:underline-homotopy}
  Given any $A \in \Ab$, we have
  \[\pi_{p+q\sigma} ^{C_2} \underline{A} \cong \bigoplus_{a} \pi_{p-a} \left(A \otimes \pi_{a+q\sigma} \uZ \right).\]
\end{lem}

\begin{proof}
  Consider the composite
  $ \Ab \to \uAb \to \Ab^{\Gr} $
  given by
  $ A \mapsto \underline{A} \mapsto \{\Map(\Ss^{q\sigma}, \underline{A})\} $.
  This composite is colimit-preserving, hence is equivalent to
  $ A \mapsto \{A \otimes \Map(\Ss^{q\sigma} , \underline{\Z})\} $.
  Now, we have
  \[\Map(\Ss^{q\sigma} , \underline{\Z}) \simeq \bigoplus_a \Sigma^{a} \pi_{a+q\sigma} \uZ \]
  because it is a $\Z$-module, and the result follows by applying $\pi_p$.
\end{proof}

\begin{proof}[Proof of \Cref{thm:Cta} and \Cref{thm:Cta-homotopy} (part 2).]
  We now turn to our assertions about homotopy groups.
  Tracing through the various equivalences, we find that for $X$ which is $\MU_{\R,2}$-projective
  \[ \pi_{p,q,w}^\R(C\ta \otimes X) \cong \pi_{(p-w) + (q-w)\sigma}^{C_2} ( (\uZt \otimes_{\Zt} p_*h \Phi^e (X))_w ). \]
  We can commute taking the $w^{\mathrm{th}}$ component past the tensor product and then apply \Cref{lem:underline-homotopy} and the fact that $p_*h$ computes $\Ext_{(\MU_2)_*\MU_2}$ to conclude that
  \begin{align*}
    \pi_{p,q,w}^\R(C\ta \otimes X)
    &\cong \pi_{(p-w) + (q-w)\sigma}^{C_2} (\uZt \otimes_{\Zt} (p_*h \Phi^e (X))_w ) \\
    &\cong \bigoplus_a \pi_{p-w-a} (\pi_{a+(q-w)\sigma}\uZt \otimes_{\Zt} \Ext_{(\MU_2)_*\MU_2}^{*,2w}((\MU_2)_*, (\MU_2)_* \Phi^e (X) )) \\
    &\cong \bigoplus_a \pi_{p-w-a} (\Ext_{(\MU_2)_*\MU_2}^{*,2w}((\MU_2)_*, \pi_{a+(q-w)\sigma}\uZt \otimes_{\Zt} (\MU_2)_* \Phi^e (X) )) \\
    &\cong \bigoplus_a \Ext_{(\MU_2)_*\MU_2}^{a+w-p,2w}((\MU_2)_*, \pi_{a+(q-w)\sigma}\uZt \otimes_{\Zt} (\MU_2)_* \Phi^e (X) ). 
  \end{align*}
  Since $\pi_{a+q\sigma} ^{C_2} \uZt$ is isomorphic to $\Z_2$ or $\F_2$, and $(\MU_2)_* X$ is torsion-free (since it is a projective $(\MU_2)_*$-module), the tensor product inside the Ext can be taken in a $1$-categorical sense.
\end{proof}

\section{The Chow $t$-structure and $\nur$} \label{sec:t-structures}
In this section, we discuss the Chow $t$-structure on $\SHRat$ and how it can be used to define an interesting lax symmetric monoidal functor
\[\nur : \Sp_{C_2,i2} \to \SHRat.\]
In the context of $\mathrm{DM}(k)$, the Chow $t$-structure was first studied by Bondarko \cite{Bondarko}.\footnote{The publication history of this definition is somewhat complicated. In the published version of \cite{Bondarko}, the Chow $t$-structure is only conjectured, not shown, to exist. However, in a later version of \cite{Bondarko} which appeared on the arXiv, Bondarko gave a construction of the Chow $t$-structure.}
In the context of $\SHC^{\at}$ it was studied by Pstragowski in \cite{Pstragowski}.
The Chow $t$-structure on $\SH(k)$ is the subject of forthcoming work of Bachmann--Kong--Wang--Xu, who we thank for carefully explaining their results to us. Among their results is a determination of the heart of the Chow $t$-structure in full generality.\footnote{In the time since the first version of this work was posted, the theorems of Bachmann--Kong--Wang--Xu appeared in \cite{Chow}}



In \Cref{sec:chow}, we define the Chow $t$-structure on $\SHRat$, determine its heart and give a formula for the $t$-structure homotopy objects.
In \Cref{sec:nu}, we use the Chow $t$-structure to define the functor
\[\nur : \Sp_{C_2} \to \SHRat.\]
We prove several basic properties of this functor, compare it with the functor $\Gamma_*$ defined in \Cref{sec:top}, and show that several Artin-Tate $\R$-motivic spectra may be recovered from their Betti realizations via $\nur$.
%
%
%
\subsection{The Chow $t$-structure} \label{sec:chow}\ 

In this section, we define the Chow $t$-structure on Artin-Tate $\R$-motivic spectra, compute its heart, and describe the $t$-structure homotopy objects, which we denote $\pi^{\ch} _k$.

\begin{dfn}
  By \cite[Proposition 1.2.1.16]{HA}, we can construct a $t$-structure on $\SHRat$ with
	\begin{itemize}
	\item connective part $(\SHRat)^{\cn} _{\geq 0}$ generated under colimits and extensions by \[\{ \Ss_2^{n+k_1,n+k_2,n}\ \vert \ k_1 \geq 0 \text{ and } k_1 + k_2 \geq 0 \}.\]
	\item $(-1)$-coconnective part $(\SHRat)^{\cn} _{< 0}$ spanned by those $X \in \SHRat$ with $\pi_{n+k_1, n+k_2, n} ^{\R} X = 0$ for all $n,k_1,k_2 \in \Z$ satisfying $k_1 \geq 0$ and $k_1+k_2 \geq 0$.
	\end{itemize}
  We call this $t$-structure the \emph{Chow $t$-structure} on $\SHRat$.
%
%
%
%
  Moreover, we let $\tau_{\geq 0} ^{\cn} : \SHRat \to (\SHRat)^{\cn} _{\geq 0}$ denote the connective cover with respect to the Chow $t$-structure, and let $\pi_n ^{\ch} : \SHRat \to \SHR_{i2}^{\at, \ch}$ denote the homotopy object functors.
\end{dfn}

The following result is immediate from the definition:
\begin{prop}
  The Chow $t$-structure on $\SHRat$ is compatible with the symmetric monoidal structure, right complete and compatible with filtered colimits.
  Moreover, $\Sigma^{n,n,n}$ restricts to an equivalence $\Sigma^{n,n,n} : (\SHRat)^{\cn} _{\geq 0} \to (\SHRat)^{\cn} _{\geq 0}$.
\end{prop}

The main two theorems that we prove in this section are as follows:

\begin{thm} \label{thm:chow-heart}
  Let $\SHRatch$ denote the heart of $\SHRat$ with respect to the Chow $t$-structure. Then there is a symmetric monoidal equivalence of categories
  \[\SHRatch \simeq \Comod(\uAbh_{i2}; \underline{(\MU_2)_{2*} \MU_2}),\]
  where $\Comod(\uAbh; \underline{(\MU_2)_{2*} \MU_2})$ is the category of evenly-graded $(\underline{(\MU_2)_{*}}, \underline{(\MU_2)_{*} \MU_2})$-comodules in $C_2$-Mackey functors.
\end{thm}

\begin{thm} \label{thm:pi0-MGL}
  Given any $X \in \SHRat$ and $n,k \in \Z$, there is an isomorphism
  \[(\pi_k ^{\ch} X)_n \cong \underline{(\MGL_2)}_{n+k,n,n} (X).\]
\end{thm}

\begin{rec}
In \Cref{thm:Cta}, we constructed a symmetric monoidal equivalence
\[ \Mod(\SHRat; C\ta) \simeq \Mod(\Sp_{C_2, i2}; \uZt) \otimes_{\Z} \IndCoh(\Mfg). \]
Moreover, in \Cref{lem:Cta-t-structure} we equipped $\Mod(\Sp_{C_2, i2}; \uZt) \otimes_{\Z} \IndCoh(\Mfg)$ with a $t$-structure whose heart is identified with $\Comod(\uAbh; \underline{(\MU_2)_{2*} \MU_2})$.
  We call the induced $t$-structure on $\Mod(\SHRat; C\ta)$ the \emph{tensor $t$-structure} and write $\Mod(\SHRat; C\ta)^{\th}$ for the heart.
\end{rec}

To prove \Cref{thm:chow-heart}, it will suffice to prove the following:

\begin{prop} \label{prop:t-compare}
  The tensor $t$-structure on $\Mod(\SHRat; C\ta)$ is induced by the Chow $t$-structure on $\SHRat$: a $C\ta$-module $X$ is tensor (co)connective if and only if its underlying Artin-Tate $\R$-motivic spectrum is Chow (co)connective.

  Moreover, the induced symmetric monoidal functor
  \[\Mod(\SHRat; C\ta)^{\th} \to \SHR_{i2}^{\at, \ch}\]
  is an equivalence of categories.
\end{prop}


First, we need a lemma:

\begin{lem} \label{lem:Cta-pi0}
  We have $C\ta \in \SHRatch$ and the unit map $\Ss_2 ^{0,0,0} \to C\ta$ induces an equivalence $\pi_0 ^{\ch} \Ss_2 ^{0,0,0} \simeq C\ta$.
\end{lem}

\begin{proof}
  First, we note that $\Ss_2^{0,0,-1} \geq 1$ in the Chow $t$-structure.
  Since $$\Ss_2^{0,0,-1} \simeq \Sigma \Ss_2^{0,1,0} \otimes \Ss_2^{-1,-1,-1},$$ it suffices to show that $\Ss_2^{0,1,0} \geq 0$.
  This follows immediately from the cofiber sequence $\Sigma^{\infty} _+ \Spec(\C)_2 \to \Ss_2^{0,0,0} \to \Ss_2^{0,1,0}$.

  Combining this fact with the cofiber sequence $\Ss_2^{0,0,-1} \xrightarrow{\ta} \Ss_2^{0,0,0} \to C\ta$, we find that $C\ta \geq 0$ and that $\Ss_2^{0,0,0} \to C\ta$ induces an equivalence on $\pi_0 ^{\ch}$.

  To conclude, it suffices to show that $C\ta \leq 0$.
  By definition, we must show that $\pi_{n+k_1, n+k_2, n} ^{\R} C\ta = 0$ for all $n,k_1,k_2$ satisfying $k_1 > 0$ and $k_1 + k_2 > 0$.
  This follows directly from \Cref{thm:omnibus-vanishing}(2).
\end{proof}

We note down the following corollary for later use: 

\begin{cor} \label{cor:cta-homotopy}
  Suppose that $X \in \SHRat$ is a filtered colimit of Artin-Tate $\R$-motivic spectra that each admit a finite cell structure with all cells of the form $\Ss_2^{n,n,n}$.
  Then $X \otimes C\ta \in \SHRatch$ and $X \to X \otimes C\ta$ induces an equivalence $\pi_0 ^{\ch} X \to X \otimes C\ta$.
\end{cor}

\begin{proof}
  It is clear that the collection of $X$ which satisfy the conclusions of the corollary is closed under filtered colimits and extensions, so it suffices to assume that $X \simeq \Ss_2^{n,n,n}$
  Since $\Sigma^{n,n,n}$ is an automorphism of the Chow $t$-structure, we may reduce to the case of $X \simeq \Ss_2^{0,0,0}$, which is precisely \Cref{lem:Cta-pi0}.
\end{proof}


\begin{proof}[Proof of \Cref{prop:t-compare}]
  We first show that the tensor $t$-structure is induced by the Chow $t$-structure.
  We begin by showing that the connective part of the tensor $t$-structure on $\Mod(\SHRat; C\ta)$ is generated under colimits and extensions by
  \[ \{ \Sigma^{n+k_1, n+k_2,n} C\ta \ \vert \ k_1 \geq 0 \text{ and } k_1 + k_2 \geq 0 \}. \]

  The $t$-structure on $\IndCoh(\Mfg)$ has connective part generated under colimits and extensions by $\omega^{\otimes n} _{\mathbb{G}/ \Mfg}$ for $n \in \Z$.
  On the other hand, the $t$-structure on $\Mod(\Sp_{C_2, i2}; \uZt)$ has connective part generated under colimits and extensions by $\Sigma^{k_1+ k_2 \sigma} \uZt$, where $k_1 \geq 0$ and $k_1 + k_2 \geq 0$.

  By definition, it follows that the connective part of the induced $t$-structure on $\Mod(\Sp_{C_2, i2}; \uZt) \otimes_{\Z} \IndCoh(\Mfg)$ is generated under colimits and extensions by the tensor products $\Sigma^{k_1 + k_2 \sigma} \uZt \otimes \omega^{\otimes n}_{\mathbb{G} / \Mfg}$, where $k_1 \geq 0$ and $k_1 + k_2 \geq 0$. Under the equivalence of \Cref{thm:Cta}, these correspond to $\Sigma^{n+k_1, n+k_2, n} C\ta$, as desired. \todo{Check sign of $n$.}

  It follows directly from the above identification of the tensor connective category that a $C\ta$-module is is tensor coconnective if and only if it is Chow coconnective.
  By Chow connectivity of $C\ta$, which follows from \Cref{lem:Cta-pi0}, it also follows that a $C\ta$-module is Chow connective if it is tensor connective.
  On the other hand, if $X$ is Chow connective, then $X \otimes C\ta$ is clearly tensor connective. Using the equivalence $X \otimes C\ta \simeq X \oplus \Sigma^{1,0,-1} X$ (since $X$ is a $C\ta$-module), we see that $X$ itself is also tensor connective.

  As a consequence of the above, we obtain a symmetric monoidal functor
  \[ \Mod(\SHRat; C\ta)^{\th} \to \SHRatch. \]
  Since the tensor $t$-structure is induced by the Chow $t$-structure, the fact that $C\ta$ lies in the Chow heart from \Cref{lem:Cta-pi0} implies that the above functor factors through a symmetric monoidal equivalence
  \[ \Mod(\SHRat; C\ta)^{\th} \simeq \Mod(\SHRatch; C\ta). \]
  Finally, the equivalence $\pi_0 ^{\ch} \Ss_2 ^{0,0,0} \simeq C\ta$ of \Cref{lem:Cta-pi0} implies that $C\ta$ is the unit of $\SHRatch$, so that the forgetful functor
  \[\Mod(\SHRatch; C\ta) \to \SHRatch\]
  is a symmetric monoidal equivalence.
\end{proof}

We now move on to the proof of \Cref{thm:pi0-MGL}. First we need to recall some basic facts about cell structures on finite approximations of $\MGL$.

\begin{dfn}
  Let $\mathrm{Gr}_k (\mathbb{A}^{n+k})$ denote the Grassmannian scheme of $k$-planes in $\mathbb{A}^{n+k}$, and let $\gamma$ denote the tautological $k$-dimensional vector bundle over $\mathrm{Gr}_k (\mathbb{A}^{n+k})$.

  We let $\MGL(n,k)$ denote the Thom spectrum of the virtual bundle $\gamma - k\cdot \mathrm{triv}$. Then there is a canonical map $\MGL (n_1,k_1) \to \MGL(n_2, k_2)$ whenever $n_2 \geq n_1$ and $k_2 \geq n_2$, and we have $\varinjlim \MGL(n,n) \simeq \MGL$.
\end{dfn}

\begin{prop}[\cite{DICell}] \label{prop:MGL-chow-conn} \label{cor:MGL-dual-Chow-conn}
  The $\R$-motivic spectrum $\MGL(n,k)$ admits a finite cell structure with cells of the form $\Ss^{m,m,m}$.

  As a consequence, $\MGL(n,k)_2$, its Spanier-Whitehead dual $\mathbb{D}(\MGL(n,k)_2)$, and $\MGL_2$ are all Chow connective.
\end{prop}

\begin{proof}
  The cell structure is a consequence of \cite{DICell}.
  It follows from \Cref{thm:complete-compare} and \Cref{prop:compl-facts}(1) that there are equivalences:
  \begin{align*}
    &\MGL(n,k) \otimes \Ss_2 \simeq \MGL(n,k)_2 \\
    &\mathbb{D} (\MGL(n,k)) \otimes \Ss_2 \simeq \mathbb{D} (\MGL(n,k)_2) \\
    &\MGL \otimes \Ss_2 \simeq \MGL_2.
  \end{align*}
  Therefore these spectra continue to have the same cell structure after $2$-completion, from which the statements about Chow connectivity follow immediately.
\end{proof}

%
%

\begin{prop} \label{prop:MGL-Cta}
  The map $\MGL_2 \to \MGL_2 \otimes C\ta$ induces an equivalence $\pi_0 ^{\ch} \MGL_2 \simeq \MGL_2 \otimes C\ta$, and the corresponding $\underline{(\MU_2)_{2*} \MU_2}$-comodule is $\underline{(\MU_2)_{2*} \MU_2}$.
\end{prop}

\begin{proof}
  The first part of the statement follows directly from \Cref{cor:cta-homotopy} and \Cref{prop:MGL-chow-conn}.
  On the other hand, \Cref{cor:gamma-MGL} states that $\MGL_2 \simeq \Gamma_* (\MU_{\R,2})$, so that \Cref{thm:Cta-homotopy} and \Cref{prop:t-compare} together imply that $\MGL_2 \otimes C\ta$ coreponds to $\piu_{*\rho} ^{C_2} (\MU_{\R,2} \otimes \MU_{\R,2}) \cong \underline{(\MU_2)_{2*} \MU_2}$.
\end{proof}

\begin{ntn}
  We let $\otimes^{\heartsuit}$ denote the symmetric monoidal structure on $\SHRatch \simeq \Comod(\uAbh_2; \underline{(\MU_2)_{2*} \MU_2})$.

  Given a graded $\underline{(\MU_2)_{2*} \MU_2}$-comodule $M_*$, we let $M[n]_*$ denote the shift with $M[n]_{k+n} = M_{k}$.
\end{ntn}

\begin{lem} \label{lem:MGL-heart-tensor}
  Let $M_*$ be a graded $\underline{(\MU_2)_{2*} \MU_2}$-comodule in $C_2$-Mackey functors, viewed as a element of $\SHRatch$.
  Then $M_* \otimes \MGL_2$ lies in $\SHRatch$ and corresponds to $M_* \otimes^{\heartsuit} \underline{(\MU_2)_{2*} \MU_2}$.
\end{lem}

\begin{proof}
  It follows from \Cref{prop:MGL-chow-conn} that $M_* \otimes \MGL_2$ is Chow connective.
  As a consequence, \Cref{prop:MGL-Cta} implies that there are equivalences
  \begin{align*}
    \pi_0 ^{\ch} (M_* \otimes \MGL_2) &\simeq M_* \otimes^{\heartsuit} \pi_0 ^{\ch} \MGL_2 \\
    &\simeq M_* \otimes^{\heartsuit} \underline{(\MU_2)_{2*} \MU_2}.
  \end{align*}

  It therefore suffices to show that $M_* \otimes \MGL_2$ is Chow coconnective.
  For this, we compute for $n,k_1, k_2 \in \Z$ with $k_1 > 0$ and $k_1+ k_2 > 0$:
  \begin{align*}
    \pi_{n+k_1, n+k_2, n} ^{\R} (M _* \otimes \MGL_2) &\cong [\Ss_2^{n+k_1, n+k_2, n}, M _* \otimes \MGL_2]_{\SHRat} \\
    &\cong \varinjlim [\Ss_2^{n+k_1, n+k_2, n}, M_* \otimes \MGL(m,m)_2]_{\SHRat} \\
    &\cong \varinjlim [\Ss_2^{n+k_1, n+k_2, n} \otimes \mathbb{D} (\MGL(m,m)_2), M_*]_{\SHRat} \\
    &\cong 0
  \end{align*}
  since $\mathbb{D} (\MGL(m,m)_2)$ is Chow connective by \Cref{prop:MGL-chow-conn}.
  It follows that $M_* \otimes \MGL_2$ is Chow coconnective, as desired.
\end{proof}

\begin{lem} \label{lem:MGL-tensor-thing}
  Suppose that $X$ is Chow connective. Then the map $X \otimes \MGL_2 \to \pi_0 ^{\ch} (X) \otimes \MGL_2$ induces an equivalence on $\piu_{n-k_1,n-k_2,n} ^{\R}$ for all $n,k_1,k_2 \in \Z$ with $k_1 \geq 0$ and $k_1 + k_2 \geq 0$.
\end{lem}

\begin{proof}
  It suffices to show that $\piu_{n-k_1, n-k_2, n} ^{\R} Y \otimes \MGL_2 = 0$ for all Chow $1$-connective $Y$.
  Since this condition is closed under filtered colimits, suspensions and extensions, it is closed under all colimits and extensions. It therefore suffices to assume that $Y = \Ss_2^{m+a_1, m+a_2, m}$ for $m,a_1,a_2 \in \Z$ satisfying $a_1 > 0$ and $a_1 + a_2 > 0$.
  In other words, we must show that $\piu_{n-k_1, n-k_2, n} ^{\R} \MGL_2$ for all $n,k_1, k_2 \in \Z$ now satisfying $k_1 > 0$ and $k_1 + k_2 > 0$.

  By \Cref{prop:P-MGL}, there is an isomorphism $\piu_{n-k_1, n-k_2, n} ^{\R} \MGL_2 \cong \piu_{n-k_1 + (n-k_2)\sigma} ^{C_2} P_{2n} \MU_{\R,2}$. This is zero because $\Phi^{e} (\Ss_2^{n-k_1 + (n-k_2)\sigma}) \simeq \Ss_2^{2n-k_1-k_2}$ is of dimension $< 2n$ and $\Phi^{C_2} (\Ss_2^{n-k_1 + (n-k_2) \sigma}) \simeq \Ss_2^{n-k_1}$ is of dimension $< n$. 
\end{proof}

\begin{proof}[Proof of \Cref{thm:pi0-MGL}]
%
  We have
  \begin{align*}
    \pi_{0} ^{\ch} (X)_n &\cong \uHom_{\underline{(\MU_2)_{2*}}-\mathrm{mod}} (\underline{(\MU_2)_{2*}}[n], \pi_0 ^{\ch} (X)) \\
    &\cong \uHom_{\underline{(\MU_2)_{2*} \MU_2}-\mathrm{comod}} (\underline{(\MU_2)_{2*}}, \pi_0 ^{\ch} (X) \otimes^{\heartsuit} \underline{(\MU_2)_{2*} \MU_2}) \\
    &\cong \piu_0 ^{C_2} \Map_{\SHRat} (\Ss_2^{n,n,n}, \pi_0 ^{\ch} (X) \otimes \MGL_2) \\
    &\cong \piu_0 ^{C_2} \Map_{\SHRat} (\Ss_2^{n,n,n}, X \otimes \MGL_2) \\
    &\cong \underline{(\MGL_2)}^{\R}_{n+k,n,n} (X),
  \end{align*}
  where the third and fourth isomorphisms follow from \Cref{lem:MGL-heart-tensor} and \Cref{lem:MGL-tensor-thing}, respectively.
\end{proof}

%

Finally, we record the following proposition for use in \Cref{sec:nu}.

\begin{prop} \label{prop:complete-compare}
  An Artin-Tate $\R$-motivic spectrum $X \in \SHRat$ is left complete with respect to the Chow $t$-structure if and only if it is $\MGL_2$-local.
\end{prop}

The proof will make use of the following lemma:

\begin{lem} \label{lem:conservative-funcs}
  The functors $X \mapsto \piu_{n+k,n,n} ^{\R} X$ are jointly conservative on $\SHRat$.
\end{lem}

\begin{proof}
  Since the functors $X \mapsto \Map_{\SHRat} (\Ss_2^{n,n,n}, X)$ are jointly conservative, it suffices to note that the functors $Y \mapsto \piu_{k} ^{C_2} Y$ are jointly conservative on $\Sp_{C_2}$.
\end{proof}

\begin{proof}[Proof of \Cref{prop:complete-compare}]
  Since the Chow $t$-structure is right complete, it is equivalent to show that $X \otimes \MGL_2 \simeq 0$ if and only if $\pi_{n} ^{\ch} X = 0$ for all $n$. This is an immediate consequence of \Cref{thm:pi0-MGL} and \Cref{lem:conservative-funcs}.
\end{proof}

\subsection{The functor $\nur$} \label{sec:nu}\ 

In this section, we will study the functor defined below:

\begin{dfn}
  We define a lax symmetric monoidal functor
  \[\nur : \Sp_{C_2, i2} \to \SHRat \]
  to be the composite
  \[\Sp_{C_2, i2} \xrightarrow{\b^{-1}} \Mod(\SHRat; \Ss_2[\ta^{-1}]) \xrightarrow{\tau_{\geq 0} ^{\cn}} \SHRat.\]
\end{dfn}

One of the main results of this section is the following theorem, which summarizes the basic properties of $\nur$:

\begin{thm} \label{thm:nu-properties}
  Let $X, Y, Z \in \Sp_{C_2, i2}$. The lax symmetric monoidal functor $\nur : \Sp_{C_2, i2} \to \SHRat$ satisfies the following properties:
  \begin{enumerate}
    \item There is a natural equivalence $\b(\nur(X)) \simeq X$.
    \item The functor $X \mapsto \nur(X)$ commutes with filtered colimits.
    \item Given any $k \in \Z$, there is a natural equivalence $\nur(\Sigma^{k\rho} X) \simeq \Sigma^{k,k,k} \nur(X)$.
    \item Suppose that $X \to Y \to Z$ is a cofiber sequence. Then
      \[\nur (X) \to \nur (Y) \to \nur (Z)\]
      is a cofiber sequence if and only if $\piu_{n\rho-1} ^{C_2} (X \otimes \MU_{\R, 2}) \to \piu_{n\rho-1} ^{C_2} (Y \otimes \MU_{\R, 2})$ is a monomorphism for all $n$.
    \item Suppose that $X$ is a filtered colimit of $C_2$-spectra that admit finite cell structures with all cells of the form $\Ss_2^{n\rho}$ for $n \in \Z$. Then, for all $Y$, the natural map
      \[\nur (X) \otimes \nur(Y) \to \nur(X \otimes Y)\]
      is an equivalence.
    \item The $C_2$-spectrum $X$ is $\MU_{\R, 2}$-local if and only if $\nur (X)$ is $\MGL_2$-local. 
  \end{enumerate}
\end{thm}

We also compare $\nur$ to the functor $\Gamma_* : \Sp_{C_2, i2} \to \SHRat$ constructed in \Cref{cnstr:gamma-star}.

\begin{thm} \label{thm:Gamma-nu-comp}
  For $X \in \Sp_{C_2, i2}$, there is a natural equivalence
  \[\Gamma_* (X) \simeq \nur (X) ^{\wedge} _{\MGL_2}.\]
\end{thm}

As evidence that the functor $\nur$ is a good way to produce Artin-Tate $\R$-motivic spectra, we prove that several $\R$-motivic spectra of interest may be recovered from their Betti realizations via $\nur$:

\begin{thm}\label{thm:recover}
  There are natural equivalences of commutative rings in $\SHRat$:
  \begin{align*}
    \MGL_2 &\simeq \nur (\MU_{\R,2}) \\
    \HFt &\simeq \nur (\uFt) \\
    \HZt &\simeq \nur (\uZt) \\
    \kgl_2 &\simeq \nur (\ku_{\R,2}) \\
    \kq_2 &\simeq \nur (\ko_{C_2, 2}).
  \end{align*}
\end{thm}

\begin{rmk}
  In light of \Cref{thm:recover}, it would be reasonable to define an $\R$-motivic spectrum of motivic modular forms as $\nur (\tmf_{C_2,2})$, assuming that one had a suitable $C_2$-equivariant spectrum of connective topological modular forms $\tmf_{C_2,2}$.
  Since no such equivariant spectrum has yet been constructed, we leave this to future work. \todo{Mention why previous attempts don't work? Errors in the literature? And perhaps mention that $\TMF_{C_2}$ is known to exist?}
\end{rmk}

\begin{rmk} \label{rmk:gamma-nu-same}
  Since all of the Artin-Tate $\R$-motivic spectra in question are $\MGL_2$-complete, in light of \Cref{thm:Gamma-nu-comp} we may replace $\nur$ in \Cref{thm:recover} by $\Gamma_*$.
\end{rmk}

We now begin with the proof of \Cref{thm:nu-properties}.

\begin{proof}[Proof of \Cref{thm:nu-properties}] \

  \textbf{Proof of (1):} Equivalently, we are to show that $\nur(X)[\ta^{-1}] \simeq \b^{-1} (X)$.
  Using the cofiber sequence $\nur(X) \to \b^{-1} (X) \to \tau_{\leq -1} ^{\cn} (\b^{-1} (X))$, we find that the cofiber of $\nur(X) [\ta^{-1}] \to \b^{-1} (X)[\ta^{-1}] \simeq \b^{-1} (X)$ is may be expressed as $\varinjlim \Sigma^{0,0,m} \tau_{\leq -1} ^{\cn} (\b^{-1} (X))$.

  Now, we know that $\pi_{n+k_1, n+k_2, n} ^{\R} (\tau_{\leq -1} ^{\cn} (\b^{-1} (X))) \cong 0$ for all $n,k_1,k_2 \in \Z$ with $k_1 > 0$ and $k_1+k_2 > 0$.
  It follows that $\pi_{n+k_1, n+k_2, n} ^{\R} ( \Sigma^{0,0,m} \tau_{\leq -1} ^{\cn} (\b^{-1} (X))) \cong 0$ for all $n,k_1,k_2 \in \Z$ with $k_1 > -m$ and $k_1+k_2 > -2m$.
  This implies that $\varinjlim \Sigma^{0,0,m} \tau_{\leq -1} ^{\cn} (\b^{-1} (X))$ is null, as desired.

  \textbf{Proof of (2):} This is immediate from the fact that the Chow $t$-structure is compatible with filtered colimits.

  \textbf{Proof of (3):} It is clear that there is a natural equivalence $\b^{-1} (\Sigma^{k\rho} X) \simeq \Sigma^{k,k,k} \b^{-1} (X)$.
  Moreover, since $\Sigma^{m,m,m}$ induces an automorphism of the Chow $t$-structure, we find further that $\tau_{\geq 0} ^{\cn}$ commutes with $\Sigma^{m,m,m}$.
  Combining these facts, we obtain the desired result.

  \textbf{Proof of (4):} It is clear that $X \to Y \to Z$ is a cofiber sequence if and only if $\b^{-1} (X) \to \b^{-1} (Y) \to \b^{-1} (Z)$ is.
  Applying a general fact about $t$-strutures, we see that
  \[\tau_{\geq 0} ^\cn (\b^{-1} (X)) \to \tau_{\geq 0} ^\cn (\b^{-1} (Y)) \to \tau_{\geq 0} ^\cn (\b^{-1} (Z)) \]
  is a cofiber sequence if and only if $\pi_{-1} ^{\ch} (\b^{-1} (X)) \to \pi_{-1} ^{\ch} (\b^{-1} (Y))$ is a monomorphism.
  The result then follows from \Cref{thm:pi0-MGL} and the isomorphism $\piu_{p,q,w} ^\R (\MGL_2 \otimes \b^{-1} (X)) \cong \piu_{p+q\sigma} ^{C_2} (\MU_{\R, 2} \otimes X)$.

  \textbf{Proof of (5):} By compatibility with filtered colimits, we may assume that $X$ admits a finite cell structure with all cells of the form $\Ss_2^{n\rho}$. Let us write
  \[\Ss_2^{n_1 \rho} \simeq X_1 \to X_2 \to \dots \to X_k = X\]
  with $X_i / X_{i-1} \simeq \Ss_2^{n_i \rho}$.
  Then it is easy to prove that $\piu_{n\rho-1} ^{C_2} (X_i \otimes \MU_{\R, 2}) = 0$ for all $n$ and $i$ by induction, so that by (3) and (4) there are cofiber sequences
  \[\nur (X_{i-1}) \to \nur (X_i) \to \Ss_2^{n_i, n_i, n_i}.\]

  It follows that $\mathbb{D}(\nur(X))$ is Chow connective. To show that the natural map $\nur(X) \otimes \nur(Y) \to \nur(X \otimes Y)$ is an equivalence, it suffices to show that it induces an equivalence on $\Omega^{\infty} \Map_{\SHRat} (Z, -)$ for all Chow connective $Z$.
  For such $Z$, we have
  \begin{align*}
    \Omega^{\infty} \Map_{\SHRat} (Z, \nur(X) \otimes \nur(Y))
    &\simeq  \Omega^{\infty} \Map_{\SHRat} (Z \otimes \mathbb{D} (\nur(X)), \nur(Y)) \\
    &\simeq \Omega^{\infty} \Map_{\SHRat} (Z \otimes \mathbb{D} (\nur(X)), \b^{-1} (Y)) \\
    &\simeq \Omega^{\infty} \Map_{\SHRat} (Z, \nur(X) \otimes \b^{-1} (Y)) \\
    &\simeq \Omega^{\infty} \Map_{\SHRat} (Z, \b^{-1} (X) \otimes \b^{-1} (Y)) \\
    &\simeq \Omega^{\infty} \Map_{\SHRat} (Z, \nur(X \otimes Y)),
  \end{align*}
  as desired. We have used Chow connectivity of $\mathbb{D} (\nur(X))$ in the second equivalence and (1) in the fourth.

  \textbf{Proof of (6):} It is clear that $X$ is $\MU_{\R, 2}$-local if and only $\b^{-1} (X)$ is $\MGL_2$-local.
  By \Cref{prop:complete-compare}, it suffices to show that $\b^{-1} (X)$ is left complete with respect to the Chow $t$-structure if and only if $\nur(X) = \tau_{\geq 0} ^{\cn} (\b^{-1} (X))$ is.
  This follows from the fact that left completeness is invariant under taking connective cover.
\end{proof}

We now move on to the proof of \Cref{thm:Gamma-nu-comp}.
We begin by studying the interaction of the Chow $t$-structure with $\MGL_2$.

\begin{prop} \label{prop:Chow-MGL}
  Let $X \in \SHRat$. Then the following statements hold:
  \begin{enumerate}
    \item If $X$ is Chow connective, then
      \[\Map_{\SHRat} (\Ss_2^{0,0,n}, \MGL_2 \otimes X)\]
      is regular slice $2n$-connective for all $n$.
    \item There is a natural equivalence
      \[\MGL_2 \otimes (\tau^{\cn} _{\geq 0} X) \simeq \tau_{\geq 0} ^{\cn} (\MGL_2 \otimes X).\]
    \item The natural map
      \[\Map_{\SHRat} (\Ss_2^{0,0,n}, \MGL_2 \otimes \tau_{\geq 0} ^{\cn} X) \to \Map_{\SHRat} (\Ss_2^{0,0,n}, \MGL_2 \otimes X)\]
      induced by the counit $\tau_{\geq 0} ^{\cn} X \to X$ is equivalent to the canonical map
      \[P_{2n} \Map_{\SHRat} (\Ss_2^{0,0,n}, \MGL_2 \otimes X) \to \Map_{\SHRat} (\Ss_2^{0,0,n}, \MGL_2 \otimes X).\]
  \end{enumerate}
\end{prop}

\begin{proof}
  We begin with the proof of (1).
  Since the full subcategory of $X$ for which $\Map_{\SHRat} (\Ss_2^{0,0,n}, \MGL_2 \otimes X)$ is regular slice $2n$-connective is closed under colimits and extensions, it suffices to prove this for $X = \Ss_2^{m+k_1,m+k_2,m}$ for all $m,k_1,k_2 \in \Z$ where $k_1 \geq 0$ and $k_1 + k_2 \geq 0$.

%
  Now, we have
  \[\Map_{\SHRat} (\Ss_2^{0,0,n}, \MGL_2 \otimes \Ss_2^{m+k_1,m+k_2,m}) \simeq \Sigma^{m+k_1+(m+k_2)\sigma} \Map_{\SHRat} (\Ss_2^{0,0,n-m}, \MGL_2).\]
  Since $\Sigma^{m+k_1+(m+k_2)\sigma}$ of a regular slice $2(n-m)$-connective $C_2$-spectrum is regular slice $2n$-connective, this reduces us to the case where $X = \Ss_2^{0,0,0}$, where this follows from \Cref{prop:P-MGL}.

  We now prove (2).
  Since $\MGL_2$ is Chow connective, there is a natural map 
  \[(\tau_{\geq 0} ^{\cn} X) \otimes \MGL_2 \to \tau_{\geq 0} ^{\cn} (X \otimes \MGL_2).\]
  To show that this is an equivalence, it suffices to show that
  \[\Omega^{\infty} \Map_{\SHRat} (Y, (\tau_{\geq 0} ^{\cn} X) \otimes \MGL_2) \to \Omega^{\infty} \Map_{\SHRat} (Y, \tau_{\geq 0} ^{\cn} (X \otimes \MGL_2))\]
  is an equivalence for all $Y$ compact and Chow connective.
  This is equivalent to the map
  \[\Omega^{\infty} \Map_{\SHRat} (Y, (\tau_{\geq 0} ^{\cn} X) \otimes \MGL_2) \to \Omega^{\infty} \Map_{\SHRat} (Y, X \otimes \MGL_2).\]
  Now, using compactness of $Y$, the identification $\MGL_2 \simeq \varinjlim_{k} \MGL(k,k)_2$, and duality, we may rewrite this map as
  \[\varinjlim_k \Omega^{\infty} \Map_{\SHRat} (Y \otimes \mathbb{D} (\MGL(k,k)_2), (\tau_{\geq 0} ^{\cn} X)) \to \varinjlim_k \Omega^{\infty} \Map_{\SHRat} (Y \otimes \mathbb{D} (\MGL(k,k)_2), X).\]
  This is an equivalence by \Cref{cor:MGL-dual-Chow-conn}.

  Finally, we prove (3).
  By (1), the map
  \[\Map_{\SHRat} (\Ss_2^{0,0,n}, \MGL_2 \otimes \tau_{\geq 0} ^{\cn} X) \to \Map_{\SHRat} (\Ss_2^{0,0,n}, \MGL_2 \otimes X)\]
  factors through a map
  \[\Map_{\SHRat} (\Ss_2^{0,0,n}, \MGL_2 \otimes \tau_{\geq 0} ^{\cn} X) \to P_{2n} \Map_{\SHRat} (\Ss_2^{0,0,n}, \MGL_2 \otimes X).\]

  Since both sides are regular slice $2n$-connective, it suffices to show that this map is an equivalence after applying $\Omega^{\infty} \Sigma^{-n\rho}$. Making some basic manipulations, this is equivalent to showing that
  \[\Omega^{\infty} \Map_{\SHRat} (\Ss_2^{n,n,n}, \MGL_2 \otimes \tau_{\geq 0} ^{\cn} X) \to \Omega^{\infty} \Map_{\SHRat} (\Ss_2^{n,n,n}, \MGL_2 \otimes X)\]
  is an equivalence.
  This is equivalence by (2) and the Chow connectivity of $\Ss_2^{n,n,n}$.
\end{proof}

We are now ready to prove \Cref{thm:Gamma-nu-comp}.
%
  %
  %
  %
%
\begin{proof}[Proof of \Cref{thm:Gamma-nu-comp}]
  In this proof, we freely use the notation of \Cref{sec:top} and \Cref{app:compare}, in particular the functors $i_*$, $\tau_{\geq 0} ^{2\slice}$, and $\cb$.

  Let $X$ be a $C_2$-spectrum. By \Cref{prop:Chow-MGL}(2), there are equivalences
  \begin{align*}
    \nur(X)^{\wedge} _{\MGL_2}
    &\simeq \Tot^{\bullet} (\nur(X) \otimes \cb(\MGL_2)) \\
    &\simeq \Tot^{\bullet} (\nur(X \otimes \cb(\MU_{\R, 2}))).
  \end{align*}
  Applying $i_*$ to the map $\nur(X \otimes \cb(\MU_{\R, 2})) \to \b^{-1} (X \otimes \cb(\MU_{\R, 2}))$ of $\cb(\MGL_2)$-modules in cosimplicial Artin-Tate $\R$-motivic spectra, we obtain a map
  \[i_* (\nur (X \otimes \cb(\MU_{\R, 2}))) \to X \otimes \cb(\MU_{\R, 2})\]
  of $i_* (\cb(\MGL_2)) \simeq \tau_{\geq 0}^{2\slice} \cb(\MU_{\R, 2})$-modules in cosimplicial filtered $C_2$-spectra, where the target is constant in the filtered direction.

  By \Cref{prop:Chow-MGL}(3), this factors through an equivalence
  \[i_* (\nur (X \otimes \cb(\MU_{\R, 2}))) \xrightarrow{\sim} \tau_{\geq 0} ^{2\slice} (X \otimes \cb(\MU_{\R, 2}))\]
  of $\tau_{\geq 0} ^{2\slice} \cb(\MU_{\R, 2})$-modules.
  Totalizing, we obtain an equivalence
  \[i_* (\nur (X)^{\wedge} _{\MGL_2}) \simeq i_* (\Gamma_* (X))\]
  of $R_{\bullet}$-modules.
  Finally, translating back along the equivalence $i_*$, we obtain the desired equivalence
  \[\nur(X)^{\wedge} _{\MGL_2} \simeq \Gamma_* (X)\]
  in $\SHRat$.
%
\end{proof}

%
%
We now move on to the proof of \Cref{thm:recover}.
We will prove \Cref{thm:recover} as an application of a general criterion for there to be an $\MGL_2$-local equivalence $X \simeq \nur (\b(X))$. This will be based on the following definition:


\begin{dfn}
  We say that $X \in \SHRat$ is \emph{slice simple} if, for all $n \in \Z$, the natural map
  \[ \Map_{\SHRat} (\Ss_2^{0,0,n}, X) \to \b(X)\]
  factors through an equivalence
  \[ \Map_{\SHRat} (\Ss_2^{0,0,n}, X) \simeq P_{2n} \b(X).\]
\end{dfn}

\begin{prop} \label{prop:slice-simple-nu}
  Let $X \in \SHR^{\at}$, and let $L_{\MGL_2}$ denote $\MGL_2$-localization. Then $L_{\MGL_2} X \simeq L_{\MGL_2} \nur (\b(X))$ if and only if $\MGL_2 \otimes X$ is slice simple.

  If $X$ is a commutative algebra, then the equivalence $L_{\MGL_2} X \simeq L_{\MGL_2} \nur (\b(X))$ is one of commutative algebras.
%
%
\end{prop}

\begin{proof}
  Suppose that $L_{\MGL_2} X \simeq L_{\MGL_2} \nur(\b(X))$, so that $\MGL_2 \otimes X \simeq \MGL_2 \otimes \nur(\b(X))$.
  Then we have
  \[\Map_{\SHRat} (\Ss_2^{0,0,n}, \MGL_2 \otimes X) \simeq \Map_{\SHRat} (\Ss_2^{0,0,n}, \MGL_2 \otimes \nur(\b(X))\]
  as $C_2$-spectra over $\b(X)$.
  By \Cref{prop:Chow-MGL}(3), the map
  \[\Map_{\SHRat} (\Ss_2^{0,0,n}, \MGL_2 \otimes \nur(\b(X)) \to \b(\MGL_2 \otimes X)\]
  factors through an equivalence
  \[\Map_{\SHRat} (\Ss_2^{0,0,n}, \MGL_2 \otimes \nur(\b(X)) \simeq P_{2n} \b(\MGL_2 \otimes X),\]
  so that $\MGL_2 \otimes X$ is slice simple.

  Now suppose that $\MGL_2 \otimes X$ is slice simple, and
%
%
  consider the following diagram:
  \begin{center}
    \begin{tikzcd}
      X \ar[r] & X[\ta^{-1}] & \\
      \tau_{\geq 0} ^{\cn} X \ar[r] \ar[u] & \tau_{\geq 0} ^{\cn} (X[\ta^{-1}]). \ar[u] 
    \end{tikzcd}
  \end{center}
  If $X$ is a commutative algebra, then this is a diagram of commutative algebras.
  Applying $\Map_{\SHRat} (\Ss_2^{0,0,n}, - \otimes \MGL_2)$, we obtain the diagram:
  \begin{center}
    \begin{tikzcd}
      \Map_{\SHRat} (\Ss_2^{0,0,n}, X \otimes \MGL_2) \ar[r] & \b(X) \otimes \MU_{\R, 2} & \\
      \Map_{\SHRat} (\Ss_2^{0,0,n}, (\tau_{\geq 0} ^{\cn} X) \otimes \MGL_2) \ar[r] \ar[u] & \Map_{\SHRat} (\Ss_2^{0,0,n}, (\tau_{\geq 0} ^{\cn} (X[\ta^{-1}])) \otimes \MGL_2). \ar[u] 
    \end{tikzcd}
  \end{center}
  Finally, applying \Cref{prop:Chow-MGL}(3) and slice simplicity of $\MGL_2 \otimes X$, we may identify this diagram with
  \begin{center}
    \begin{tikzcd}
      P_{2n} (\b(X) \otimes \MU_{\R, 2}) \ar[r] & \b(X) \otimes \MU_{\R, 2} & \\
      P_{2n} (\b(X) \otimes \MU_{\R, 2}) \ar[r] \ar[u] & P_{2n} (\b(X) \otimes \MU_{\R, 2}). \ar[u] 
    \end{tikzcd}
  \end{center}
  In conclusion, we find that the maps $\tau_{\geq 0} ^{\cn} X \to X$ and $\tau_{\geq 0} ^{\cn} X \to \tau_{\geq 0} ^{\cn} (X[\ta^{-1}]) \simeq \nur(\b(X))$ are $\MGL_2$-equivalences, as desired.
%
%
%
%
%
%
\end{proof}

To prove \Cref{thm:recover}, we now need to show that the $\R$-motivic spectra in question are slice simple after being tensored with $\MGL_2$. To this end, we record the following simple proposition:
\begin{prop} \label{prop:sl-si-cond}
  The collection of slice simple Artin-Tate $\R$-motivic spectra is closed under direct sums, retracts, and pure suspensions $\Sigma^{n,n,n}$.
%
\end{prop}
%
%
%
%
\begin{proof}[Proof of \Cref{thm:recover}]
  By \Cref{prop:slice-simple-nu}, it suffices to show that $\MGL_2 \otimes E$ is slice simple for $E=\MGL_2, \HFt, \HZt, \kgl_2, \kq_2$. Indeed, each of these $E$ is $\MGL_2$-local, and $\nur(\b(E))$ is also $\MGL_2$-local by \Cref{thm:nu-properties}(6), since $\b(E)$ is $\MU_{\R, 2}$-local.

  Now, for $E=\MGL_2, \HFt, \HZt, \kgl_2$, $\MGL_2 \otimes E$ is a direct sum of pure suspensions of $E$, so it suffices to show that $E$ itself is slice simple.
  For $\MGL_2$, this follows from \Cref{prop:P-MGL}. For $\HZt$ and $\HFt$, it follows from \Cref{lem:Zcalc}.
  For $\kgl_2$, it follows from \Cref{thm:eff-fil} and \cite[Theorem 1.1]{SliceComp}.

%
  Finally, for the case of $\kq_2$, we use the equivalence $\kq_2 \otimes C\eta \simeq \kgl_2$ and the fact that $\eta = 0$ in $\pi_{0,1,1} ^\R \MGL_2$ to deduce that $\MGL_2 \otimes \kgl_2 \simeq \MGL_2 \otimes \kq_2 \otimes C\eta \simeq \MGL_2 \otimes \kq_2 \oplus \Sigma^{1,1,1} \MGL_2 \otimes \kq_2$.
  It follows that $\MGL_2 \otimes \kq_2$ is a retract of $\MGL_2 \otimes \kgl_2$, so that $\MGL_2 \otimes \kq_2$ is slice simple.
\end{proof}

\section{The $a$-local category} \label{sec:a-local}
In this section, we study the $a$-localized category, $\Mod(\SHR^{\at}_{i2} ; \Ss_2 [a^{-1}])$.
The authors find this category to be the most enduring mystery in our study of Artin--Tate motivic spectra over $\R$.
Some of the behavior of this category is the subject of \Cref{cnj:72} and \Cref{cnj:73}.

In the category of $C_2$-equivariant spectra,
inverting the class $a_{\sigma} \in \pi_{-\sigma}^{C_2} \Ss$ corresponds to taking geometric fixed points.
As a consequence, since $a = c_{\C/\R}(a_\sigma)$, we may interpret $\Mod(\SHR^{\at}_{i2} ; \Ss_2 [a^{-1}])$ as the natural target for some kind of ``motivic geometric fixed points'' operation.
Using $\ta$, we may also view the category $\Mod(\SHRat ; \Ss_2 [a^{-1}])$ as a deformation of $\Sp_{i2}$ with algebraic special fiber.
Since $\Phi^{C_2} \MU_{\R,2} \simeq \MO$, one might imagine that this degeneration is related to $\Syn_{\MO} \simeq \Syn_{\F_2}$. In \Cref{prop:a-inv-compare}, we will construct a comparison functor between $\Mod(\SHRat ; \Ss_2 [a^{-1}])$ and $\Syn_{\F_2}$. However, it is not an equivalence.

Another feature of the equivariant setting is the existence of the diagram,
\begin{center} \begin{tikzcd}
    & \Sp_{C_2} \ar[dr, "\Phi^{C_2}"] & \\
    \Sp \ar[ur, "\mathrm{Nm}_e^{C_2}"] \ar[rr, "\mathrm{Id}"] & & \Sp.
\end{tikzcd} \end{center}
This suggests we study the composite $\mathrm{Nm}_\C^\R(-)[a^{-1}]$.
Although the norm functor itself is not exact, in \Cref{lem:norm-invert} we observe that this composite preserves all colimits. However, it is not an equivalence.

Altogether, we conclude that the $a$-local category sits between the even $\BP$-synthetic category and the $\F_2$-synthetic category in a nontrivial way. In fact, although we are able to identify the three algebraic categories which arise as the special fibers, the maps between them are surprisingly nontrival. 

\subsection{The $a$-local category as a deformation}\

The symmetric monoidal category $\Mod(\SHR; \Ss[a^{-1}, \ta^{-1}])$ is equivalent to $\Sp$ by \cite{RealEtale} and \Cref{cor:rels}. As a consequence, there is an equivalence \footnote{Of course, this can also be deduced directly from \Cref{thm:comparison-invert-ta} and the equivalence $\Mod(\Sp_{C_2, i2} ; \Ss_2 [a_\sigma ^{-1}]) \simeq \Sp_{i2}$.}
\[ \Mod(\SHRat; \Ss_2[a^{-1}, \ta^{-1}]) \simeq \Sp_{i2}. \]
On the other hand, the following is an immediate corollary of \Cref{thm:Cta}.

\begin{cor}
  There is an equivalence of symmetric monoidal categories
  \[\Mod(\SHRat; C\ta[a^{-1}]) \simeq \Mod(\Sp_{i2}; \Phi^{C_2} \uZt) \otimes_\Z \IndCoh(\Mfg).\]
  On homotopy groups, this induces an isomorphism \footnote{Since $\abs{a} = (0,-1,0)$, the trigraded homotopy groups of an $a$-local $\R$-motivic spectrum are periodic in the $q$ degree. This is reflected by the fact that the given formula has no dependence on $q$.}
  \[ \pi_{p,q,w}^\R(C\ta[a^{-1}]) \cong \bigoplus_{w+2b-s = p} \left( \F_2\{u^{2b}\} \otimes_{\F_2} \Ext_{\BP_*\BP}^{s,2w}(\BP_*, \BP_*/2)\right). \]
\end{cor}

\begin{rec}
  There is an isomorphism $\pi_* \Phi^{C_2} \uZt \cong \F_2 [y]$, where $\abs{y} = 2$. Viewing $\Phi^{C_2} \uZt$ as $\uZt [a_{\sigma}^{-1}]$, $y$ may be identified with $\frac{u_{2\sigma}}{a_\sigma ^2} = \frac{u_{\sigma} ^2}{a_\sigma ^2}$.
\end{rec}

The presence of the non-nilpotent element $u$ in the homotopy of $C\ta[a^{-1}]$ prevents this category from being equivalent to either a $\BP$-synthetic category or an $\F_2$-synthetic category.

The simplest way to gain computational access to the $a$-local category seems to be the $\ta$-Bockstein spectral sequence.

\subsection{The $a$-local norm}\ 

Equivariantly, the composition of the norm functor with geometric fixed points is the identity.
As we shall see, over $\R$ the situation is not so straightforward.

\begin{lem}\label{lem:norm-invert}
  The composite
  \[ \SHC_{i2}^{\at} \xrightarrow{\mathrm{Nm}_\C^\R} \SHRat \xrightarrow{(-)[a^{-1}]} \Mod ( \SHRat ; \Ss_2[a^{-1}] ) \]
  is a symmetric monoidal left adjoint.
  On Picard elements this composite sends $\Ss_2^{s,w}$ to $\Ss_2^{s,*,2w}[a^{-1}]$.
  The map $\tau$ is sent to $\ta^2$.
\end{lem}

\begin{proof}
  The norm functor $\mathrm{Nm}_\C^\R$ is symmetric monoidal and commutes with sifted colimits \cite[Proposition 4.5]{norms}. Since we compose it with a symmetric monoidal left adjoint, in order to prove the first claim we need only show the composite preserves binary sums. From \cite[Corollary 5.13]{norms}, we have a formula
  \[ \mathrm{Nm}_\C^\R(X \oplus Y) \simeq \mathrm{Nm}_\C^\R(X) \oplus i_*(X \otimes Y) \oplus \mathrm{Nm}_\C^\R(Y), \]
  where $i$ is the inclusion $\R \to \C$. Since the functor $i_*$ lands in modules over the cofiber of $a$, the middle term vanishes upon inverting $a$, as desired.
  The claim about Picard elements is just a restatement of what we already know about Picard elements from \Cref{rec:norms}.

  Now we examine $\mathrm{Nm}_\C^\R(\tau) \in \pi_{0,0,-2}^\R \Ss_2$.
  Since the invert $\ta$ map
  \[\pi_{0,0,-2} ^{\R} \Ss_2 \to \pi_{0,0,-2} ^{\R} \Ss_2 [\ta^{-1}] \cong \pi_{0,0} ^{C_2} \Ss_2\]
  is an isomorphism which sends $\ta$ to $1$, it suffices to note that  \[\mathrm{Nm}_\C^\R(\tau)[\ta^{-1}] = \mathrm{Nm}_e^{C_2}((\tau)[\tau^{-1}]) = \mathrm{Nm}_e^{C_2}(1) = 1. \qedhere\]
\end{proof}

\begin{rmk}
  Despite the fact that upon inverting $\ta$ the functor $\mathrm{Nm}_\C ^\R [a^{-1}]$ becomes the identity on $\Sp_{i2}$, for degree reasons the class $\eta$ in $\pi_{1,1}^\C \Ss_2$ cannot map to the class $1 \otimes \alpha_{1}$ in the homotopy of $C\ta[a^{-1}]$ which one might expect detects $\eta$. \todo{This needs more details. $\Ext^1$ where?}
\end{rmk}

\begin{qst} \label{cnj:72}
Is there an equivalence $\mathrm{Nm}^{\R} _\C (C\tau) \simeq C\ta^2$ of commutative algebras? Postcomposing with the quotient map $C\ta^2 \to C\ta$, one obtains functors
\[ \Mod ( \SHCat ; C\tau) \to \Mod ( \SHRat ; C\ta^2 [a^{-1}]) \to \Mod ( \SHRat ; C\ta [a^{-1}]).\]
  Can the composite functor above be identified with the composite functor below?
\begin{center} \begin{tikzcd}
  \IndCoh(\Mfg)_{i2} \ar[d, "\F_2 \otimes -"] \ar[r] & \Mod(\Ab; \Phi^{C_2}\uZt ) \otimes_{\Zt} \IndCoh(\Mfg)_{i2} \\
    \mathrm{Vect}_{\F_2} \otimes_{\Zt} \IndCoh(\Mfg)_{i2} \ar[r, "\mathrm{Frobenius}"] & \mathrm{Vect}_{\F_2} \otimes_{\Zt} \IndCoh(\Mfg)_{i2} \ar[u, hook]
\end{tikzcd} \end{center}
\end{qst}

\subsection{Mapping down to $\F_2$-synthetic spectra}\

In this section, we will construct a realization functor from the $a$-local category to the category of $\F_2$-synthetic spectra. In some ways, this functor is even more surprising than the norm functor. While the norm functor seems to double degrees on the special fiber, \Cref{cnj:73} suggests that inverting $a$ cuts degrees in half.  

\begin{prop} \label{prop:a-inv-compare}
  There is a symmetric monoidal left adjoint
  \[ \Re_{\F_2} : \mathrm{Mod}( \SHRat ; \Ss_2 [a^{-1}] ) \to \Syn_{\F_2, i\tau}, \]
  which sends $\Ss_2^{p,q,w}$ to $\Ss_\tau^{p,w}$ and $\ta$ to $\tau$.
\end{prop}

\begin{proof}
  We begin by noting that under the equivalence
  $ \SHRat \simeq \Mod(\Sp_{C_2, i2} ^{\Fil}; R_\bullet) \simeq \Mod(\Sp_{C_2} ^{\Fil}; R_\bullet)$ from \Cref{thm:filt-model},
  inverting $a$ becomes levelwise geometric fixed points.
  Thus, $\Mod(\SHRat; \Ss_2 [a^{-1}])$ is equivalent to modules over the commutative algebra $\Phi^{C_2}R_\bullet$ 
  obtained by applying $\Phi^{C_2}$ levelwise. 
  Now recall that
  \[ R_\bullet \coloneqq \Tot^* \left( P_{2\bullet} \MU_{\R,2}^{* + 1} \right). \]
  We may then construct a chain of comparison maps
  \begin{align*}
    \Phi^{C_2} R_\bullet &\xrightarrow{\simeq}
    \Phi^{C_2} \Tot^* \left( P_{2\bullet} \MU_{\R,2}^{* + 1} \right) \to
    \Tot^{*} \Phi ^{C_2} \left( P_{2\bullet} \MU_{\R,2}^{* + 1} \right) \\
    &\to \Tot^{*} \tau_{\geq \bullet} \left( \Phi^{C_2}  \MU_{\R,2}^{* + 1} \right) \xrightarrow{\simeq}
    \Tot^{*} \tau_{\geq \bullet} \left( \MO^{* + 1} \right),
  \end{align*}
  where the key step is a use of the fact that 
  the geometric fixed points of a regular slice $2n$-connective $C_2$-spectrum are $n$-connective.
  This is used to produce the map across the line-break.

  In \Cref{prop:syn-to-fil}, we produced a symmetric monoidal equivalence
  \[\Mod \left(\Sp^{\Fil} ; \Tot^{*} \tau_{\geq \bullet} \left( \MO^{* + 1} \right) \right) \simeq \Syn^{\mathrm{cell}}_{\MO, i\tau}.\]
  Since the category of $\MO$-finite-projective spectra is equivalent to the category of $\F_2$-finite-projective spectra, and a map $X \to Y$ of spectra is an $\MO_*$-surjection if and only if it's an $(\F_2)_*$-surjection, there is a symmetric monoidal equivalence $\Syn_{\MO} \simeq \Syn_{\F_2}$. Finally, we can use the fact that $\Syn_{\F_2}^{\mathrm{cell}} \simeq \Syn_{\F_2}$ to drop the decoration. Altogether, base-change along this ring map produces a symmetric monoidal left adjoint
  \[ \mathrm{Mod}( \SHRat ; \Ss_2 [a^{-1}] ) \to \Syn_{\F_2, i\tau}. \qedhere \]
     
\end{proof}

Much of what was said in this section can be summarized with the existence of the following diagram of symmetric monoidal left adjoints:

\begin{center}
  \begin{tikzcd}
    \Sp_{i2} \ar[r, "\mathrm{Id}"] &
    \Sp_{i2} \ar[r, "\mathrm{Id}"] &
    \Sp_{i2}  \\
    \SHC_{i2} ^{\at} \ar[u, "{(-)[\tau^{-1}]}"'] \ar[d, " - \otimes C\tau"] \ar[r, "{(\mathrm{Nm}_\C^\R(-))[a^{-1}]}"'] &
    \SHRat \ar[u, "{(-)[\ta^{-1}]}"'] \ar[d, " - \otimes C\ta"] \ar[r, "\Re_{\F_2}"'] &
    \Syn_{\F_2, i\tau} \ar[u, "{(-)[\tau^{-1}]}"] \ar[d, " - \otimes C\tau"'] \\
    \IndCoh(\Mfg)_{i2} \ar[r] &
    \Mod(\Ab ; \Phi^{C_2} \uZ_2) \otimes_{\Zt} \IndCoh(\Mfg)_{i2} \ar[r] &
    \Mod ( \Syn_{\F_2} ; C\tau).
  \end{tikzcd}
\end{center}

The bottom right corner can be described in terms of comodules over the dual Steenrod algebra \cite[Section 4.5]{Pstragowski}.  The maps from the left to the right are comparison maps from the Adams--Novikov to the Adams spectral sequence.

\begin{qst} \label{cnj:73}
Is the map from the bottom middle to the bottom right induced by the map of Hopf algebroids $(\BP_*, \BP_*\BP) \to (\F_2, \A)$ which sends $t_i$ to $\zeta_i$? Note that this map is \emph{not} the usual Thom reduction map that sends $t_i$ to $\zeta_i^2$.
\end{qst}




\section{Completions} \label{sec:completion}
In this paper correctly handling a variety of different types of completion has been a key point.
The purpose of this section is to give a single uniform source of information on completeness questions and how we handle them.

In \Cref{sec:2vsi2complete}, we show that $2$-completion agrees with tensoring with $\Ss_2$ for some important $\R$-motivic spectra. Since these $\R$-motivic spectra are cellular, this implies that their $2$-completions in $\SHR$ lie in $\SHRat$ and that that their cell decompositions carry over unchanged to the $2$-completion.
In the case of $\MGL$, this is an important technical point in the rest of the paper.

In \Cref{sec:ta-comp}, we record the fact that every dualizable object of $\SHRat$ is $\ta$-complete.

Finally, in \Cref{sec:a-comp}, we make some remarks about $a$-completion.

\subsection{2 completion is sometimes i2 completion}\label{sec:2vsi2complete} \

The main goal of this section is to prove the following theorem:

\begin{thm}\label{thm:complete-compare}
  There are natural equivalences in $\SHR$:
  \begin{align*}
    &\MGL \otimes \Ss_2 \to \MGL_2 \\
    &\HZ \otimes \Ss_2 \to \HZ_2 \\
    &\kgl \otimes \Ss_2 \to \kgl_2 \\
    &\kq \otimes \Ss_2 \to \kq_2.
  \end{align*}
\end{thm}

The proof for $\MGL$, $\HZ$ and $\kgl$ will be completely independent of the rest of this paper. On the other hand, the proof for $\kq$ will use results from Sections \ref{sec:ta} through \ref{sec:top}.

We begin with some basic facts about the situation.
In the following, we will make use of the homotopy $t$-structure on the category of $\SHR$ of $\R$-motivic spectra.
We refer to the reader to \cite[\S 2.1]{HM} for basic facts about the homotopy $t$-structure.

\begin{prop} \label{prop:compl-facts}
  Given $X \in \SHR$, the following statements are true:
  \begin{enumerate}
    \item If $X$ is dualizable, then
  \[X \otimes \Ss_2 \to X_2\]
  is an equivalence.

%
\item Suppose that $\varinjlim X(n) \simeq X$ and that the connectivity of the maps $X(n) \to X$ tend to infinity with respect to the homotopy $t$-structure. Then if
  \[X(n) \otimes \Ss_2 \to X(n)_2\]
  is an equivalence for all $n$, we also learn that
  \[X \otimes \Ss_2 \to X_2\]
  is an equivalence.
  \end{enumerate}
\end{prop}

\begin{proof}
  To prove (1), it simply suffices to note that since $X$ is dualizable, the functor $- \otimes X$ preserves limits.

%
  To prove (2), we note that, since tensoring with $\Ss_2$ preserves colimits, it is equivalent to show that
  \[\varinjlim X(n)_2 \to X_2\]
  is an equivalence, i.e. that
  \[\varinjlim (X_2 / X(n)_2) \simeq \varinjlim (X/X(n))_2 \simeq 0 .\]

  Since $\{\Ss^{p,w,w} \otimes S_+ \vert S \in \Sm/\R \text{ and } p,w \in \Z\}$ forms a set of compact generators for $\SHR$, it suffices to show that
  \[\Map_{\SHR} (\Ss^{p,w,w} \otimes S_+, \varinjlim(X/X(n))_2 ) \simeq \varinjlim \left( \Map_{\SHR} (\Ss^{p,w,w} \otimes S_+, X/X(n))_2 \right)\]
  is equivalent to $0$ for all $S \in \Sm/\R$ and $p,w \in \Z$.

  This follows from the definition of the homotopy $t$-structure, which implies that, for fixed $S \in \Sm/\R$ and $p,w \in \Z$, the connectivity of the spectrum
  \[\Map_{\SHR} (\Ss^{p,w,w} \otimes S_+, X/X(n))\]
  goes to infinity as $n$ goes to infinity.
%
\end{proof}

The key input to our proof of \Cref{thm:complete-compare} will be the following:

\begin{lem}\label{lem:cof-conn}
  Let $\MGL(n,k)$ denote the Thom spectrum of the bundle $\gamma - k \cdot \triv$ over $\mathrm{Gr}_k (\mathbb{A}^{n+k})$.
  Then the cofiber of the natural map $\MGL(n,n) \to \MGL$ is $n$-connective with respect to the homotopy $t$-structure.
\end{lem}

\begin{proof}
  This is an immediate consequence of \cite[Lemma 3.4]{HM}.
%
\end{proof}
%
%
%
%
%
%
%
%
%
\begin{proof}[Proof of \Cref{thm:complete-compare} for $\MGL$, $\HZ$ and $\kgl$]
  Since $\MGL(n,k)$ is dualizable, the result for $\MGL$ follows from the combination of \Cref{prop:compl-facts}(1,2) with \Cref{lem:cof-conn}.

  We now turn to the case of $\HZ$.
  By the Hopkins--Morel theorem \cite{HM}, there is an equivalence
  \[ \MGL / (a_1, a_2, \dots) \simeq \HZ \]
  for certain classes $a_i \in \pi_{i,i,i} ^\R \MGL$.
  In other words,
  \[\varinjlim \MGL / (a_1, a_2, \dots, a_k) \simeq \HZ.\]
  By the result for $\MGL$, we know that the result holds for all $\MGL / (a_1, a_2, \dots, a_k)$. By \Cref{prop:compl-facts}(2), it therefore suffices to show that the connectivity of
  \[\MGL / (a_1, a_2, \dots, a_k) \to \MGL / (a_1, a_2, \dots) \simeq \HZ\]
  goes to infinity.
  To do this, it suffices to show that the connectivity of
  \[\MGL / (a_1, a_2, \dots, a_k) \to \MGL / (a_1, a_2, \dots, a_k, a_{k+1})\]
  goes to infinity, which follows from the cofiber sequences
  \[\MGL / (a_1, a_2, \dots, a_k) \to \MGL / (a_1, a_2, \dots, a_k, a_{k+1}) \to \Sigma^{k+2, k+1,k+1} \MGL / (a_1, a_2, \dots, a_k).\]

  Finally, we handle the case of $\kgl$.
  Combining the Hopkins--Morel theorem and convergence of the effective slice towers for $\MGL$ and $\kgl$ \cite[Lemmas 8.10 and 8.11]{HM} with \cite[Proposition 5.4]{Spitzweck}, we find that there is an equivalence
  \[\MGL / (a_2, a_3, \dots) \simeq \kgl. \]
  We may therefore prove the result for $\kgl$ exactly as we did for $\HZ$.
\end{proof}

Before we move on to the case of $\kq$, we collect some facts that we will need:
\begin{prop} \label{prop:completion-useful}
  The following are true:
  \begin{enumerate}
    \item Suppose that $X \in \SHRat$ is effective slice connective. Then the natural map
      \[\Map_{\SHR} (\Ss^{0,0,n}, X) \to \b (X)\]
      is an equivalence for any $n \leq 0$.
    \item In $C_2$-equivariant spectra, there is a canonical equivalence
      \[\Phi^{C_2} (\ko_{C_2, 2}) \simeq \Phi^{C_2} (\ko_{C_2})_2.\]
  \end{enumerate}
\end{prop}

\begin{proof}
  Given $X \in \SHRat$, it follows from \Cref{thm:comparison-invert-ta}
  that $\b(X) \simeq \varinjlim \Map_{\SHR} (\Ss^{0,0,n}, X)$.
  Part (1) therefore follows from \Cref{thm:eff-fil}.
%

  We now prove (2). By the cofiber sequence
  \[(\ko_{C_2})_{hC_2} \to (\ko_{C_2})^{C_2} \to \Phi^{C_2} (\ko_{C_2})\]
  and the fact that $(-)^{C_2}$ commutes with limits, it suffices to show that the canonical map
  \[(\ko_{C_2,2})_{hC_2} \to (\ko_{C_2})_{hC_2, 2}\]
  is an equivalence.
  This map is equivalent to
  \[\ko_{2} \otimes \RP^{\infty}_{+} \to (\ko \otimes \RP^{\infty} _+ )_2, \]
  so this follows from connectivity of $\ko$ and the fact that $\RP^\infty _+$ is of finite type. 
%
\end{proof}

The key input will be the following special case:

\begin{prop} \label{prop:2-eta-invert-win}
  The map
  \[(\kq \otimes \Ss_2)[1/2, 1/\eta] \to \kq_2 [1/2, 1/\eta]\]
  is an equivalence.
\end{prop}

This will make essential use of the following result of Bachmann.

\begin{prop} \label{prop:invert-2-eta-rho}
  The assignment $X \mapsto X(\R)$ induces an equivalence
  \[\SHR[1/\rho] \simeq \Sp.\]
  Moreover, we have
  \[\SHR_{(2)} [1/2, 1/\eta] = \SHR_{(2)} [1/2, 1/\rho] \simeq \Mod_\Q.\]
\end{prop}

\begin{proof}
  The first statement follows from \cite[Theorem 35 and Proposition 36]{RealEtale}.
  The second statement is a consequence of the first and \cite[Lemma 39]{RealEtale}.
\end{proof}

Given \Cref{prop:invert-2-eta-rho}, \Cref{prop:2-eta-invert-win} is reduced to the following proposition:

\begin{prop}
  There are isomorphisms:
  \begin{align*}
    &\pi_{*,0,0} ^{\R} (\kq_{(2)}[1/2, 1/\eta]) \cong \Q[\beta], \abs{\beta} = 4 \\
    &\pi_{*,0,0} ^{\R} (\Ss_2 [1/2, 1/\eta]) \cong \Q_2 \\
    &\pi_{*,0,0} ^{\R} (\kq_2 [1/2, 1/\eta]) \cong \Q_2 [\beta], \abs{\beta} = 4.
  \end{align*}
\end{prop}

\begin{proof}
  Let $\mathrm{W} (\R)$ denote the Witt group of the real numbers. There is an isomorphism $\mathrm{W} (\R) \cong \Z$.
  It follows from \cite[Section 6.3.2]{etaPeriodic} that
  \[\pi_{*,0,0} ^{\R} \kq[1/\eta] \cong \mathrm{W} (\R) [\beta] \cong \Z[\beta],\]
  from which we deduce the first isomorphism.

  By \Cref{thm:comparison-invert-ta}, Betti realization induces an isomorphism 
  \[\pi_{*,*,0} ^{\R} (\Ss_2)[1/\ta] \cong \pi_{*+*\sigma} ^{C_2} (\Ss_2).\]
  As a consequence, we find
  \begin{align*}
    \pi_{*,0,0} ^{\R} (\Ss_2)[1/\rho]
    &\cong \pi_{*,0,0} ^{\R} (\Ss_2)[1/\ta, 1/a] \\
    &\cong \pi_{*} ^{C_2} (\Ss_2)[1/a_\sigma] \\
    &\cong \pi_* \Phi^{C_2} (\Ss_2) \\
    &\cong \pi_* \Ss_2.
  \end{align*}
  Inverting $2$ and applying \Cref{prop:invert-2-eta-rho}, we deduce the second isomorphism.

  For the third isomorphism, we make use of the sequence of equivalences
  \begin{align*}
    \Map_{\SHR} (\Ss^{0,0,n}, \kq_2)
    &\simeq \varprojlim \Map_{\SHR} (\Ss^{0,0,n}, \kq/2^k) \\
    &\simeq \varprojlim \b(\kq/2^k) \\
    &\simeq \b(\kq)_2 \\
    &\simeq \ko_{C_2,2},
  \end{align*}
  where the second equivalence follows from \Cref{prop:completion-useful}(1) and the fourth equivalence follows from \cite[Corollary 2.30]{Kong}.

  As a consequence, we find that
  \begin{align*}
    \pi_{*,0,0} ^\R \kq_2 [1/\rho] 
    &\cong \pi_{*,0,0} ^\R \kq_2 [1/\ta, 1/a] \\
    &\cong \pi_{*} ^{C_2} \ko_{C_2,2} [1/a_\sigma] \\
    &\cong \pi_{*} \Phi^{C_2} (\ko_{C_2,2}) \\
    &\cong \pi_* \Phi^{C_2} (\ko_{C_2})_2 \\
    &\cong \Z_2 [\beta],
  \end{align*}
  where the fourth isomorphism follows from \Cref{prop:completion-useful}(2) and the fifth isomorphism follows from \cite[Propositon 10.18]{koC2}.
  Inverting $2$ and applying \Cref{prop:invert-2-eta-rho}, we obtain the result.
\end{proof}

\begin{proof}[Proof of \Cref{thm:complete-compare} for $\kq$]
  We want to prove that
  \[\kq \otimes \Ss_2 \to \kq_2\]
  is an equivalence.
  It is clearly an equivalence after smashing with $C(2)$, so it suffices to show that it is an equivalence after inverting $2$.
  Moreover, by the equivalence $\kq \otimes C(\eta) \simeq \kgl$ and the $\kgl$ case, it is also an equivalence after smashing with $C(\eta)$.
  It therefore suffices to show that
  \[(\kq \otimes \Ss_2)[1/2, 1/\eta] \to \kq_2 [1/2, 1/\eta]\]
  is an equivalence, which is \Cref{prop:2-eta-invert-win}.
%
%
%
%
%
 %
%
%
 %
\end{proof}

\subsection{Ta completion} \label{sec:ta-comp}\

At several points we have suggested that from a computational viewpoint the $\ta$-Bockstein spectral sequence is probably the best way to gain access to various objects. Here we show such computations often converge to the correct answer.

\begin{prop}\label{prop:unit-ta-complete}
 The unit in $\SHRat$ is $\ta$-complete.
\end{prop}

Before proving this proposition we give a corollary of it.

\begin{cor}\label{cor:dualizable-ta-complete}
  Every dualizable object of $\SHgc$ is $\ta$-complete.
\end{cor}

Now we turn to the proof of this proposition.
We begin with the following lemma.

\begin{lem} \label{lem:Z-ta-complete}
  $\HZt$ is $\ta$-complete.
\end{lem}

\begin{proof}
  Since the tri-graded spheres are a family of compact generators it will suffice to show that
  \[ \varprojlim \left( \cdots \to \pi_{p,q,2} ^\R (\HZt) \to \pi_{p,q,1} ^\R (\HZt) \to \pi_{p,q,0} ^\R (\HZt) \right) \]
  is zero and there is no $\mathrm{lim}^1$ term. Both of these claims follow from the fact that $\pi_{p,q,w}^\R (\HZt) = 0$ for $w \geq 0$.
\end{proof}

\begin{proof}[Proof of \Cref{prop:unit-ta-complete}]
  Since the $\HZt$-Adams spectral sequence converges \cite{MAdamsConv}  and limits of complete objects are complete we may use \Cref{lem:Z-ta-complete} to conclude.
\end{proof}

\subsection{$a$-completion} \label{sec:a-comp}\ 

Almost none of the objects we have encountered are $a$-complete.  This stands in contrast to $C_2$-equivariant homotopy theory, where the Segal conjecture implies that the $2$-completed unit is $a$-complete. To see explicit failures of completeness, note that if an object $X$ is $a$-complete then $C\ta \otimes X $ is also $a$-complete. On the other hand, the homotopy groups of $C\ta$ contain copies of the homotopy of $\uZt$, which is not $a$-complete.

\begin{exm}
  The elements $\frac{\theta}{a^n} \otimes 1$ in the homotopy of $C\ta$ give an explicit example of non-$a$-completeness.
  If we consider the image of $\frac{\theta}{a^n}$ under the connecting map $\delta: \Sigma^{-1,0,1}C\ta \to \Ss$ we obtain an example of an infinitely $a$-divisible element in the sphere. \footnote{This image is non-trivial for $n$ sufficiently large, though we don't prove it. Of course if the image is trivial for all $n$, then a lift of these classes to the sphere gives a different example of non-completeness.}    
\end{exm}

On the other hand,
it often happens that after inverting $\ta$ objects becomes $a$-complete.

\begin{lem}\label{lem:using-lin}
  On dualizable objects of $\SHRat$ the following functors are equivalent:
  $(-)[\ta^{-1}]$, $(-)_a^\wedge[\ta^{-1}]$ and $(-)[\ta^{-1}]_a^\wedge$.
\end{lem}

\begin{proof}
  For any dualizable object $X$ there is an $n$ such that $\pi_{p,q,w} ^\R (C\ta \otimes X) = 0$ for $w \leq n$. Using this we learn that on the level of trigraded homotopy groups the colimit inverting $\ta$ commutes with the inverse limit completing at $a$. Now, we only need to explain why $(-)[\ta^{-1}]$ is already $a$-complete on dualizable inputs. If we invert $\ta$ on a dualizable $i2$-complete $\R$-motivic spectrum we get a dualizable $i2$-complete $C_2$-spectrum. Thus, we're reduced to showing that dualizable $i2$-complete $C_2$-spectra are $a_\sigma$-complete. Using the fact that tensoring with a dualizable object commutes with limits it suffices to show this is true for the unit. The 2-complete sphere in $C_2$-spectra is $a_\sigma$-complete as a consequence of Lin's theorem \cite{Lin}.
\end{proof}

The tension this introduces is heightened when one recalls Gregersen's motivic analog of Lin's theorem, which is proved over any field of characteristic zero \cite{gregersen}. However, Gregersen's result only asserts a $\pi_{**}$-isomorphism, i.e. an equivalence after reflecting into the Tate category. 
In fact, a careful examination of the formula for the homotopy groups of $C\ta$ reveals that the copies of the negative cone (from which the non-$a$-completeness originates) all lie below the plane of Tate spheres. Therefore, upon running the $\ta$-Bockstein spectral sequence the tri-degrees $(p,w,w)$ never receive contributions from copies of the negative cone. 



\section{Odd primes} \label{sec:odd}
In this section we show that, when working at an odd prime $p$, the category $\SHR^{\at}_{ip}$ admits a simple description in terms of $\SHCat$. The corresponding result for Tate spectra is discussed in \cite[p.23-24]{BehrensShah}, following results of \cite{RealEtale}.

\begin{prop}\label{prop:odd}
  There is an equivalence of categories
  \[ \SHR^{\at}_{ip} \simeq \Sp_{ip} \times \SHCat \times \SHCat. \]
\end{prop}

We begin with a well-known decomposition of $\SHR^{\at} _{ip}$ into its $+$ and $-$ components. Our treatment of this is heavily inspired by \cite[pp.23-25]{BehrensShah}

\begin{dfn}
  Let $\SHR_{ip}^{\at,+}$ denote the category $\Mod ( \SHR^{\at} ; \Ss_{p,\eta} )$. 
\end{dfn}

\begin{lem}
  There is an equivalence of categories
  \[\SHR^{\at} _{ip} \simeq \SHR^{\at,+} _{ip} \times \Sp_{ip}.\]
\end{lem}

\begin{proof}
  Let $\epsilon = -\eta \rho - 1 \in \pi_{0,0,0} ^{\R} \Ss$, which is known to be equal to the swap map
  \[\epsilon = s_{\Ss^{1,1,1}} : \Ss^{1,1,1} \otimes \Ss^{1,1,1} \xrightarrow{\sim} \Ss^{1,1,1} \otimes \Ss^{1,1,1}. \]
  It is clear from this description that $\epsilon^2 = 1$, so that after inverting $2$ we have a decomposition of the unit into $\pm 1$-eigenspaces
  \[\Ss [1/2] \simeq \Ss [1/2]^+ \times \Ss[1/2]^-.\]
  Now, $\epsilon = +1$ implies that $\eta\rho = -2$, so that $\eta$ and $\rho$ must be units in $\pi_{*,*,*} ^{\R} \Ss[1/2]^+$.
  On the other hand, the fact that $\epsilon = s_{\Ss^{1,1,1}}$ implies that the graded ring
  $\pi_{a,a,a}^{\R} \Ss[1/2]^{-}$
  must obey the Koszul sign rule, so that $2 \eta^2 = 2 \rho^2 = 0$.

  As a consequence, we find that $\Ss[1/2]^{+} \simeq \Ss[1/2,1/\eta] \simeq \Ss[1/2,1/\rho]$ and that $\Ss[1/2]^- \simeq \Ss[1/2]_\eta \simeq \Ss[1/2]_\rho$.
  To conclude, we base change from $\Ss[1/2]$ to $\Ss_p$ and use the equivalence $\SHR[1/\rho] \simeq \Sp$ of Bachmann \cite{RealEtale}.
\end{proof}
%
%
%

\Cref{prop:odd} will now follow from showing that $\SHR_{ip}^{\at,+}$ splits as a product of two copies of $\SHCat$. This splitting will come from a decomposition of the category into two blocks (in the sense of representation theory). We begin with the following pair of lemmas.

\begin{lem} \label{lem:odd-split}
  In $\SHR_{ip}^{\at,+}$, there is a splitting $\Spec(\C) \simeq \Ss_{p,\eta}^{0,0,0} \oplus \Ss_{p,\eta}^{1,-1,0}$.
\end{lem}

\begin{proof}
  It suffices to show that $a : \Ss_{p,\eta}^{0,-1,0} \to \Ss_{p,\eta}$ is zero.
  This follows from the facts that $c_{\C/\R}$ is fully-faithful (see \cite{HellerOrmsbyII}) and that $a_{\sigma}=0$ in the $C_2$-equivariant category after inverting $2$ and $\eta$-completing (indeed, $a_{\sigma} \eta = 2$).
\end{proof}

\begin{lem} \label{lem:odd-vanishing}
  The groups $\pi_{s,q,w}^\R \Ss_{p,\eta}$ are zero for $q$ odd.
\end{lem}

The proof of this lemma will take us farther afield, so we defer it for the moment.

\begin{proof}[Proof of \Cref{prop:odd}.]
  We need to show that
  \[ \SHR_{ip}^{\at,+} \simeq \SHCat \times \SHCat. \]
  Let $\A_{\mathrm{even}}$ (resp. $\A_{\mathrm{odd}}$) denote the stable full subcategory of $\SHR_{ip}^{\at,+}$ generated under colimits by the objects $\Ss_{p,\eta}^{s,q,w}$ for $s,w \in \Z$ and $q$ even  (resp. odd).
  
  We begin by showing that if $A \in \A_{\mathrm{even}}$ and $B \in \A_{\mathrm{odd}}$, then there are no nontrivial maps between $A$ and $B$. It suffices to show this for compact generators, where it follows from \Cref{lem:odd-vanishing}. From this we may conclude that $\SHR_{ip}^{\at,+} \simeq \A_{\mathrm{even}} \times \A_{\mathrm{odd}}$.

  Since tensoring with $\Ss^{0,1,0}$ provides an equivalence between $\A_{\mathrm{even}}$ and $\A_{\mathrm{odd}}$ it will now suffice to show that the composite
  \[ \A_{\mathrm{even}} \to \SHR^{\at,+}_{ip} \to \SHCat \]
  is an equivalence. Since $\Ss^{s,0,w}$ is sent to $\Ss^{s,w}$ we know this map hits a family of compact generators, so it will suffice to show that it is fully faithful. By the usual argument it will suffice to have fully-faithfulness on compact generators. Thus, we are reduced to showing that the map
  \[ \pi_{s,q,w}^\R\Ss_{p,\eta} \to \pi_{s+q,w}^\C\Ss_{p,\eta}. \]
  is an isomorphism for $q$ even. Using \Cref{lem:odd-split} this map factors as
  \[ \pi_{s,q,w}^\R\Ss_{p,\eta} \hookrightarrow \pi_{s,q,w}^\R( \Ss_{p,\eta}^{0,0,0} \oplus \Ss_{p,\eta}^{1,-1,0} ) \xrightarrow{\cong} \pi_{s+q,w}^\C\Ss_{p,\eta}, \]
  where the first map is the inclusion of the left summand.
  We conclude by noting that the relevant homotopy group of the right summand vanishes by \Cref{lem:odd-vanishing}.
\end{proof}

We now return to proving \Cref{lem:odd-vanishing}.
The proof will be via an Adams spectral sequence argument, so we begin by computing the homology of a point. Note that $\MFp$ is $\eta$-complete, and so inside of $\SHR_{ip}^{\at,+}$.

\begin{lem}\label{lem:odd-pt}
  The tri-graded homotopy of $\MFp$ is given by
  \[ \pi_{s,q,w}^\R \MFp \cong \F_p[u_{2\sigma}^{\pm}, \ta], \]
  where $|u_{2\sigma}| = (2,-2,0)$.
\end{lem}

\begin{proof}
  From the computation of the homology of a point over $\C$ in \cite{Voe03} we may conclude that $\pi_{s,q,w}^\R(\Spec(\C) \otimes \MFp) \cong \F_p[u_\sigma^{\pm}, \ta]$. Using the splitting of $\Spec(\C)$ we may conclude that the homology of a point over $\R$ is an index 2 subalgebra which contains $\ta$ (since $\ta$ is defined in the sphere). There is a unique such subalgebra.
\end{proof}

\begin{lem}\label{lem:odd-dual-enough}
  The tensor product $\MFp \otimes \MFp$ splits as a sum of tri-graded suspensions of copies of $\MFp$ whose $q$-components are even.
\end{lem}

\begin{proof}
  From \cite[Theorem 1.1 (3)]{HoyoisSteenrod} (which follows from work of Voevodsky \cite{Voe03, Voe10} in case of $\R$ where we work) we know that $\MFp \otimes \MFp$ decomposes as a sum of copies of $\Sigma^{a,b,b} \MFp $. More specifically, the copies of $\MFp$ are indexed by monomials in the $\xi_i$ and $\tau_i$ which live in degrees,
  \[ |\xi_i| = (p^i - 1, p^i - 1, p^i - 1) \qquad \text{ and } \qquad |\tau_i| = (p^i, p^i - 1, p^i - 1). \]
  Since $p$ is odd the $q$-component of every such monomial is even.
\end{proof}

\begin{proof}[Proof of \Cref{lem:odd-vanishing}.]
  Since the motivic Adams spectral sequence for the sphere converges strongly to $\Ss_{p,\eta}$ by \cite{MAdamsConv},
  it will suffice to show the desired vanishing result on the $\mathrm{E}_1$-page.  
  This spectral sequence takes the form,
  \[ E_1^{s,t} = \pi_{t,q,w}(\HFt^{\otimes {s+1}}) \implies \pi_{t-s,q,w} \Ss_{p,\eta}. \]
  Using \Cref{lem:odd-pt} and \Cref{lem:odd-dual-enough} we may conclude that the spectral sequence is zero at the $\mathrm{E}_1$ page for $q$ odd.
\end{proof}

\section{Examples and computations} \label{sec:computations}
In this section we employ technology from the rest of the paper to give example, concrete computations of trigraded homotopy groups.  First, we discuss vanishing regions in the homotopy groups of objects obtained from the sphere by either killing or inverting $a$ and $\ta$.  Then we give a detailed discussion of $\kq_2$, which we hope will be a useful guide to computationally minded stable homotopy theorists.

\subsection{Vanishing regions for trigraded homotopy}\

The coarsest information about the trigraded homotopy groups of Artin--Tate $\R$-motivic spectra comes from understanding which regions contain only zero groups. We summarize what we know in the following omnibus theorem.

\begin{thm}\label{thm:omnibus-vanishing}
  By either coning or inverting $a \in \pi_{0,-1,0}^{\R} \Ss$ and $\ta \in \pi_{0,0,-1}^{\R} \Ss_2$, we may build $9$ natural objects. The trigraded homotopy groups of these objects are concentrated in the following regions.
  \begin{enumerate}
  \item $\pi_{p,q,w}^\R (\Ss_2)$
    is concentrated in the region,
    \[ \{ p+q \geq w \geq 0 \} \cup \{p+q \geq 0, w \leq 0\} \cup \{ p \geq 0 \}. \]

  \item $\pi_{p,q,w}^\R (C\ta)$
    is concentrated in the region,
    \[ \{ 0 \leq p \leq 2w-q, w \geq 0 \} \cup \{ w-q \leq p \leq w-2, w \geq 0 \}. \]

  \item $\pi_{p,q,w}^\R (\Ss_2[\ta^{-1}])$
    is periodic in the $w$-degree and concentrated in the region,
    \[ \{ p \geq 0 \} \cup \{ p+q \geq 0 \}. \]

  \item $\pi_{p,q,w}^\R (Ca)$
    is periodic along lines of the form $(1,-1,0)$ and concentrated in the region,
    \[ \{ 0 \leq w \leq p+q \} \cup \{ 0 \leq p+q, w \leq 0 \}. \]
    
  \item $\pi_{p,q,w}^\R (Ca \otimes C\ta)$
    is periodic along lines of the form $(1,-1,0)$ and concentrated in the region,
    \[ \{ w \leq p+q \leq 2w \}. \]
    
  \item $\pi_{p,q,w}^\R (Ca \otimes \Ss_2[\ta^{-1}])$
    is periodic in the $w$-degree, periodic along lines of the form $(1,-1,0)$ and concentrated in the region,
    \[ \{p+q \geq 0\}. \]
    
  \item $\pi_{p,q,w}^\R (\Ss_2 [a^{-1}])$
    is periodic in the $q$-degree and concentrated in the region,
    \[ \{ p \geq 0 \}. \] 

  \item $\pi_{p,q,w}^\R (\Ss_2 [a^{-1}] \otimes C\ta)$
    is periodic in the $q$-degree and concentrated in the region,
    \[ \{ p \geq 0, w \geq 0\}. \]
    
  \item $\pi_{p,q,w}^\R (\Ss_2 [a^{-1}, \ta^{-1}])$
    is periodic in the $q$-degree, periodic in the $w$-degree and concentrated in the region,
    \[ \{ p \geq 0 \}. \]

  \end{enumerate}
\end{thm}

\begin{proof}
	We begin with (4), (5) and (6), where we have killed $a$.  In the identification of $Ca$-modules with the $\C$-motivic category, the trigraded homotopy group $\pi_{p,q,w} Ca$ is identified with the bigraded stem $\pi_{p+q,w}\Ss_2^\C$. Hence, $\pi_{p,q,w} Ca \cong \pi_{p+m,q-m,w} Ca$ for any integer $m$.  That $\pi_{p,q,w} Ca$ vanishes outside the indicated regions then reduces to well-known vanishing statements for the $\C$-motivic stable stems, as displayed on \cite[p.2]{GI}.  Similarly, again as indicated on \cite[p.2]{GI}, (5) and (6) are consequences of \cite[Theorem 6.7]{LevineComparison}.  Analogously, statements $(3)$ and $(9)$ follow from combining \Cref{thm:comparison-invert-ta} with the known vanishing regions in $C_2$-equivariant stable stems (for such vanishing regions, see e.g. \cite[Figure 5]{c2stems}).
	
	Upon inverting $a$ in (1) and (2) we recover (7) and (8), respectively, so it remains only to prove (1) and (2).  In fact, since \Cref{prop:unit-ta-complete} proves that $\Ss_2$ is $\ta$-complete, we may obtain (1) from (2) by examining the $\mathrm{E}_1$-page of the $\ta$-Bockstein spectral sequence. Finally, in order to prove (2) we directly examine the homotopy of $C\ta$.
	  In \Cref{thm:Cta} we computed that  
  \[ \pi_{p,q,w}^\R(C\ta) \cong \bigoplus_{w+a-s = p} \Ext_{(\MU_2)_* \MU_2 }^{s,2w}((\MU_2)_*, (\MU_2)_* \otimes_{\Z_2} \pi_{a + (q-w) \sigma}^{C_2} \uZt). \]
  The result then follows by combining the following vanishing results:
  \begin{itemize}
    \item $\Ext_{(\MU_2)_* \MU_2 }^{s,2w} ((\MU_2)_*, (\MU_2)_*)$ and $\Ext_{(\MU_2)_* \MU_2 }^{s,2w} ((\MU_2)_*, (\MU_2)_*/2)$ are concentrated in the region $\{ 0 \leq s \leq 2w \}$.
  \item $\pi_{p + q\sigma} ^{C_2} \uZt$ is concentrated in the region $\{ p \geq 0, p+q \leq 0\} \cup \{ p \leq -2, p+q \geq 0 \}$.
  \end{itemize}



\end{proof}

\begin{rmk}
  As a corollary of the vanishing region for $C\ta$ and the $\ta$-completeness of the unit, we recover \cite[Theorem 1.1]{c2stems}, which describes the region in which $\b : \pi_{p,w,w}^\R \Ss_2 \to \pi_{p+w\sigma}^{C_2} \Ss_2$ is an isomorphism.
\end{rmk}
  
 \subsection{The homotopy of $\kq_2$}\
 To aid the computationally minded reader, we give below some charts of the trigraded homotopy groups of $\mathrm{kq}_2$.
These charts are meant to be read in concert with those of \cite[\S 6.5]{Kong}, and the notation here matches that in Kong's work.
Indeed, all of the information within these charts is easily accessed within Kong's work, and we merely repackage it here.

We focus on the homotopy groups $\pi_{p,q,w} (\mathrm{kq}_2 \otimes C\ta)$ and $\pi_{p,q,w}(\mathrm{kq}_2)$ for $p=0$ and $p=-1$.
In the language of \cite{Kong}, this corresponds to a focus on \emph{coweights} $0$ and $-1$.

\begin{figure}[ht]
\centering
\includegraphics[trim=100 470 100 100, clip, width=\textwidth, scale=0.6]{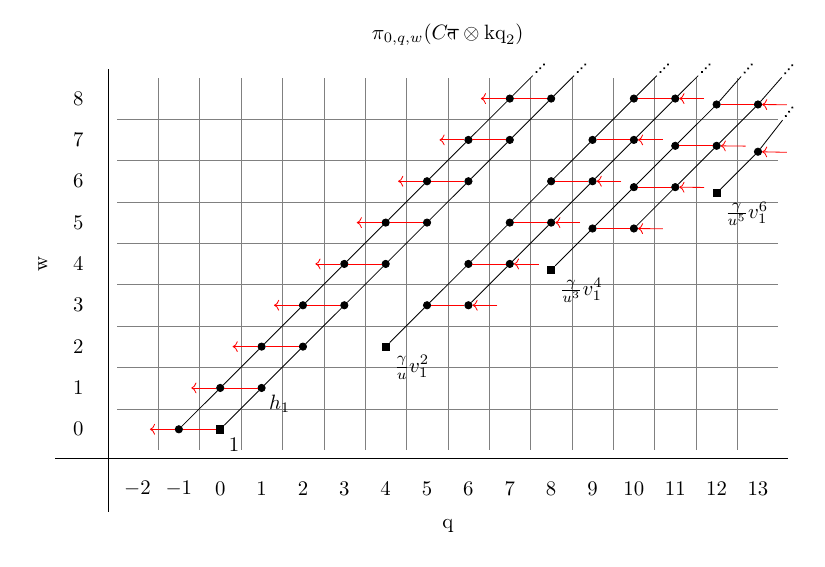}
\end{figure}

Red lines denote multiplication by $a \in \pi_{0,-1,0} \Ss$.  Black lines denote multiplication by the motivic $\eta:\mathbb{G}_m \to \Ss$, which in our grading convention lives in $\pi_{0,1,1} \Ss$.

\newpage

\begin{figure}[ht]
\centering
\includegraphics[trim=100 470 100 120, clip, width=\textwidth, scale=0.6]{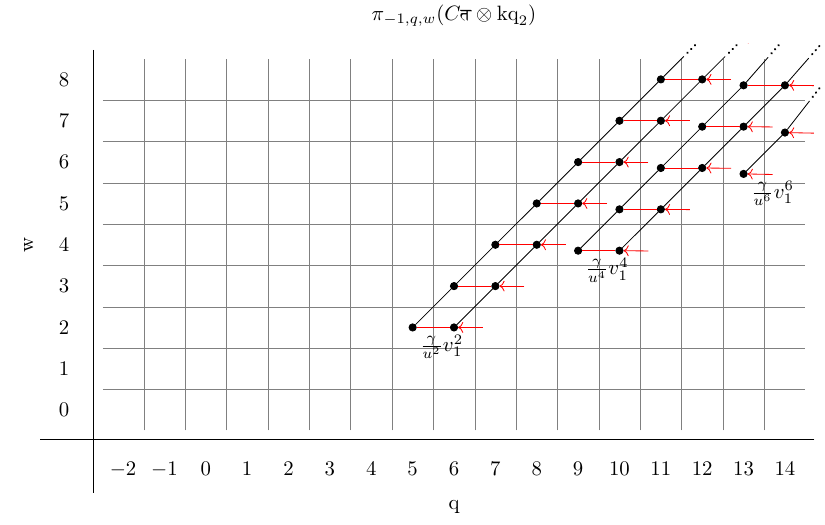}
\end{figure}

Below, we depict some of the $\mathrm{E}_{\infty}$-page of the $\ta$-Bockstein spectral sequence for $\pi_{*,*,*} \mathrm{kq}_2$.
We follow the convention introduced in \cite[\S A.2]{Boundaries} and \cite{BurklundExtension} of using blue symbols to denote $\ta$-torsion classes, while black symbols denote $\ta$-torsion free classes.
In particular, black symbols contribute not only to the homotopy group corresponding to the box in which they appear, but also to the groups corresponding to boxes directly below where they appear.
In the language of \cite{Kong}, the blue dots in our charts contain information about which dots on the $\mathrm{E}_1$-page of the $C_2$-effective spectral sequence are targets of differentials, as opposed to sources of differentials.

\begin{figure}[ht]
\centering
\includegraphics[trim=100 470 100 100, clip, width=\textwidth, scale=0.6]{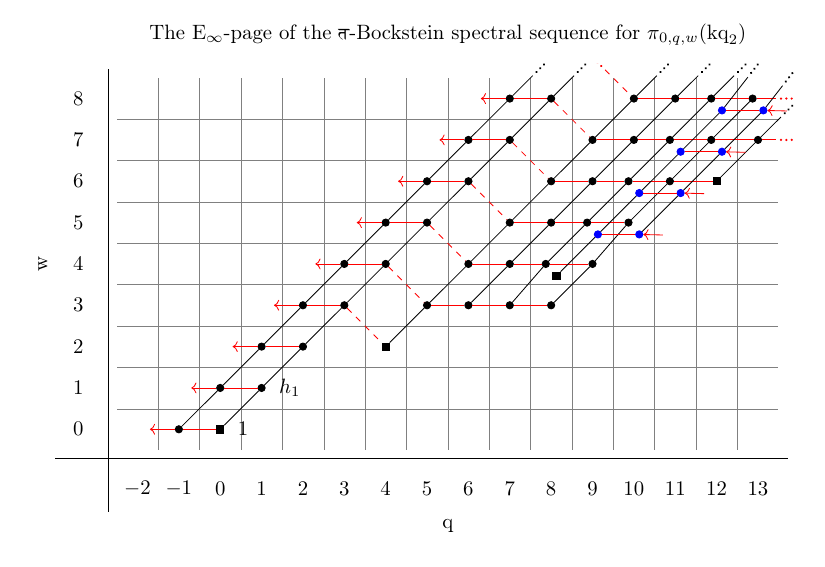}
\end{figure}

\begin{figure}[ht]
\centering
\includegraphics[trim=100 470 100 100, clip, width=\textwidth, scale=0.6]{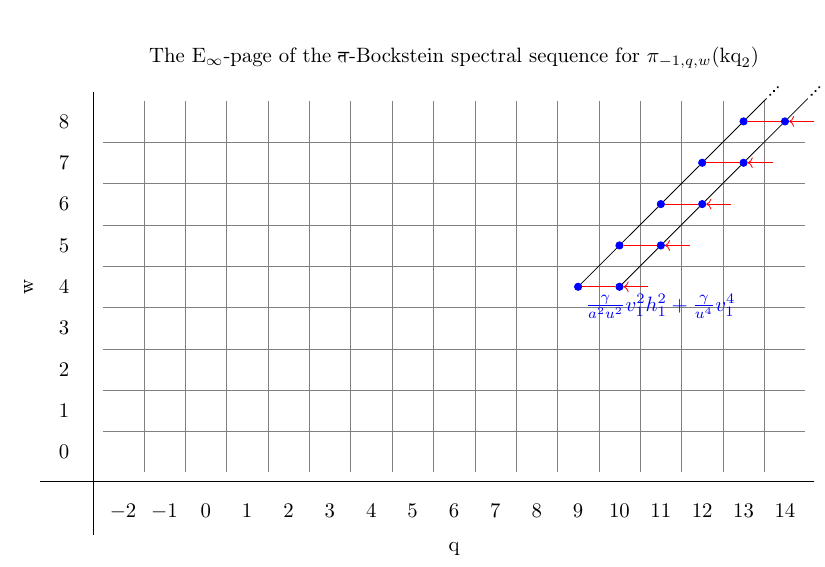}
\end{figure}

\begin{rmk}
The groups $\pi_{-1,*,*}\mathrm{kq}_2$ and $\pi_{0,*,*}\mathrm{kq}_2$ assemble, via the multiplication by $a$ long exact sequence, into the homotopy groups $\pi_{0,*,*}\left(Ca \otimes \mathrm{kq}_2\right)$.  These groups in turn record the bigraded homotopy of the $\C$-motivic $\mathrm{kq}_2$.  The reader may note an interesting extension, of the form 
\[2 \mathbb{Z}_2[\tau] \oplus \tau \mathbb{Z}_2[\tau] \to \mathbb{Z}_2[\tau] \to \mathbb{F}_2,\]
which appears when computing $\pi_{0,8,*}(\mathrm{kq}_2^{\wedge} \otimes Ca)$.
This extension is related to the orange dashed line in the final chart of \cite{Kong}.
\end{rmk}

\appendix

\section{Recollections on compact rigid generation} \label{app:boilerplate}
In this appendix we recall some useful material on compact rigid generation.
Most of this material has appeared elsewhere and all of it is certainly known to experts.
This appendix was mostly included for the convenience of the authors,
though we do hope the reader finds it to be a concise summary of a basic technique in higher algebra.

\begin{dfn}
  A stable, presentably monoidal category $\CC$ is \emph{rigidly generated} if
  it has a family of compact dualizable generators. \footnote{Here and throughout this appendix we only require that an object has a one-sided dual. The side on which an object has a dual will not be important.}
\end{dfn}

Our goal will be to study some properties of monoidal left adjoints out of a rigidly generated category.
This begins with the following construction.

\begin{cnstr}
  Given an $\En$-monoidal left adjoint
  $ f^* : \CC \to \mathcal{D} $,
  its right adjoint $f_*$ is lax $\En$-monoidal.
  In particular, $f_* (\o)$ is an $\En$-ring,
  so we obtain a factorization of $f^*$ as an $\mathbb{E}_{n-1}$-monoidal left adjoint into
  \[ \CC \xrightarrow{ - \otimes f_*\o} \Mod ( \CC ; f_*\o ) \xrightarrow{g^*} \mathcal{D}. \]
  We will say that the adjunction $f$ is \emph{$0$-affine} if $g$ is an adjoint equivalence.
  Note that being $0$-affine is a property of the underlying monoidal functor. \todo{This really should have a cite to Lurie for the functor Mod out of some cartesian fibration.}
\end{cnstr}

The main result of this appendix is a convenient criterion for $f$ to be $0$-affine.
Before we can state that result we make a preparatory definition.

\begin{dfn}
  Given a monoidal left adjoint $f^* : \CC \to \mathcal{D}$,
  we may consider the projection map,
  \[ X \otimes f_*Y \xrightarrow{\eta_f} f_*f^*(X \otimes f_*Y) \simeq f_*(f^*X \otimes f^*f_*Y) \xrightarrow{\epsilon_f} f_*(f^*X \otimes Y). \]
  We will say that $f$ \emph{satisfies the projection formula at $X$},
  if for all $Y$ the projection map is an equivalence.
  If $f$ satisfies the projection formula at $X$ for all $X$, then we say that $f$ \emph{satisfies the projection formula}.
\end{dfn}

\begin{prop} \label{prop:rigid}
  Given a monoidal left adjoint $f^* : \CC \to \mathcal{D}$ between presentable categories, the following statements are true:
  \begin{enumerate}
  \item $f$ satisfies the projection formula for dualizable objects in $\CC$.
  \item If $\CC$ is rigidly generated and $\mathcal{D}$ has compact unit, then $f_*$ preserves colimits.
  \item If $\CC$ is rigidly generated  and $f_*$ preserves colimits, then $f$ satisfies the projection formula.
  \item $f_*$ is conservative if and only if the essential image of $f^*$ contains a family of generators.
  \item If $f$ satisfies the projection formula, $f_*$ preserves colimits and $f_*$ is conservative, then $f$ is $0$-affine.
  \end{enumerate}
\end{prop}

Our arguments follow those from \cite{MNN} fairly closely, though the hypotheses are somewhat different. \todo{Are there more things to cite here?} Before proceeding to the proof of this proposition we give some examples to demonstrate its effectiveness.

\begin{exm} \label{exm:fil-ctau}
In appendix \Cref{app:fil} we study the category $\Sp^{\Fil}$ of filtered spectra, where we denote the shift map $\o(-1) \to \o$ by $\tau$.
  The associated graded functor $\Gr : \Sp^{\Fil} \to \Sp^{\Gr}$ satisfies the conditions of \Cref{prop:rigid}, so we obtain an equivalence of presentably symmetric monoidal categories
  \[ \Sp^{\Gr} \cong \Mod ( \Sp^{\Fil} ; C\tau ). \]
  Here $C\tau$ acquires a commutative algebra structure from the fact that it equivalent to the image of the unit under the right adjoint of $\Gr$.
\end{exm}

\begin{exm} \label{exm:fil-inv-tau}
  Again working with $\Sp^{\Fil}$,
  the realization functor $\Re : \Sp^{\Fil} \to \Sp$ satisfies the conditions of \Cref{prop:rigid}, so we obtain an equivalence of presentably symmetric monoidal categories
  \[ \Sp \cong \Mod ( \Sp^{\Fil} ; \o[\tau^{-1}] ). \]
  Here $\o[\tau^{-1}]$ acquires a commutative algebra structure from the fact that it is equivalent to the image of the unit under $Y$ (the right adjoint of $\Re$).
\end{exm}

\begin{exm}
  Given a finite extension of characteristic zero fields $\ell/k$ we have a symmetric monoidal left adjoint $ i^* : \SH(k) \to \SH(\ell) $.
Using resolution of singularities, Nagata compactification and purity we know that both of these categories are generated by the motives of smooth projective schemes, each of which is dualizable \cite{MotivicModules}. In short,
   both the source and target categories are rigidly generated, with families of compact dualizable generators given by $\Ss^m \otimes (\P^1)^{\otimes n} \otimes X$ where $X$ is a smooth projective scheme over the base.
  
  In order to show that $i$ is 0-affine we only need to show that the image of $i^*$ contains a family of generators. Given a smooth projective $\ell$-scheme $X$ the projection formula \Cref{prop:rigid}(1) gives us an equivalence
  \[ i_*(\Ss^m \otimes (\P^1)^{\otimes n} \otimes X) \simeq \Ss^m \otimes (\P^1)^{\otimes n} \otimes i_*X. \]
  Since for field extensions $i_{\sharp} = i_*$, we learn that $i_*X = X$ where the second copy of $X$ is considered as a $k$-scheme. It will now suffice to show that $X$ is a retract of $i^*i_*X$. At the level of schemes this means looking at $ X \times_{\Spec(\ell)} \Spec(\ell) \times_{\Spec(k)} \Spec(\ell) $. Using the maps $\ell \to \ell \otimes_k \ell \to \ell$ we may conclude.

  Stated more explicitly,
  we have an equivalence of presentably symmetric monoidal categories
  \[ \SH(\ell) \cong \Mod( \SHk ; \Spec(\ell) ). \]
  Using the fact that $i$ restricts to the subcategories of Artin--Tate objects the same argument provides an equivalence of presentably symmetric monoidal categories
  \[ \SH(\ell)^{\at} \cong \Mod ( \SHkat ; \Spec(\ell) ). \]
\end{exm}

\begin{exm} \label{exm:c2-underlying}
  The category of $C_2$-spectra is rigidly generated and the underlying spectrum functor, $\Phi^e$, is essentially surjective.  We may thus apply \Cref{prop:rigid} to obtain a symmetric monoidal equivalence
  \[ \Mod( \Sp_{C_2} ; R_1) \simeq \Sp, \]
  where $R_1$ is the image of $\Ss$ under the right adjoint to $\Phi^e$.
  Since $\Phi^e$ can be described as homotopy fixed points (or homotopy orbits) composed with $(C_{2})_+ \otimes -$, we find that its right adjoint is given by tensoring with $(C_{2})_+ \simeq Ca_\sigma$. Therefore, $R_1 \simeq Ca_\sigma$.

  Similarly, with $\Phi^e$ replaced by $\Phi^{C_2}$ we obtain a symmetric monoidal equivalence
  \[ \Mod( \Sp_{C_2} ; R_2) \simeq \Sp, \]
  where $R_2$ is the image of $\Ss$ under the right adjoint to geometric fixed points.
  We can compute that $R_2 \simeq \Ss[a_\sigma^{-1}]$ using the presentation of $C_2$-spectra via an isotropy separation square. 
\end{exm}

\begin{exm} \label{exm:c2-graded}
  Taking loops on the map $i : S^1 \to B\Pic(\Sp_{C_2})$ which sends a generator to $\Ss^{\sigma}$ we get a monoidal functor $i : \Z \to \Pic(\Sp_{C_2})$. Embedding the target into $\Sp_{C_2}$ and tensoring up to spectra we obtain a monoidal left adjoint,
  \[ i^* : \Sp^{\Gr} \to \Sp_{C_2} \]
  which sends $\Ss(1)$ to $\Ss^{\sigma}$. Since $\Sp^{\Gr}$ is rigidly generated, the unit in $\Sp_{C_2}$ is compact and the representation spheres generate $\Sp_{C_2}$ we may apply \Cref{prop:rigid} to conclude that $i$ is $0$-affine. Stated more explicitly,
  we have an equivalence of presentable categories,
  \[ \Sp_{C_2} \cong \Mod( \Sp^{\Gr} ; i_*\Ss ). \]
  The graded ring $i_*\Ss$ has $n^{\mathrm{th}}$ term given by $\Map^{\Sp} ( \Ss^{n\sigma}, \Ss )$.
  Expressed in terms of stunted projective spaces we have $(i_*\Ss)_n \simeq \Sigma \R \mathrm{P}_{-\infty}^{-n-1}$. 
\end{exm}

\begin{proof}[Proof (of \Cref{prop:rigid}(1)).]    
  Consider the following commutative diagram, which is natural in $X \in \CC^{\mathrm{dbl}}$, $Y \in \mathcal{D}$ and $Z \in \mathcal{C}$. 
  {\footnotesize
    \begin{center} \begin{tikzcd}
        \Map_{\mathcal{C}}(X^\vee \otimes Z, f_*Y) \ar[r] \ar[rr, "\simeq", bend left=10] &
        \Map_{\mathcal{D}}(f^*(X^\vee \otimes Z), f^*f_*Y) \ar[d, "\simeq"] \ar[r] &
        \Map_{\mathcal{D}}(f^*(X^\vee \otimes Z), Y) \ar[d, "\simeq"] \\        
        & \Map_{\mathcal{D}}(f^*(X^\vee) \otimes f^*Z, f^*f_*Y) \ar[d, "\simeq"] \ar[r] &
        \Map_{\mathcal{D}}(f^*(X^\vee) \otimes f^*Z, Y) \ar[d, "\simeq"] \\
        \Map_{\mathcal{C}}(Z, X \otimes f_*Y) \ar[d] \ar[uu, no head, "\simeq"] &
        \Map_{\mathcal{D}}((f^*X)^\vee \otimes f^*Z, f^*f_*Y) \ar[d, "\simeq"] \ar[r] &
        \Map_{\mathcal{D}}((f^*X)^\vee \otimes f^*Z, Y) \ar[d, "\simeq"] \\
        \Map_{\mathcal{D}}(f^*Z, f^*(X \otimes f_*Y)) \ar[r, "\simeq"] &
        \Map_{\mathcal{D}}(f^*Z , f^*X \otimes f^*f_*Y) \ar[r] &
        \Map_{\mathcal{D}}(f^*Z , f^*X \otimes Y) 
    \end{tikzcd} \end{center}
  }
  The rectangle on the left commutes due to the compatibility of dualization with the monoidal structure on $f^*$. The remaining squares commute for easier reasons. Now observe that starting in the middle of the left side and proceeding counter-clockwise gives the projection map, while proceeding clockwise given an equivalence.
\end{proof}

\begin{proof}[Proof (of \Cref{prop:rigid}(2)).]
  Since $\CC$ is rigidly generated it has a family of compact dualizable generators.
  Monoidal functors send dualizable objects to dualizable objects.
  Since the unit of $\mathcal{D}$ is compact we learn that $f^*$ sends dualizable objects to compact objects.
  Thus, since $f^*$ sends a family of compact generators to compact objects its right adjoint $f_*$ preserves colimits.
\end{proof}

\begin{proof}[Proof (of \Cref{prop:rigid}(3)).]
  The projection formula asks that the natural projection map
  \[ X \otimes f_*Y \to f_*(f^*X \otimes Y) \]
  be an equivalence.
  Using the hypotheses that $f_*$ preserves colimits and $\CC$ is rigidly generated we can reduce to the case where $X$ is a compact dualizable generator. 
  We may now use \Cref{prop:rigid}(1) to conclude.
\end{proof}

\begin{proof}[Proof (of \Cref{prop:rigid}(4)).]
  Clear.
\end{proof}

\begin{proof}[Proof (of \Cref{prop:rigid}(5)).]
  We begin by showing that $g^*$ is fully faithful.
  This is equivalent to showing that the unit map
  $X \to g_*g^*X$ is an equivalence.
  Applying the projection formula with $Y = \o$, we can conclude that this is true for induced $f_*\o$-modules.
  Since $f_*$ preserves colimits and $\Mod( \CC ; f_*\o)$ is generated by induced $f_*\o$-modules this is sufficient to conclude.

  Now, we show that $g^*$ is essentially surjective.
  By \Cref{prop:rigid}(4) we know the essential image of $f^*$ contains a family of generators (and $g^*$ has the same property).
  Using fully-faithfulness we can now conclude that $g^*$ is essentially surjective.  
\end{proof}

We close the appendix with another useful lemma.

\begin{lem} \label{lem:tensor-mod}
  Suppose that $\CC$ and $\D$ are stable presentably symmetric monoidal categories, and let $R$ be a commutative algebra in $\CC$.
  Given a symmetric monoidal left adjoint $f^* : \CC \to \mathcal{D}$, 
  there is an equivalence of presentably symmetric monoidal categories
  \[\Mod(\D; f^*R) \simeq \Mod(\CC; R) \otimes_\CC \D.\]
\end{lem}

\begin{proof}
  This follows from \cite[Theorem 4.8.5.16]{HA} after unraveling the definitions.
\end{proof}

\begin{exm}  
  Given a stable presentably symmetric monoidal category $\CC$ and a commutative algebra $R \in \CC^{\Fil}$, this lemma provides an equivalence
  \[ \Sp^{\Gr} \otimes_{\Sp^{\Fil}} \Mod ( \CC^{\Fil} ; R ) \simeq \CC^{\Gr} \otimes_{\CC^{\Fil}} \Mod ( \CC^{\Fil} ; R ) \simeq \Mod ( \CC^{\Gr} ; \Gr(R) ) \]
\end{exm}

\section{Recollections on filtered objects} \label{app:fil}


In this appendix we give a more detailed introduction to filtered objects.
The results here are well-known (cf. \cite{RotationInvariance}), and we aim mainly to fix notation.

\begin{cnv}
  In this appendix, $\CC$ will denote a stable, presentable category.
  We will use $\o$ to denote the unit of $\CC$ when it is monoidal.
\end{cnv}

\begin{dfn}
  We let $\Z$ denote the symmetric monoidal category with underlying category given by the discrete set $\Z$ and symmetric monoidal structure given by addition.

  Similarly, we let $\Z^{\Fil}$ denote the symmetric monoidal category with underlying category given by the poset $\Z$ with its order $\leq$, so that there is a unique map $n \to m$ whenever $n \leq m$, and symmetric monoidal structure given by addition.
\end{dfn}

\begin{dfn}
  Given a category $\CC$, we let $\CC^{\Fil} \coloneqq \Fun(\Z^{\Fil,op}, \CC)$ denote the category of filtered objects in $\CC$.
  Objects of $\CC^{\Fil}$ are diagrams
  \[ \dots \to C_2 \to C_1 \to C_0 \to C_{-1} \to C_{-2} \to \dots \]
  in the category $\CC$. We will sometimes use the notation $C_\bullet$ for an object of $\CC^{\Fil}$.
  \begin{itemize}
  \item There is a natural left adjoint $c : \CC \to \CC^{\Fil}$, which sends $C$ to the constant object
    \[ \dots \to 0 \to 0 \to C \xrightarrow{\id} C \xrightarrow{\id} C \xrightarrow{\id} \dots \]
    that is equal to $C$ in nonpositive degrees and $0$ in positive degrees.
  \item There is a natural left adjoint $Y : \CC \to \CC^{\Fil}$, which sends $C$ to the constant object
    \[ \dots \xrightarrow{\id} C \xrightarrow{\id} C \xrightarrow{\id}  C \xrightarrow{\id} C \xrightarrow{\id} C \xrightarrow{\id} \dots \]
    that is equal to $C$ in each degree.
  \item The functor $Y$ admits a left adjoint $\Re : \CC^{\Fil} \to \CC$, which sends
    $C_\bullet$ to $\varinjlim_{n} C_{-n}$.
  \item The category $\CC^{\Fil}$ admits natural automorphisms $(k) : \CC^{\Fil} \to \CC^{\Fil}$ that send $C_\bullet$ to $C_{\bullet - k}$.
  \item There is a natural transformation $\tau : (-1) \to \mathrm{Id}$, which captures the shift map in the filtration. We depict $\tau$ on $cX$ below,
  \begin{center}
    \begin{tikzcd}
      \dots \ar[r] & 0 \ar[r] \ar[d] & 0 \ar[r] \ar[d] & 0 \ar[r] \ar[d] & X \ar[r] \ar[d] & X \ar[d] \ar[r] & \dots \\
      \dots \ar[r] & 0 \ar[r] & 0 \ar[r] & X \ar[r] & X \ar[r] & X \ar[r] & \dots
    \end{tikzcd}
  \end{center}
  If $\CC$ is monoidal, then we will refer to the cofiber of $\tau : \o(-1) \to \o$ as $C\tau$.
  \end{itemize}

  We let $\CC^{\Gr} \coloneqq \Fun(\Z^{op}, \CC)$ denote the category of graded objects in $\CC$.
  Objects of $\CC^{\Gr}$ are collections $\{C_n\}_n$ of objects of $\CC$.
  We will sometimes use the notation $C_*$ for an object of $\CC^{\Gr}$.
  \begin{itemize}
  \item There is a natural fully-faithful left adjoint $c : \CC \to \CC^{\Gr}$, which sends $C$ to $C_*$ with $C_0 = C$ and $C_k = 0$ for $k \neq 0$.
  \item There is a natural left adjoint $\Gr : \CC^{\Fil} \to \CC^{\Gr}$, which sends
  \[ \dots \to C_2 \to C_1 \to C_0 \to C_{-1} \to C_{-2} \to \dots \]
  to $\{C_n / C_{n+1}\}_n$.
  \end{itemize}
  
  If $\CC$ is a presentably $\mathbb{E}_{n}$-monoidal category,
  then $\CC^{\Fil}$ and $\CC^{\Gr}$ inherit the structure of $\En$-monoidal categories under Day convolution,
  and the functors $c$, $Y$, $\Re$, $c$ and $\Gr$ are all $\En$-monoidal.
\end{dfn}

\begin{rmk}
  Using the assumption that $\CC$ is stable and presentable we can offer another description of the categories of filtered and graded objects, which is often useful in proofs.
  \[ \CC^{\Fil} \simeq \Sp^{\Fil} \otimes\ \CC \quad \text{ and } \quad \CC^{\Gr} \simeq \Sp^{\Gr} \otimes\ \CC. \]
  Since it is well-known that the analogs of $c$, $Y$, $\Re$, $c$ and $\Gr$ are symmetric monoidal in the case of spectra, the claims about $\En$ monoidality made above follow by tensoring up. 
\end{rmk}

Using the fact that $\o(1) \otimes \o(-1) \simeq \o$, we learn that if $X$ is a dualizable object of $\CC$, then $c(X)(n)$ is a dualizable object of $\CC^{\Fil}$. Similarly, we have that if $\{X_\alpha\}$ is a set of compact (dualizable) generators of $\CC$, then $\{c(X_\alpha)(k)\}$ is a set of compact (dualizable) generators of $\CC^{\Fil}$.






\begin{lem} \label{prop:mod-fil}
  Given an $\En$-monoidal category $\CC$,
  the image of $\o$ under the right adjoint of $\Re$ is $\o[\tau^{-1}]$;
  therefore this object is an $\En$-algebra and there is an equivalence of $\mathbb{E}_{n-1}$-monoidal categories
  $ \Mod(\CC^{\Fil}; \o[\tau^{-1}]) \simeq \CC $.
  Similarly,
  the image of $\o \in \CC^{\Gr}$ under the right adjoint of $\Gr$ is $C\tau$;
  therefore this object is an $\En$-algebra and there is an equivalence of $\mathbb{E}_{n-1}$-monoidal categories
  $ \Mod(\CC^{\Fil}; C\tau) \simeq \CC^{\Gr}. $
  Moreover, the functors $- \otimes \o[\tau^{-1}]$ and $- \otimes C\tau$ are identified under these equivalences with $\Re$ and $\Gr$, respectively.
\end{lem}

\begin{proof}
  The case of spectra was handled in Examples \ref{exm:fil-ctau} and \ref{exm:fil-inv-tau}.
  For a general $\CC$ we just tensor the case of spectra with $\CC$.
\end{proof}

\section{A machine for deforming homotopy theories} \label{app:def}
In this appendix, we have two goals.
The first is to identify a technique which produces 1-parameter deformations in homotopy theory.
This technique is essentially an elaboration of \cite[Definition 3.2]{cmmf}.
The second goal is to provide a recognition criterion for 1-parameter deformations.
This criterion will be applied in \Cref{sec:top} of the main paper to identify $\SHRat$ with the category of modules over a commutative algebra in $i2$-complete filtered $C_2$-spectra.

The approach to deforming categories taken here is different than the approach taken in Pstr\k{a}gowski's theory of synthetic spectra \cite{Pstragowski}, which is a theory specifically about deformations of $\Sp$. 
When our construction here and Pstr\k{a}gowki's construction are both defined, we will show that they are equivalent up to a minor completion issue. This in turn suggests that a more general version of Pstr\k{a}gowski's approach is possible, where an arbitrary (symmetric monoidal) presentable category is deformed. 

This appendix is not intended to be a definitive treatment of deformations.  Instead, we view at as an illustration of a variety of elementary techniques, which in combination produce a large collection of interesting examples.

\subsection{Constructing deformations} \label{app:def-const}\

In this section we will give techniques for producing 1-parameter deformations.

\begin{cnv}
  In this section, $\CC$ will denote a stable presentably symmetric monoidal category.
  We will use $\o$ to denote the unit of $\CC$, and assume the notation of \Cref{app:fil} 
\end{cnv}

The deformations we produce will all be of the form
\[ \Mod ( \CC^{\Fil} ; R ) \]
for some kind of algebra $R$.\footnote{Although much of the material in this appendix applies for $\En$-algebras in $\En$-monoidal categories, the extra generality was not necessary for this work.}
Therefore, what we really do is give methods for constructing algebras in filtered objects.

We will produce commutative algebras in $\CC^{\Fil}$ through the following method.
Given a lax symmetric monoidal functor $F : \CC \to \CC^{\Fil}$, the image of the unit $F(\o)$ is a commutative algebra in $\CC^{\Fil}$. If we let $L$ denote the composite $\Re \circ F$, then we have the following proposition.

\begin{prop}
  The presentably symmetric monoidal category $\Mod(\CC^{\Fil}; F(\o))$ is a 1-parameter deformation in the sense that,
  \begin{enumerate}
  \item There is a colimit-preserving symmetric monoidal functor out of $\CC^{\Fil}$ with target $\Mod(\CC^{\Fil}; F(\o))$.
  \item The generic fiber is given by $\Mod(\CC; L(\o))$, in the sense that
    there is an equivalences of presentably symmetric monoidal categories
    \[\Mod(\CC^{\Fil}; F(\o)[\tau^{-1}]) \simeq \Mod(\CC; L(\o)). \]
  \item The special fiber is given by $\Mod(\CC^{\Gr}; \Gr F(\o) )$, in the sense that
    there is an equivalences of presentably symmetric monoidal categories
    \[\Mod(\CC^{\Fil}; F(\o) \otimes C\tau) \simeq \Mod(\CC^{\Gr}; \Gr(F(\o))).\]
  \end{enumerate}
  Moreover, when viewed as a lax symmetric monoidal functor, $F$ factors as
  \[ \CC \xrightarrow{G} \Mod(\CC^{\Fil}; F(\o)) \to \CC^{\Fil}. \]
\end{prop}

\begin{proof}
  This follows immediately from \Cref{prop:mod-fil} and \Cref{lem:tensor-mod}. \todo{update when this breaks.}
\end{proof}

We will refer to lax symmetric monoidal functors $T : \CC \to \CC^{\Fil}$ as \emph{tower functors} and give several examples.

\begin{exm}
  The functors $c$ and $Y$ are tower functors.
\end{exm}

\begin{exm}
  The functor $\tau_{\geq \bullet} : \Sp \to \Sp^{\Fil}$ is a tower functor.
\end{exm}

\begin{exm}
  If $\CC$ has a $t$-structure which is compatible with the tensor product in the sense that the unit is connective and the tensor product of two connective objects is connective, then
  \[ \tau_{\geq \bullet} : \CC \to \CC^{\Fil} \]
  is a tower functor. More generally we have a tower functor
  $ \tau_{\geq m \bullet} : \CC \to \CC^{\Fil} $ for natural numbers $m$.
\end{exm}

\begin{cnstr} \label{cnstr:trunc-tower}
  Suppose we are given a collection $A$ of coreflective subcategories $\CC^A_{\geq k} \subset \CC$,
  such that
  \begin{itemize}
  \item if $X \in \CC^A_{\geq k}$ and $Y \in \CC^A_{\geq \ell}$, then $X \otimes Y \in \CC^A_{\geq k+\ell}$.
  \item $\CC^A_{\geq k+1} \subset \CC^A_{\geq k}$ and $\o \in \CC^A_{\geq 0}$.
  \end{itemize}

  Let $\tau_{\geq k} : \CC \to \CC^A_{\geq k}$ denote the right adjoints of the inclusions. 
  Then, we can assemble these categories into a single coreflective subcategory $\CC_{\geq 0}^{\Fil, A}$
  consisting of filtered objects
  \[\dots \to X_2 \to X_1 \to X_0 \to X_{-1} \to X_{-2} \to \dots\]
  for which $X_i \in \CC^A_{\geq i}$.
  The right adjoint $\tau_{\geq 0}^A : \CC^{\Fil} \to \CC^{\Fil, A}_{\geq 0}$ to the inclusion is given by applying $\tau_{\geq i}$ in position $i$.

  Our assumptions guarantee that $\CC_{\geq 0}^{\Fil, A}$ is closed under the tensor product and so admits a natural presentably symmetric monoidal structure, for which the inclusion $\CC^{\Fil, A}_{\geq 0} \subset \CC^{\Fil}$ is a symmetric monoidal functor. As a consequence, we obtain a lax symmetric monoidal endofunctor  
   \[ \tau^A_{\geq 0} : \CC^{\Fil} \to \CC^{\Fil}. \]
\end{cnstr}

\begin{exm}
  If $\CC$ is the category of $G$-equivariant spectra for some finite group $G$, then we can let $\CC^{\mathrm{slice}}_{\geq k}$ be the regular slice $k$-connective $G$-spectra.\footnote{Note that we cannot use the classical slice filtration here, unless $m$ is divisible by $\abs{G}$ in which case it agrees with the regular one, since it is not compatible with the tensor product.} As a variant we could also take the regular slice $mk$-connective subcategories for some positive integer $m$.
\end{exm}

Each of the above examples of tower functors can be described as a composite of $Y$ with an appropriately chosen $\tau_{\geq 0}$.
This construction can be considered a generalization of the twisted $t$-structures on filtered objects considered in \Cref{subsec:twists}.
Thus, so far we have only produced ``truncation type towers.''
We now give a construction which takes in a tower functor $T$ and a commutative algebra in $\CC$ and produces a new tower functor. This construction will have the effect of shearing the Adams spectral sequence based on $E$ along the tower $T$. To begin, we review the monoidal properties of the cobar construction.

\begin{cnstr}
  Since the coproduct of commutative algebras in $\CC$ is given by the tensor product, 
  the cobar construction can be upgraded into a functor
  \[ \mathrm{cb} : \CAlg (\CC) \to \CAlg (\CC^{\Delta}) .\]

  

\end{cnstr}

\begin{cnstr}\label{cnstr:shear}
  Given a tower functor $T$ and a commutative algebra $E$ in $\CC$, we define a new tower functor $\mathrm{Sh}(T ; E)$ which is the composite,
  \[\CC \xrightarrow{- \otimes \mathrm{cb}(E)} \CC^{\Delta} \xrightarrow{T} \CC^{\Fil, \Delta} \xrightarrow{\Tot} \CC^{\Fil}.\]
\end{cnstr}

On spectra, the tower functor $\mathrm{Sh}(\tau_{\geq \bullet} ; \F_p)$ produces the d\'ecalage of the $\F_p$-Adams tower on the input. More generally, given a $t$-structure t that is compatible with the monoidal structure, we can let $\CC^t_{\geq n}$ be the $n$-connective objects.
Then $\mathrm{Sh}(\tau_{\geq 0}^t ; E)(X)$ captures the $E$-based Adams spectral sequence for the $t$-structure homotopy groups of $X$. As an analog of the fact that the $E$-Adams spectral sequence for $E$ collapses, we have the following:

\begin{exm} \label{exm:E-collapse}
  The cosimplicial diagram $E \otimes \mathrm{cb}(E)$ admits a contracting homotopy, and therefore the totalization commutes with any functor. In particular, we obtain an equivalence of commutative algebras,
  $ \mathrm{Sh}(T ; E)(E) \simeq T(E) $.
  \todo{Is this right?}
\end{exm}

We close with a simple lemma that lets us identify the generic fiber in certain cases.

\begin{lem}\label{lem:easy-conv}
  Suppose we are given a collection of coreflective subcategories $A$ that satisfy the conditions of \Cref{cnstr:trunc-tower}, together with a commutative algebra $E \in \CC^A_{\geq 0}$.
  Then for any $X \in \CC_{\geq k}^A$ there is an equivalence
  $\Re(\mathrm{Sh}(\tau_{\geq 0}^A ; E)) \simeq X^{\wedge} _E$, where the second object is the $E$-nilpotent completion of $X$.
\end{lem}

\begin{proof}
  By construction $\tau^A_{\geq 0}$ comes equipped with a natural transformation $\tau^A_{\geq 0} \to Y$.
  Since $\mathrm{Sh}(Y ; E) \simeq Y(-)^{\wedge}_E$, we have a natural transformation
  \[ \mathrm{Sh}(\tau_{\geq 0}^A ; E) \to  Y(X)^{\wedge}_E.\]
  In sufficiently negative degrees the $\tau_{\geq i}$'s have no effect by hypothesis, so 
  the totalizations are levelwise equivalences.  There is thus an equivalence
  \[ \Re(\mathrm{Sh}(\tau_{\geq 0}^A ; E)) \simeq  \Re Y(X)^{\wedge} _E \simeq (X)^{\wedge} _E.\qedhere \]
\end{proof}

\begin{rmk}
 Under the assumptions of \Cref{lem:easy-conv}, we have an equivalence
\[ \Mod(\CC^{\Fil}; \mathrm{Sh}(\tau_{\geq 0}^A ; E)(\o)[\tau^{-1}]) \simeq \Mod(\CC; \o^{\wedge}_{E}). \]
\end{rmk}

\subsection{Recognizing Deformations} \label{app:compare}\ 

In this section, we give one answer to the following open-ended question:

\begin{qst}
  Given a pair of presentably symmetric monoidal categories $\CCdef$ and $\CC$,
  when can we identify $\CCdef$ with $\Mod(\CC^{\Fil}; R)$ for some commutative algebra $R \in \CC^{\Fil}$?
\end{qst}

This section grew out of a recognition that our original arguments in \Cref{sec:top}, which related $\SHRat$ and  $\Sp_{C_2,i2}$, used only very general information.\todo{Refine this}

\begin{dfn} \label{dfn:def-pair}
  A deformation pair consists of the following data:
  \begin{itemize}
  \item A diagram of symmetric monoidal left adjoints
    \begin{center} \begin{tikzcd}
        & \CCdef \ar[dr, "\Re"] & \\
        \CC \ar[ur, "c"] \ar[rr, "\mathrm{Id}"] & & \CC.
    \end{tikzcd} \end{center}
  \item A map of abelian groups
    \[ i_0 : \Z \to \ker \left( \pi_0\Pic(\CCdef) \to \pi_0\Pic(\CC) \right). \]
    We denote the invertible object corresponding to $i_0(n)$ by $\o(n)$.
  \end{itemize}
  These data are subject to the following conditions:
  \begin{itemize}
  \item For $a \leq b$, the map on mapping spaces
    \[ \Hom_{\CCdef} (\o(a),\o(b)) \to \Hom_{\CC} (\o,\o) \]
    induced by $\Re$ is an equivalence.
  \item The category $\CC$ is rigidly generated, and there is a set
    $\{C_\alpha\}$ of compact dualizable objects in $\CC$
    with the property that $\{c(C_{\alpha}) \otimes \o(n)\}$
    is a set of compact dualizable generators for $\CCdef$.
  \end{itemize}
\end{dfn}


\begin{exm}
  Suppose that $\CC$ is a rigidly generated, stable, presentably symmetric monoidal category.
  Then it is easy to verify that $(\CC^{\Fil}, \CC)$ is a deformation pair where
  $i_0 : \Z \to \Pic(\CC^{\Fil})$ sends $n$ to $\o(n)$.
\end{exm}

\begin{exm} \label{exm:syn-part-one}
  Let $E$ denote an Adams type homology theory, and let $\Syn_E$ denote Pstr\k{a}gowski's category of $E$-synthetic spectra \cite{Pstragowski}.
  This category is equipped with a natural notion of bigraded sphere, and we let $\Syn_E ^{\cell} \subset \Syn_E$ denote the \emph{cellular} subcategory generated under colimits by $\Ss^{p,q}$.\footnote{ When $E$ is $\F_p$ or $\MU$, then $\Syn_E ^{\cell} = \Syn_E$.}
  There is a natural realization functor $\Syn_E \to \Sp$, as well as a symmetric monoidal left adjoint $\Sp \to \Syn_E$ provided by \cite[Corollary 4.8.2.19]{HA}.
  
  Then $(\Syn_E ^{\cell}, \Sp)$ is a deformation pair, 
  where the map $i_0$ picks out the spheres $\Ss^{0,n}$.
  The first condition is satisfied by \cite[Corollary 4.12]{Pstragowski}, and the second is satisfied since we restricted to the full subcategory generated by the bi-graded spheres.
\end{exm}


\begin{exm}
  Another example is given by $(\SHC^{\cell}_{ip}, \Sp_{ip})$ where the map $i_0$ picks out the Tate twists.
  The realization functor here is Betti realization, and we set $c$ to be the unital, colimit-preserving map in from $\Sp_{ip}$ as in \Cref{exm:Sp-unit}.
	By restricting to cellular objects the second condition is automatically satisfied. The first condition is ultimately a corollary of the vanishing of the homotopy of $C\tau$ in positive Chow degrees.
  This example is discussed at length in \cite{cmmf}.

\end{exm}

Given a deformation pair $(\CCdef, \CC)$, we would like to construct a symmetric monoidal left adjoint
$ \CC^{\Fil} \to \CCdef $
to which we may apply \Cref{prop:rigid}.
We will do this in two steps:
\begin{enumerate}
\item We construct a symmetric monoidal functor $i : \Z^{\Fil} \to \CCdef$, which sends $n$ to $\o(n)$.
\item We tensor $i^{\Fil}$ up to $\CC$ using $c$.
\end{enumerate}



\begin{cnstr} \label{lem:fil-build}
  Given a deformation pair $(\CCdef, \CC)$, we construct a square of symmetric monoidal functors,
  \begin{center} \begin{tikzcd}
      \Z^{\Fil} \ar[r,"i"] \ar[d] & \mathcal{C}_{\mathrm{def}} \ar[d, "\Re"] \\
      * \ar[r] & \mathcal{C} 
  \end{tikzcd} \end{center}
  such that $i(n) = \o(n)$.
\end{cnstr}
  

\begin{proof}[Details.]
  Let $\mathcal{D}$ denote the full subcategory of $\CC$ on the unit.
  Let $\mathcal{D}_{\mathrm{def}}$ denote the full subcategory of $\mathcal{C}_{\mathrm{def}}$ on the objects $\o(n)$ in the image of $i_0$.
  Since $\mathcal{D}_{\mathrm{def}}$ and $\mathcal{D}$ are closed under the tensor product they are each symmetric monoidal categories.
 


  We now form the following diagram of symmetric monoidal categories
  \begin{center} \begin{tikzcd}
      \mathcal{D}' \ar[r, "f"] &
      \mathcal{D}_{\mathrm{def}} \times \Z^{\Fil} \ar[r, "\Re \times \mathrm{Id}"] \ar[d, "\pi_1"] &
      \mathcal{D} \times \Z^{\Fil} \ar[r, "\pi_2"] \ar[d, "\pi_1"] &
      \Z^{\Fil} \ar[d] \\
      & \mathcal{D}_{\mathrm{def}} \ar[r, "\Re"] &
      \mathcal{D} \ar[r] & *,
  \end{tikzcd} \end{center}
  where 
  $\mathcal{D}'$ is the full subcategory of $\mathcal{D}_{\mathrm{def}} \times \Z^{\Fil}$ spanned by the objects $(\o(n),n)$. Since $\mathcal{D}'$ is closed under the monoidal structure, it canonically inherits a symmetric monoidal structure from $\mathcal{D}_{\mathrm{def}} \times \Z^{\Fil}$.

  We claim that the composite $\mathcal{D}' \to \mathcal{D} \times \Z^{\Fil}$ is an equivalence.
  The objects of $\mathcal{D}'$ may be identified with pairs $(\o(n),n)$ and the objects of $\mathcal{D} \times \Z^{\Fil}$ may be identified with pairs $(\o, n)$. 
  The mapping spaces are given by:
  \[ \Hom_{\mathcal{D}'}((\o(n),n), (\o(m),m)) = \begin{cases} \Hom_{\mathcal{D}_{\mathrm{def}}}(\o(n),\o(m)) & n \leq m \\ \emptyset & n > m \end{cases} \]
  \[ \Hom_{\mathcal{D}\times \Z^{\Fil}} ((\o,n), (\o,m)) = \begin{cases} \Hom_{\mathcal{D}}(\o,\o) & n \leq m \\ \emptyset & n > m, \end{cases} \]
  with $\Re$ giving the maps between these mapping spaces.  Each of these maps of mapping spaces is an equivalence, by hypothesis.

  The composition of symmetric monoidal functors
  \[ \Z^{\Fil} \xrightarrow{\iota} \mathcal{D} \times \Z^{\Fil} \xleftarrow{\simeq} \mathcal{D}' \to \mathcal{D}_{\mathrm{def}} \times \Z^{\Fil} \to \mathcal{D}_{\mathrm{def}} \subset \CCdef \]
  is the desired symmetric monoidal functor $\Z^{\Fil} \to \CCdef$, and indeed sends $n$ to $\mathbb{1}(n)$.
  The square of \Cref{lem:fil-build} commutes, because
  \[ \Re \circ \pi_1 \circ f \circ (\Re \times \mathrm{Id})^{-1} \circ \iota = \pi_1 \circ (\Re \times \mathrm{Id}) \circ f \circ ((\Re \times \mathrm{Id}) \circ f)^{-1} \circ \iota = \pi_1 \circ \iota = * .\] \qedhere

\end{proof}


\begin{cnstr}\label{cnstr:filt-diagram}
  Now that we have the symmetric monoidal functor $i : \Z^{\Fil} \to \CCdef$, we may induce it up to a
  symmetric monoidal left adjoint,
  $ i^* : \Sp^{\Fil} \to \CCdef $.
  Tensoring up to $\CC$ using $c$, we can use \Cref{lem:fil-build} to build a diagram of symmetric monoidal left adjoints,
  \begin{center} \begin{tikzcd}
      \CC \ar[r, "c"] \ar[d, "\mathrm{Id}"] & \CC^{\Fil} \ar[r, "\Re"] \ar[d, "i^*"] & \CC \ar[d, "\mathrm{Id}"] \\
      \CC \ar[r, "c"]  & \CCdef \ar[r, "\Re"] & \CC .
  \end{tikzcd} \end{center}  
  
\end{cnstr}

\begin{prop}\label{prop:filt-mod}
  Suppose that $(\CCdef, \CC)$ is a deformation pair.
  Then, the symmetric monoidal left adjoint $i^*$ from \Cref{lem:fil-build} is $0$-affine.
  Stated more explicitly, we have a diagram of symmetric monoidal left adjoints as shown.
  \begin{center} \begin{tikzcd}
      \CC \ar[r, "c"] \ar[d, "\mathrm{Id}"] &
      \CC^{\Fil} \ar[dr, "i^*"] \ar[r, "- \otimes i_*\o"] &
      \Mod ( \CC^{\Fil} ; i_*\o ) \ar[d, "\simeq"] \ar[r, "\Re"] & 
      \CC \ar[d, "\mathrm{Id}"] \\
      \CC \ar[rr, "c"]  & &
      \CCdef \ar[r, "\Re"] & \CC .
  \end{tikzcd} \end{center}  
\end{prop}

\begin{proof}
  \Cref{cnstr:filt-diagram} produces most of the desired diagram.
  The remaining claims will follow from an application of \Cref{prop:rigid}.

  Recall that by hypothesis $\CC$ and $\CCdef$ are both rigidly generated.
  Therefore, in order to apply \Cref{prop:rigid} we only need to check that the essential image of $i^*$ contains a family of generators.
  Since $i^*(C_\alpha \otimes \o(n)) \simeq c(C_\alpha) \otimes \o(n)$, this is true by hypothesis.
\end{proof}

\begin{rmk} \label{rmk:i-lower-hom}
  The $n^{\mathrm{th}}$ piece of $i_*X$ can be extracted by taking the $\CC$-enriched mapping object,
  i.e. there is an equivalence $ (i_*X)_n \simeq \Map^{\CC}( \o(n), i_*X ) $.
  Now, since the adjunction $i$ is $\CC$-linear, we have an equivalence $\Map^{\CC}( \o(n), i_*X) \simeq \Map^{\CC}( i^*\o(n), X)$.
\end{rmk}

We close by considering once again the deformation context of synthetic spectra (\Cref{exm:syn-part-one}). Our aim will be to show that we can (nearly) identify $i_*\o$ with a commutative algebra produced via the constructions from the previous subsection.

\begin{prop}\label{prop:syn-to-fil}
  Given an Adams-type, commutative algebra $E$ in $\Sp$,
  there is an equivalence of presentably symmetric monoidal categories,
  \[ \Mod( \Syn_E^{\cell} ; \o_\tau^{\wedge} ) \simeq \Mod( \Sp^{\Fil} ; \mathrm{Sh}(\tau_{\geq \bullet} ; E)(\Ss)). \]
\end{prop}
   
\begin{proof}
  In \Cref{exm:syn-part-one} we showed that $\Syn_E^{\cell}$ and $\Sp$
  form a deformation pair. Using \Cref{prop:filt-mod} we obtain an adjunction $i$ and an equivalence,
  $ \Syn_E^{\cell} \simeq \Mod( \Sp^{\Fil} ; i_*\Ss)$. The proposition will now follow from an identification of the $\tau$-completion of $i_*\Ss$. 
  
  The identification of $i_*\Ss$ will follow the pattern established in \Cref{sec:top}.  
  We begin by identifying $i_* \nu (E^{\otimes k})$.
  By \cite[Proposition 4.21]{Pstragowski} we know that the $n^{\mathrm{th}}$ piece of this object is $n$-connective. Therefore, the natural comparison map $i_* \nu (E^{\otimes k}) \to Y(E^{\otimes k})$ factors as
  \[ i_*\nu ( E^{\otimes k} ) \to \tau_{\geq \bullet} E^{\otimes k} \to Y(E^{\otimes k}). \]
  Examining \cite[Proposition 4.21]{Pstragowski} more closely we can actually conclude that the first map is an equivalence.
  
  Now, we can pass to totalizations and observe the chain of equivalences:
  \begin{align*}
    (i_*\nu \Ss)_\tau^{\wedge}
    &\xrightarrow{\simeq} i_*((\nu \Ss)_\tau^{\wedge})
    \xrightarrow{\simeq} i_* \mathrm{Tot}^* (\mathrm{cb}(\nu E)) 
    \xrightarrow{\simeq} i_* \mathrm{Tot}^* (\nu (\mathrm{cb}(E))) \\
    &\xrightarrow{\simeq} \mathrm{Tot}^* (i_* \nu (\mathrm{cb}(E)))
    \xrightarrow{\simeq} \mathrm{Tot}^* (\tau_{\geq \bullet} \mathrm{cb}(E))
    \xrightarrow{\simeq} \mathrm{Sh}(\tau_{\geq \bullet} ; E)(\Ss). 
  \end{align*}  
  The first equivalence uses that $i_*$ is a right adjoint.
  The second equivalence follows from \cite[Proposition A.11]{Boundaries}.
  The third equivalence follows from \cite[Lemma 4.24]{Pstragowski}, together with the assumption that $E$ is Adams type.
\end{proof}



\bibliographystyle{alpha}
\bibliography{bibliography}

\begin{thebibliography}{BKWX22}

\bibitem[Ada95]{Adams}
J.~F. Adams.
\newblock {\em Stable homotopy and generalised homology}.
\newblock Chicago Lectures in Mathematics. University of Chicago Press,
  Chicago, IL, 1995.
\newblock Reprint of the 1974 original.

\bibitem[Bac18]{RealEtale}
Tom Bachmann.
\newblock Motivic and real \'{e}tale stable homotopy theory.
\newblock {\em Compos. Math.}, 154(5):883--917, 2018.

\bibitem[Bac20]{TomEtaleI}
Tom Bachmann.
\newblock Rigidity in {\'e}tale motivic stable homotopy theory.
\newblock {\em Algebr. Geom. Topol.}, 2020.

\bibitem[BE{\O}20]{TomEtaleII}
Tom Bachmann, Elden Elmanto, and Paul~Arne {\O}stv{\ae}r.
\newblock Stable motivic invariants are eventually {\'e}tale local.
\newblock 2020.
\newblock \href{https://arxiv.org/abs/2003.04006}{arxiv:2003.04006}.

\bibitem[BGI20]{c2stems}
Eva Belmont, Bertrand Guillou, and Daniel Isaksen.
\newblock {$C_2$}-equivariant and {$\R$}-motivic stable stems, ii.
\newblock 2020.
\newblock \href{https://arxiv.org/abs/2001.02251}{arXiv:2001.02251}.

\bibitem[BH20a]{etaPeriodic}
Tom Bachmann and Michael~{J}. Hopkins.
\newblock $\eta$-periodic motivic stable homotopy theory over fields.
\newblock 2020.
\newblock \href{https://arxiv.org/abs/2005.06778}{arxiv:2005.06778}.

\bibitem[BH20b]{norms}
Tom Bachmann and Marc Hoyois.
\newblock Norms in motivic homotopy theory.
\newblock {\em Astérisque}, 2020.

\bibitem[BHS19]{Boundaries}
Robert Burklund, Jeremy Hahn, and Andrew Senger.
\newblock On the boundaries of highly connected, almost closed manifolds.
\newblock 2019.
\newblock \href{https://arxiv.org/abs/1910.14116}{arXiv:1910.14116}.

\bibitem[BHV18]{LocDual}
Tobias Barthel, Drew Heard, and Gabriel Valenzuela.
\newblock Local duality in algebra and topology.
\newblock {\em Adv. Math.}, 335:563--663, 2018.

\bibitem[BKWX22]{Chow}
Tom Bachmann, Hana~Jia Kong, Guozhen Wang, and Zhouli Xu.
\newblock The {C}how {$t$}-structure on the {$\infty$}-category of motivic
  spectra.
\newblock {\em Ann. of Math. (2)}, 195(2):707--773, 2022.

\bibitem[Bon10]{Bondarko}
M.~V. Bondarko.
\newblock Weight structures vs. {$t$}-structures; weight filtrations, spectral
  sequences, and complexes (for motives and in general).
\newblock {\em J. K-Theory}, 6(3):387--504, 2010.

\bibitem[BS20]{BehrensShah}
Mark Behrens and Jay Shah.
\newblock {$C_2$}-equivariant stable homotopy from real motivic stable
  homotopy.
\newblock {\em Ann. K-Theory}, 5(3):411--464, 2020.

\bibitem[Bur20]{BurklundExtension}
Robert Burklund.
\newblock An extension in the {A}dams spectral sequence in dimension 54.
\newblock 2020.
\newblock \href{https://arxiv.org/abs/2005.08910}{arxiv:2005.08910}.

\bibitem[DI05]{DICell}
Daniel Dugger and Daniel~C. Isaksen.
\newblock Motivic cell structures.
\newblock {\em Algebr. Geom. Topol.}, 5:615--652, 2005.

\bibitem[DI10]{DI}
Daniel Dugger and Daniel~C. Isaksen.
\newblock The motivic {A}dams spectral sequence.
\newblock {\em Geom. Topol.}, 14(2):967--1014, 2010.

\bibitem[ES19]{EldenJay}
Elden Elmanto and Jay Shah.
\newblock Scheiderer motives and equivariant higher topos theory.
\newblock 2019.
\newblock \href{https://arxiv.org/abs/1912.11557}{arxiv:1912.11557}.

\bibitem[Ghe18]{Ctau}
Bogdan Gheorghe.
\newblock The motivic cofiber of {$\tau$}.
\newblock {\em Doc. Math.}, 23:1077--1127, 2018.

\bibitem[GHIR20]{koC2}
Bertrand~J. Guillou, Michael~A. Hill, Daniel~C. Isaksen, and Douglas~Conner
  Ravenel.
\newblock The cohomology of {$C_2$}-equivariant {$\mathcal A(1)$} and the
  homotopy of {${\rm ko}_{C_2}$}.
\newblock {\em Tunis. J. Math.}, 2(3):567--632, 2020.

\bibitem[GI17]{GI}
Bogdan Gheorghe and Daniel~C. Isaksen.
\newblock The structure of motivic homotopy groups.
\newblock {\em Bol. Soc. Mat. Mex. (3)}, 23(1):389--397, 2017.

\bibitem[GIKR18]{cmmf}
Bogdan Gheorghe, Daniel~C. Isaksen, Achim Krause, and Nicolas Ricka.
\newblock C-motivic modular forms.
\newblock 2018.
\newblock \href{https://arxiv.org/abs/1810.11050}{arxiv:1810.11050}.

\bibitem[{Goe}08]{GoerssMfg}
Paul~G. {Goerss}.
\newblock {Quasi-coherent sheaves on the moduli stack of formal groups}.
\newblock 2008.
\newblock \href{https://arxiv.org/abs/0802.0996}{arxiv:0802.0996}.

\bibitem[Gre12]{gregersen}
Thomas Gregersen.
\newblock {\em A {S}inger construction in motivic homotopy theory}.
\newblock PhD thesis, University of Oslo, 2012.

\bibitem[GWX20]{ctauactalol}
Bogdan Gheorghe, Guozhen Wang, and Zhouli Xu.
\newblock The special fiber of the motivic deformation of the stable homotopy
  category is algebraic.
\newblock {\em Acta Math.}, 2020.

\bibitem[Hea19]{SliceComp}
Drew Heard.
\newblock On equivariant and motivic slices.
\newblock {\em Algebr. Geom. Topol.}, 19(7):3641--3681, 2019.

\bibitem[HHR16]{HHR}
M.~A. Hill, M.~J. Hopkins, and D.~C. Ravenel.
\newblock On the nonexistence of elements of {K}ervaire invariant one.
\newblock {\em Ann. of Math. (2)}, 184(1):1--262, 2016.

\bibitem[HK01]{HuKriz}
Po~Hu and Igor Kriz.
\newblock Real-oriented homotopy theory and an analogue of the
  {A}dams-{N}ovikov spectral sequence.
\newblock {\em Topology}, 40(2):317--399, 2001.

\bibitem[HKO11a]{MAdamsConv}
P.~Hu, I.~Kriz, and K.~Ormsby.
\newblock Convergence of the motivic {A}dams spectral sequence.
\newblock {\em J. K-Theory}, 7(3):573--596, 2011.

\bibitem[HKO11b]{HKO}
Po~Hu, Igor Kriz, and Kyle Ormsby.
\newblock Remarks on motivic homotopy theory over algebraically closed fields.
\newblock {\em J. K-Theory}, 7(1):55--89, 2011.

\bibitem[HK{\O}17]{HoyoisSteenrod}
Marc Hoyois, Shane Kelly, and Paul~Arne {\O}stv{\ae}r.
\newblock The motivic {S}teenrod algebra in positive characteristic.
\newblock {\em J. Eur. Math. Soc. (JEMS)}, 19(12):3813--3849, 2017.

\bibitem[HO16]{HellerOrmsby}
J.~Heller and K.~Ormsby.
\newblock Galois equivariance and stable motivic homotopy theory.
\newblock {\em Trans. Amer. Math. Soc.}, 368(11):8047--8077, 2016.

\bibitem[HO18]{HellerOrmsbyII}
J.~Heller and K.~Ormsby.
\newblock The stable {G}alois correspondence for real closed fields.
\newblock In {\em New directions in homotopy theory}, volume 707 of {\em
  Contemp. Math.}, pages 1--9. Amer. Math. Soc., Providence, RI, 2018.

\bibitem[Hoy15]{HM}
Marc Hoyois.
\newblock From algebraic cobordism to motivic cohomology.
\newblock {\em J. Reine Angew. Math.}, 702:173--226, 2015.

\bibitem[Hoy17]{hoyois}
Marc Hoyois.
\newblock The six operations in equivariant motivic homotopy theory.
\newblock {\em Adv. Math.}, 305:197--279, 2017.

\bibitem[Hu05]{HuPicard}
Po~Hu.
\newblock On the {P}icard group of the stable {$\Bbb A^1$}-homotopy category.
\newblock {\em Topology}, 44(3):609--640, 2005.

\bibitem[Isa19]{StableStems}
Daniel~C. Isaksen.
\newblock Stable stems.
\newblock {\em Mem. Amer. Math. Soc.}, 262(1269):viii+159, 2019.

\bibitem[Kon20]{Kong}
Hana~Jia Kong.
\newblock The {$C_2$}-effective spectral sequence for {$C_2$}-equivariant
  connective real {$K$}-theory.
\newblock 2020.
\newblock \href{https://arxiv.org/abs/2004.00806}{arxiv:2004.00806}.

\bibitem[Lan67]{LandweberI}
Peter~S. Landweber.
\newblock Fixed point free conjugations on complex manifolds.
\newblock {\em Ann. of Math. (2)}, 86:491--502, 1967.

\bibitem[Lan68]{LandweberII}
Peter~S. Landweber.
\newblock Conjugations on complex manifolds and equivariant homotopy of {$MU$}.
\newblock {\em Bull. Amer. Math. Soc.}, 74:271--274, 1968.

\bibitem[Lev14]{LevineComparison}
Marc Levine.
\newblock A comparison of motivic and classical stable homotopy theories.
\newblock {\em J. Topol.}, 7(2):327--362, 2014.

\bibitem[Lev15]{LevineSlicemU}
Marc Levine.
\newblock The {A}dams-{N}ovikov spectral sequence and {V}oevodsky's slice
  tower.
\newblock {\em Geom. Topol.}, 19(5):2691--2740, 2015.

\bibitem[Lin80]{Lin}
Wen~Hsiung Lin.
\newblock On conjectures of {M}ahowald, {S}egal and {S}ullivan.
\newblock {\em Math. Proc. Cambridge Philos. Soc.}, 87(3):449--458, 1980.

\bibitem[Loy17]{CpBox}
Kaitlyn Loyd.
\newblock Box {P}roduct of {$C_p$}-{M}ackey {F}unctors.
\newblock 2017.
\newblock \href{https://arxiv.org/abs/1710.08560}{arxiv:1710.08560}.

\bibitem[Lur15]{RotationInvariance}
Jacob Lurie.
\newblock Rotation {I}nvariance in {A}lgebraic {K}-theory.
\newblock 2015.
\newblock \newline {A}vailable at
  \href{http://www.math.ias.edu/~lurie/}{http://www.math.ias.edu/~lurie/}.

\bibitem[Lur17]{HA}
Jacob Lurie.
\newblock Higher {A}lgebra.
\newblock 2017.
\newblock \newline {A}vailable at
  \href{http://www.math.ias.edu/~lurie/}{http://www.math.ias.edu/~lurie/}.

\bibitem[Lur18a]{ECII}
Jacob Lurie.
\newblock Elliptic {C}ohomology {II}: {O}rientations.
\newblock 2018.
\newblock \newline {A}vailable at
  \href{http://www.math.ias.edu/~lurie/}{http://www.math.ias.edu/~lurie/}.

\bibitem[Lur18b]{SAG}
Jacob Lurie.
\newblock Spectral {A}lgebraic {G}eometry.
\newblock 2018.
\newblock \newline {A}vailable at
  \href{http://www.math.ias.edu/~lurie/}{http://www.math.ias.edu/~lurie/}.

\bibitem[MNN17]{MNN}
Akhil Mathew, Niko Naumann, and Justin Noel.
\newblock Nilpotence and descent in equivariant stable homotopy theory.
\newblock {\em Adv. Math.}, 305:994--1084, 2017.

\bibitem[Mor05]{MorelConn}
Fabien Morel.
\newblock The stable {${\Bbb A}^1$}-connectivity theorems.
\newblock {\em $K$-Theory}, 35(1-2):1--68, 2005.

\bibitem[Mor06]{Morelpizero}
Fabien Morel.
\newblock {$\Bbb A^1$}-algebraic topology.
\newblock In {\em International {C}ongress of {M}athematicians. {V}ol. {II}},
  pages 1035--1059. Eur. Math. Soc., Z\"{u}rich, 2006.

\bibitem[MV99]{MorelVoevodsky}
Fabien Morel and Vladimir Voevodsky.
\newblock {${\bf A}^1$}-homotopy theory of schemes.
\newblock {\em Inst. Hautes \'{E}tudes Sci. Publ. Math.}, (90):45--143 (2001),
  1999.

\bibitem[Pel13]{Pelaez}
Pablo Pelaez.
\newblock On the functoriality of the slice filtration.
\newblock {\em J. K-Theory}, 11(1):55--71, 2013.

\bibitem[Pos11]{DMRecon}
Leonid Positselski.
\newblock Mixed {A}rtin-{T}ate motives with finite coefficients.
\newblock {\em Mosc. Math. J.}, 11(2):317--402, 407--408, 2011.

\bibitem[Pos14]{GalKosz}
Leonid Positselski.
\newblock Galois cohomology of a number field is {K}oszul.
\newblock {\em J. Number Theory}, 145:126--152, 2014.

\bibitem[Pst18]{Pstragowski}
Piotr Pstr\k{a}gowski.
\newblock Synthetic spectra and the cellular motivic category.
\newblock 2018.
\newblock \href{https://arxiv.org/abs/1803.01804}{arxiv:1803.01804}.

\bibitem[PSW20]{SpectrumMackey}
Irakli {Patchkoria}, Beren {Sanders}, and Christian {Wimmer}.
\newblock {The spectrum of derived Mackey functors}.
\newblock 2020.
\newblock \href{https://arxiv.org/abs/2008.02368}{arxiv:2008.02368}.

\bibitem[ROsr08]{MotivicModules}
Oliver R\"{o}ndigs and Paul~Arne \O~stv\ae r.
\newblock Modules over motivic cohomology.
\newblock {\em Adv. Math.}, 219(2):689--727, 2008.

\bibitem[Spi10]{Spitzweck}
Markus Spitzweck.
\newblock Relations between slices and quotients of the algebraic cobordism
  spectrum.
\newblock {\em Homology Homotopy Appl.}, 12(2):335--351, 2010.

\bibitem[Tho85]{Thomason}
R.~W. Thomason.
\newblock Algebraic {$K$}-theory and \'{e}tale cohomology.
\newblock {\em Ann. Sci. \'{E}cole Norm. Sup. (4)}, 18(3):437--552, 1985.

\bibitem[Voe02]{VoeOpenProbs}
Vladimir Voevodsky.
\newblock Open problems in the motivic stable homotopy theory. {I}.
\newblock In {\em Motives, polylogarithms and {H}odge theory, {P}art {I}
  ({I}rvine, {CA}, 1998)}, volume~3 of {\em Int. Press Lect. Ser.}, pages
  3--34. Int. Press, Somerville, MA, 2002.

\bibitem[Voe03a]{Voe03b}
Vladimir Voevodsky.
\newblock Motivic cohomology with {${\bf Z}/2$}-coefficients.
\newblock {\em Publ. Math. Inst. Hautes \'{E}tudes Sci.}, (98):59--104, 2003.

\bibitem[Voe03b]{Voe03}
Vladimir Voevodsky.
\newblock Reduced power operations in motivic cohomology.
\newblock {\em Publ. Math. Inst. Hautes \'{E}tudes Sci.}, (98):1--57, 2003.

\bibitem[Voe10]{Voe10}
Vladimir Voevodsky.
\newblock Motivic {E}ilenberg-{M}aclane spaces.
\newblock {\em Publ. Math. Inst. Hautes \'{E}tudes Sci.}, (112):1--99, 2010.

\bibitem[Wen10]{Wendt}
Matthias Wendt.
\newblock More examples of motivic cell structures.
\newblock 2010.
\newblock \href{https://arxiv.org/abs/1012.0454}{arxiv:1012.0454}.

\end{thebibliography}

\end{document}